\documentclass[12pt]{amsart}

\usepackage{amsmath,amsfonts,amsthm,amssymb}
\usepackage{hyperref}
\usepackage{tikz}
\usepackage{diagrams}
\usetikzlibrary{arrows.meta, positioning, calc}
\usepackage{bm}

\usepackage{xcolor}

\usepackage{lscape}

\renewcommand{\arraystretch}{1.3}

\usepackage{caption}

\usepackage{a4wide}
\usepackage{enumitem}
\usepackage{multicol, multirow, longtable}
\newtheorem{theorem}{Theorem}[section]
\newtheorem{lemma}[theorem]{Lemma}
\newtheorem{proposition}[theorem]{Proposition}
\newtheorem{corollary}[theorem]{Corollary}

\newtheorem{definition}[theorem]{Definition}

\newtheorem{theo}{Theorem}

\theoremstyle{remark}
\newtheorem{remark}[theorem]{Remark}

\newtheorem*{claim*}{Claim}

\numberwithin{equation}{section}

\newcommand{\R}{\ensuremath{\mathbb{R}}}

\newcommand{\g}[1]{\ensuremath{\mathfrak{#1}}}

\DeclareMathOperator{\id}{Id}

\DeclareMathOperator{\Ad}{Ad}
\DeclareMathOperator{\ad}{ad}

\DeclareMathOperator{\spann}{span}
\DeclareMathOperator{\diag}{diag}
\DeclareMathOperator{\Ric}{Ric}

\DeclareMathOperator{\rank}{rank}

\newcommand{\T}{\ensuremath{\mathsf{T}}}
\renewcommand{\mod}[1]{\ensuremath{\;(\mathrm{mod\;}#1)}}

\newcommand{\spin}[1]{\ensuremath{\mathsf{Spin}_{#1}}}
\newcommand{\su}[1]{\ensuremath{\mathsf{SU}_{#1}}}
\renewcommand{\u}[1]{\ensuremath{\mathsf{U}_{#1}}}
\newcommand{\so}[1]{\ensuremath{\mathsf{SO}_{#1}}}
\newcommand{\oo}[1]{\ensuremath{\mathsf{O}_{#1}}}
\renewcommand{\sp}[1]{\ensuremath{\mathsf{Sp}_{#1}}}

\newcommand{\sg}{\ensuremath{\mathsf{S}}}

\renewcommand{\gg}{\ensuremath{\mathsf{G}_2}}

\newcommand{\bq}{/\!\!/}

\makeatletter
\newsavebox{\@brx}
\newcommand{\llangle}[1][]{\savebox{\@brx}{\(\m@th{#1\langle}\)}%
	\mathopen{\copy\@brx\kern-0.5\wd\@brx\usebox{\@brx}}}
\newcommand{\rrangle}[1][]{\savebox{\@brx}{\(\m@th{#1\rangle}\)}%
	\mathclose{\copy\@brx\kern-0.5\wd\@brx\usebox{\@brx}}}
\makeatother

\usepackage{parskip}

\newcommand{\sph}{\mathbb{S}}
\newcommand{\RP}{\mathbb{RP}}
\newcommand{\CP}{\mathbb{CP}}
\newcommand{\HP}{\mathbb{HP}}

\begin{document}
\title[Positive curvature on products of spheres via fatness]{Positive $\mathrm{\bf Ric}_{\bf 2}$ curvature on products of spheres and their quotients via intermediate fatness}

\author[J.~DeVito]{Jason DeVito}
\address{The University of Tennessee at Martin, Tennessee, USA}
\email{jdevito1@ut.utm.edu}

\author[M.~Dom\'{\i}nguez-V\'{a}zquez]{Miguel Dom\'{\i}nguez-V\'{a}zquez}
\address{CITMAga, 15782 Santiago de Compostela, Spain.\newline\indent Department of Mathematics, Universidade de Santiago de Compostela, Spain}
\email{miguel.dominguez@usc.es}

\author[D.~Gonz\'alez-\'Alvaro]{David Gonz\'alez-\'Alvaro}
\address{Universidad Polit\'ecnica de Madrid, Spain.}
\email{david.gonzalez.alvaro@upm.es}

\author[A.~Rodr\'iguez-V\'azquez]{Alberto Rodr\'iguez-V\'azquez}
\address{Department of Mathematics, Université Libre de Bruxelles, Brussels, Belgium.}
\email{alberto.rodriguez.vazquez@ulb.be}

\begin{abstract}
We construct metrics of positive $2^{\rm nd}$  intermediate Ricci curvature, $\Ric_2>0$, on closed manifolds of dimensions 10, 11, 12, 13 and 14, including $\sph^6\times\sph^7$, $\sph^7\times\sph^7$ and all their simply connected isometric quotients.  In particular, we obtain infinitely many examples in dimension $13$.  We also produce infinitely many non-simply connected spaces with $\Ric_2>0$ in dimensions $13$ and $14$, including $\RP^6\times \RP^7$ and $\RP^7\times \RP^7$, which cannot admit a metric of positive sectional curvature. 
The main new idea is a generalization of the concept of fatness which ensures the existence of $\Ric_2>0$ metrics on the total space of certain homogeneous bundles.
\end{abstract}

%\subjclass[2010]{Primary: 53C20. Seconday: 53C21, 53C30, 53C35, 58D19.}

\maketitle

\section*{Introduction}

A general problem in Riemannian geometry is to study to what extent the existence of a metric satisfying a prescribed curvature condition on a manifold determines the topology of the underlying space. Among the most classical curvature conditions one finds positive sectional curvature and positive Ricci curvature, to be denoted by $\sec>0$ and $\Ric>0$, respectively. There are two important results that illustrate the relation between curvature and topology of $\sec>0$ and $\Ric>0$ manifolds, and show the gap existing between these two classes of manifolds. On the one hand, Gromov~\cite{G81} proved that the total Betti number of a closed $n$-dimensional manifold $M^n$ of $\sec>0$ is bounded by a universal constant $c(n)$ depending only on the dimension. On the other hand, Sha and Yang~\cite{ShaYang} were the first to construct infinite sequences of closed manifolds $\{M_\ell^n\}_{\ell\geq 1}$ of $\Ric>0$ in each dimension $n\geq 4$ with unbounded total Betti number, implying that infinitely many of them cannot admit metrics of $\sec>0$. 

In the last years, a family of curvature conditions $\Ric_k>0$ interpolating between $\sec>0$ and $\Ric>0$ has attracted considerable attention; see \cite{Mo} for an updated collection of works. Given $k$ with $1\leq k\leq n-1$, a Riemannian manifold $M$ of dimension $n$ is said to have positive $k^{\rm th}$ intermediate Ricci curvature, to be denoted by $\Ric_k>0$, if for any point $p\in M$ and any orthonormal set $\{x,e_1,\dots,e_k\}\subset T_p M$ the sum of sectional curvatures $\sum_{i=1}^k\sec(x,e_i)$ is positive. It follows from the definition that $\Ric_1>0$ and $\Ric_{n-1}>0$ are equivalent to $\sec>0$ and $\Ric>0$, respectively. Also, it is straightforward to check that $\Ric_k>0$ implies $\Ric_{j}>0$ for all $k\leq j\leq n-1$. Thus, the conditions $\Ric_k>0$ interpolate between $\sec>0$ and $\Ric>0$. 

One of the hopes of studying the filtration of curvature conditions $\Ric_k>0$ is that it will improve our understanding of the classes $\sec>0$ and $\Ric>0$ and the differences between them. Indeed, Reiser and Wraith have refined in \cite{RW23} the construction of Sha and Yang to show that there exist infinite sequences of closed manifolds $\{M_\ell^n\}_{\ell\geq 1}$ of $\Ric_{[\frac{n}{2}]+2}>0$ in each dimension $n\geq 5$ with unbounded total Betti number. On the other side, it remains an important open question whether there is a universal bound on the total Betti number of manifolds of $\Ric_2>0$, the strongest condition next to $\sec>0$. One of the main difficulties is that the number of existing examples of $\Ric_2>0$ is rather small. 
Moreover, it is not known whether the class of closed simply connected manifolds with $\Ric_2>0$  is strictly larger than the class of $\sec>0$ manifolds, for each dimension $\geq 4$.

The goal of this article is to improve this situation by providing new examples of $\Ric_2>0$ manifolds. Before presenting our main results, let us review the existing examples in the closed simply connected case. As for the stronger condition $\sec>0$, apart from compact rank one symmetric spaces $\sph^n$, $\CP^n$, $\HP^n$, $\mathbb O\mathbb P^2$, we only know of sporadic examples in dimensions 6, 7, 12, 13 and 24; see~\cite{Zi07} for a survey. Besides the $\sec>0$ spaces, the only additional $\Ric_2>0$ examples we know of are $\sph^3\times\sph^3$, $\sph^3\times\sph^2$ and $\sph^2\times\sph^2$, by unpublished work of Wilking; see \cite[Remark~1.3]{DGM} for details.

\begin{theo}\label{theo:short}
In each dimension $10\leq n\leq 14$ there exist closed simply connected manifolds which carry metrics of $\Ric_2>0$ and are not even rationally homotopy equivalent to any of the known examples of $\sec>0$. In dimension $13$ there exist infinitely many homeomorphism types of such manifolds.
\end{theo}

An important question in this context is whether, for a fixed $k$, there exist infinitely many dimensions supporting a manifold of $\Ric_k>0$ not topologically equivalent to a compact rank one symmetric space. See \cite[Remark~3, p.~666]{wilking:annals} for a partial negative answer in the presence of group actions.

The description of our $\Ric_2>0$ simply connected examples is given in the following result. This, along with the discussion of the rational homotopy types of the known $\sec>0$ examples in Remark~\ref{rem:homotopy_types}, implies Theorem~\ref{theo:short}.

\begin{theo}\label{THM:Main_THM}
The following manifolds admit a metric of $\Ric_2>0$. First,
\begin{enumerate}[label = \rm(\arabic*)]
\item $\sph^6\times\sph^7$, \label{item:s6xs7}
\item a circle quotient of $\sph^6\times\sph^7$ with the same integral cohomology ring of $\sph^6\times\CP^3$ but of a different homotopy type,
\item an $\su{2}$-quotient of $\sph^6\times\sph^7$ with the same integral cohomology ring of $\sph^6\times\sph^4$  but of a different homotopy type.
\end{enumerate}
Second,
\begin{enumerate}[label = \rm(\arabic*)]
	\setcounter{enumi}{3}
\item $\sph^7\times\sph^7$, \label{item:s7xs7}
\item infinitely many circle quotients of $\sph^7\times\sph^7$ all of which have the same integral cohomology ring of $\sph^7\times\CP^3$ but realize infinitely many homeomorphism types, none of which is homotopy equivalent to $\sph^7\times\CP^3$,
\item an $\su{2}$-quotient of $\sph^7\times\sph^7$ with the same integral cohomology ring of $\sph^7\times\sph^4$ but of a different homotopy type.
\end{enumerate}
\end{theo}

It is especially interesting that the manifolds $\sph^7\times\sph^7$ and $\sph^6\times\sph^7$ (together with the previously known $\sph^3\times\sph^3$, $\sph^3\times\sph^2$ and $\sph^2\times\sph^2$) carry metrics of $\Ric_2>0$, since a generalized version of the Hopf conjecture suggests that a product manifold should not admit a metric of $\sec>0$, see \cite{Zi07}. Note as well that our $14$-dimensional example has vanishing Euler characteristic, while another conjecture of Hopf suggests that the Euler characteristic of an even-dimensional $\sec>0$ manifold should be positive.  In particular, these conjectures, if true, would imply that, from the perspective of the conditions $\Ric_k>0$, our metrics are the best possible.

Now we move to examples with non-trivial fundamental group. As observed in \cite{DGM} or \cite[p.~5]{KM24}, Wilking's metric on $\sph^n\times\sph^m$, with $n,m\in\{2,3\}$, descends to $\Ric_2>0$ metrics on some finite quotients like $\RP^n\times \RP^m$ and products of lens spaces $L_p^3\times L_q^3$, which cannot admit metrics of $\sec >0$. This shows that the class of $\Ric_2>0$ manifolds is strictly larger than the class of $\sec>0$ manifolds in dimensions $4,5,6$, with infinitely many examples in dimension $6$ of pairwise distinct fundamental group. By considering finite quotients of our metrics on $\sph^6\times\sph^7$ and $\sph^7\times\sph^7$ we produce $\Ric_2>0$ manifolds, including $\RP^6\times \RP^7$ and $\RP^7\times \RP^7$, showing that the same statement holds in dimensions $13$ and $14$. We also observe that Wilking's metric on $\sph^2\times\sph^3$ descends to infinitely many quotients.

\begin{theo}\label{THM:non_simply_connected}
There exist infinitely many closed manifolds of pairwise distinct fundamental group in dimensions $5$, $6$, $13$ and $14$ which admit metrics of $\Ric_2>0$ but cannot admit metrics of $\sec >0$. 
\end{theo}

Some of our examples enjoy very special geometric properties. As illustrating cases, our $\Ric_2>0$ metrics on $\sph^k\times\sph^7$ and $\RP^k\times \RP^7$, with $k\in\{6,7\}$, are homogeneous. Also, the examples $\sph^7\times\sph^7$ and $\sph^6\times\sph^7$ admit Riemannian submersions over a round sphere $\sph^4$. Moreover, $\sph^7\times\sph^7$ and $\sph^6\times\sph^7$ contain $\sph^3\times\sph^3$ and $\sph^2\times\sph^3$ as totally geodesic submanifolds, respectively, thus relating our examples to those previously constructed by Wilking. The construction of the metrics and their properties are explained in detail in the following section.

\section{Summary of results}

In this section we state the main results of the article, which will imply Theorems~\ref{THM:Main_THM}~and~\ref{THM:non_simply_connected}.

\subsection{Homogeneous examples}

All the examples in this article arise as isometric quotients of the following two spaces.

\begin{theo}\label{THM:homogeneous_Ric2}
The homogeneous spaces $\spin{7}/\su{3}$ and $\spin{8}/\gg$ carry homogeneous metrics of $\Ric_2>0$.
\end{theo}
It is well known that $\spin{7}/\su{3}$ and $\spin{8}/\gg$ are diffeomorphic to $\sph^6\times\sph^7$ and $\sph^7\times\sph^7$, respectively, which proves Items~\ref{item:s6xs7} and~\ref{item:s7xs7} in Theorem~\ref{THM:Main_THM}. In Subsection~\ref{SS:metric_construction} we explain the construction behind Theorem~\ref{THM:homogeneous_Ric2}.

The spaces in Theorem~\ref{THM:homogeneous_Ric2} together with $\sph^2\times\sph^3$ and $\sph^3\times\sph^3$ are the only simply connected ones which are known to carry homogeneous metrics of $\Ric_2>0$, apart from those which admit metrics of $\sec>0$ (see \cite{WZ:crelle} for the complete classification). The classification of homogeneous spaces of $\Ric_2 > 0$ is open, even in the simply connected case. We warn the reader that one cannot expect the existence of new examples in infinitely many dimensions. Indeed, it follows from a result of Wilking \cite[Remark~3, p.~666]{wilking:annals} that any homogeneous space of $\Ric_2>0$ and dimension greater than $3600$ is homotopy equivalent to a compact rank one symmetric space, which in turn implies that it is diffeomorphic to a compact rank one symmetric space \cite[p.~195]{Be:geodesics} and hence carries a homogeneous metric of $\sec >0$. 

Let us compare our examples with various structural theorems for manifolds with symmetries and lower curvature bounds. Recall that the symmetry rank of a Riemannian manifold is defined as the rank of its isometry group. In this article we shall show that our metrics of $\Ric_2>0$ on $\sph^6\times\sph^7$ and $\sph^7\times\sph^7$ have rank $3$ and $4$, respectively, see Theorems~\ref{thm:isometry_group} and \ref{thm:isometriesqt}.

First, it has recently been shown in the papers \cite{KWW21,Ni22} that a compact even-dimensional manifold of $\sec>0$ and symmetry rank $\geq 4$ must have positive Euler characteristic. Our metric on $\sph^7\times\sph^7$ implies that their result is sharp from the perspective of the $\Ric_k>0$ conditions, since the $\sec>0$ assumption cannot be weakened to $\Ric_2>0$.

Second, it follows from \cite[Theorem~A]{Mo22} and \cite[Theorem~A]{KM24} that a manifold of dimension $13$ (resp.\ $14$) of $\Ric_2>0$ and symmetry rank $7$ (resp.\ $7$) is diffeomorphic to $\sph^{13}$ (resp.\ diffeomorphic to $\sph^{14}$ or homeomorphic to $\CP^7$). By our results for $\sph^6\times\sph^7$ (resp. $\sph^7\times\sph^7$), the condition on the symmetry rank cannot be relaxed to be $3$ (resp. $4$). Let us remark that in dimension 5, \cite[Theorem~A]{Mo22} shows that a $\Ric_2>0$ manifold with symmetry rank equal to $3$ must be diffeomorphic to $\sph^5$, so in this case the rank assumption is optimal since $\sph^2\times\sph^3$ carries a homogeneous metric of rank 2, see Proposition~\ref{prop:quotients_S2S3}.

\subsection{Submersion metrics and intermediate fatness}\label{SS:metric_construction}

The construction behind Theorem~\ref{THM:homogeneous_Ric2} constitutes the main geometrical contribution of this work. 

Recall that a nested triple $H<K<G$ of compact Lie groups induces a so-called homogeneous bundle $K/H\to G/H\to G/K$. There is a natural way to construct metrics on $G/H$ by starting with a normal homogeneous metric and deforming it in the direction of the fibers. We call them submersion metrics. One advantage of submersion metrics is that they have no more zero-curvature planes than normal metrics. Moreover, one can compute sectional curvatures in terms of the curvature of base $G/K$ and fiber $K/H$ and the geometry of the bundle.

In this work we derive sufficient conditions on the triple $H<K<G$ for the total space $G/H$ to admit a submersion metric of $\Ric_2>0$, and more generally of $\Ric_k>0$ for the appropriate $k$. In order to motivate our construction, let us briefly discuss a classical result of Wallach. First, we recall the following notation from \cite[Section~1.1]{DGM} which measures the curvature of normal homogeneous metrics on a compact homogeneous space $G/H$ of finite fundamental group:
$$
b(G/H)=\min \{k\in\mathbb N : \Ric_k>0 \text{ for a normal homogeneous metric on } G/H\}.
$$
We set $b(G/H)=\dim G/H$ for spaces of infinite fundamental group. In this terminology, spaces $G/H$ with $b(G/H)=1$ correspond to normal homogeneous spaces of $\sec >0$ together with $G/H=\sg^1$; see Section~\ref{SEC:b_and_f} for more information. The construction of Wallach can be stated as follows.

\begin{theorem}[Wallach]\label{thm:wallach}
Let $H<K<G$ be a triple of compact Lie groups. Suppose that the following four conditions are satisfied:
\begin{enumerate}[label = \rm(\arabic*)]
\item $G/K$ is a symmetric space,\label{item:base_sym}
\item $b(G/K)=1$,\label{item:base_b_1}
\item $b(K/H)=1$, and\label{item:fiber_b_1}
\item the bundle is fat.\label{item:bundle_fat}
\end{enumerate}
Then the total space $G/H$ admits a homogeneous metric of $\sec>0$.
\end{theorem}

Note that all the assumptions are fully understood: symmetric spaces were classified by Cartan almost a century ago, normal homogeneous spaces of $\sec >0$ were classified by Berger \cite{Berger:normal} (with an omission fixed by Wilking \cite{wilking:pams}), and fat homogeneous bundles were classified by B\'erard-Bergery in \cite{B75}. The notion of fatness was introduced by Weinstein in the context of general fiber bundles in order to ensure the positivity of the sectional curvature of vertizontal planes \cite{We80}, see also \cite{Zi99}. We recall the characterization of fatness in the case of homogeneous bundles in the paragraph following Definition~\ref{DEF:intermediate_fatness} below.

The idea now is to relax the assumptions in Wallach's construction and determine which curvature condition $\Ric_k>0$ one gets on the total space. The following generalization was obtained in \cite[Theorem~F]{DGM}, which removes assumption~\ref{item:base_b_1}, and shows that the curvature on the total space is determined by that of the base space.

\begin{theorem}[Domínguez-Vázquez, González-Álvaro, Mouillé]\label{thm:dgm}
Let $H<K<G$ be a triple of compact Lie groups. Suppose that the following three conditions are satisfied:
\begin{enumerate}[label = \rm(\arabic*)]
\item $G/K$ is a symmetric space,
\item[$(3)$] $b(K/H)=1$, and
\item[$(4)$] the bundle is fat.
\end{enumerate}
Then the total space $G/H$ admits a homogeneous metric of $\Ric_k>0$ for $k=b(G/K)$.
\end{theorem}

The determination of $b(G/K)$ for all symmetric spaces $G/K$ is known, see e.g.\ \cite{AQZ,DGM}. It is remarkable that $b$ of a symmetric space is never $2$, so there are no examples of $\Ric_2>0$ among symmetric spaces and Theorem~\ref{thm:dgm} does not yield any examples of $\Ric_2>0$, except for those which were already known to carry metrics of $\sec>0$. 

In order to get examples of $\Ric_2>0$, in this work we shall relax appropriately the remaining assumptions \ref{item:base_sym}, \ref{item:fiber_b_1} and \ref{item:bundle_fat} from Theorem~\ref{thm:wallach}. We caution the reader that assumption \ref{item:bundle_fat} has strong implications on the other assumptions. More precisely, B\'erard-Bergery showed that fatness implies that $b(K/H)=1$ \cite[Lemme~6]{B75}. Moreover, he showed that, if $\dim K/H>1$, then fatness implies that $G/K$ is a symmetric pair \cite[Lemme~16]{B75}.

Hence, we propose to generalize the fatness condition. The main new ingredient in this work is the following notion, which somehow measures the fatness of a homogeneous bundle.

\begin{definition}\label{DEF:intermediate_fatness}
Let $H<K<G$ be a triple of compact Lie groups with associated Lie algebras $\g{h}\subset\g{k}\subset\g{g}$. Endow $G$ with a bi-invariant metric and set $\mathfrak{m}:=\mathfrak{h}^\perp \cap\mathfrak{k}$ and $\mathfrak{p}:= \mathfrak{k}^\perp$. We define the fatness coindex $f$ of the triple as the integer:$$
f(H<K<G)=\max\left\lbrace \max_{x\in\g{p}\setminus\{0\}}\dim Z_{\g{m}}(x) , \max_{y\in\g{m}\setminus\{0\}}\dim Z_{\g{p}}(y)\right\rbrace,
$$
where $Z_{\g{m}}(x)$ (resp.\ $Z_{\g{p}}(y)$) denotes the centralizer of $x$ in $\g{m}$ (resp.\ of $y$ in $\g{p}$).
\end{definition}

We shall show that this definition does not depend on the bi-invariant metric chosen on~$G$, see Lemma~\ref{lem:independence}. Observe that the case $f=0$ recovers precisely the classical definition of fatness stating that $[x,y]\neq 0$ for all non-zero $x\in\g{p}$ and $y\in\g{m}$.

As anticipated above, we shall also relax assumption \ref{item:base_sym}. We consider a technical property of homogeneous spaces $G/K$, which we call Property (P), that allows us to use a theorem of Eschenburg characterizing flat planes of certain homogeneous metrics. The precise condition is given in Definition \ref{def:p} (see also Definition \ref{def:pprime}).  Property (P) is satisfied by all symmetric spaces and by other homogeneous spaces like $\spin{7}/\gg$, which is the most relevant one to this work. 

We are now ready to state the main metric construction which generalizes Theorems~\ref{thm:wallach} and \ref{thm:dgm}. See Theorem~\ref{THM:Wallach_generalization} for a more detailed statement.

\begin{theo}\label{THM:metric_construction}
Let $H < K < G$ be a triple such that $G/K$ satisfies Property (P). Define
	\[
	k:=b(K/H)+b(G/K)+f(H<K<G)-1.
	\]
If $k<\dim G/H$, then the total space $G/H$ admits a homogeneous metric of $\Ric_k>0$.	
\end{theo}

As a corollary of Theorem~\ref{THM:metric_construction}, any triple with $b(G/K)=b(K/H)=f(H<K<G)=1$ and $G/K$ satisfying Property~(P) yields a metric of $\Ric_2>0$ on $G/H$. In the following Theorem~\ref{THM:classification_triples}, we classify the associated Lie algebra triples of all such Lie group triples up to triple-isomorphism, see Definition~\ref{DEF:pair_triple_isomorphic}. We refer to Theorem~\ref{THM:submersion_triples} and Section~\ref{SEC:classification_triples} for an alternative equivalent statement and its proof.

\begin{theo}\label{THM:classification_triples}
Let $H<K<G$ be a triple with 
$$b(G/K)=b(K/H)=f(H<K<G)=1$$
and with $G/K$ satisfying Property~(P). If the Lie algebras in the associated triple $\g{h}\subset\g{k}\subset\g{g}$ do not share a common non-trivial ideal, then $\g{h}\subset\g{k}\subset\g{g}$ is triple-isomorphic to one of the following:
\begin{enumerate}[label = \rm(\arabic*)]
\item $\g{so}_{2}\subset\g{so}_{3}\subset\g{so}_{4}$,
\item $\{0\}\subset\g{so}_{3}\subset\g{so}_{4}$,
\item $\g{u}_{1} \subset \g{u}_2 \cong \g{u}_1\oplus\g{so}_3\subset\g{u}_1\oplus\g{so}_4$, where the Lie subalgebra $\g{u}_{1} \subset \g{u}_2$ has non-trivial projection into the $\g{so}_3$-factor of $\g{u}_2$, \label{item:thm_classi_irrational}
\item $\g{so}_3 \subset \g{so}_4\cong\g{so}_3\oplus\g{so}_3\subset\g{so}_3\oplus\g{so}_4$,
\item $\g{su}_{3} \subset \g{g}_2 \subset \g{so}_{7}$, or
\item $\g{g}_2\subset\g{so}_7\subset\g{so}_{8}$.
\end{enumerate}
\end{theo}

Note that all Lie algebra triples $\g{h}\subset\g{k}\subset\g{g}$ from Theorem~\ref{THM:classification_triples} can be realized by a nested triple $H<K<G$ of compact connected Lie groups (except for Item~\ref{item:thm_classi_irrational}, where one must add the additional hypothesis that the $\g{u}_1$ exponentiates to a circle instead of a dense winding). Moreover, if $G$  is taken to be simply connected, then $G/H$ is simply connected as well. For the first four possibilities, the space $G/H$ is diffeomorphic to $\sph^2\times \sph^3$ or $\sph^3\times \sph^3$.  The last two cases of Theorem~\ref{THM:classification_triples}, in combination with Theorem~\ref{THM:metric_construction}, give the proof of Theorem~\ref{THM:homogeneous_Ric2}.

We remark that for the homogeneous space $\spin{8}/\gg$ we find out that normal homogeneous metrics also have $\Ric_2>0$; in other words $b(\spin{8}/\gg)=2$, see Lemma~\ref{LEM:G2_SO7_SO8}. In the next subsection we will discuss isometric quotients of $\spin{8}/\gg$ to obtain more $\Ric_2>0$ examples. For this purpose, it is irrelevant whether we consider normal metrics or submersion metrics on $\spin{8}/\gg$, since the connected component of their isometry groups is the same, see Theorem~\ref{thm:isometriesqt}. In contrast, normal metrics on $\spin{7}/\su{3}$ do not have $\Ric_2>0$, see Lemma~\ref{lem:su3_g2_so7}.

\subsection{Isometry group, free actions, and topology}\label{subsec:summary_free_actions}

In order to construct more examples of $\Ric_2>0$ manifolds, we shall take isometric quotients of the spaces in Theorem~\ref{THM:homogeneous_Ric2}. Recall that if $M$ is a manifold of $\Ric_2>0$ and $L$ is a closed group of isometries acting freely on $M$, then the quotient $M/L$ inherits a metric which is of $\Ric_2>0$ if $\dim M/L\geq 3$, by the Gray-O'Neill formula.

A great portion of this article is dedicated to classifying all compact connected groups acting freely and isometrically on $\spin{7}/\su{3}$ and $\spin{8}/\gg$ endowed with our $\Ric_2>0$ metrics, and to study the topology of the corresponding quotients. This is done in three steps.

First, we determine the identity component of the isometry group of our metrics of $\Ric_2>0$ on $\spin{7}/\su{3}$ and $\spin{8}/\gg$, see Theorems~\ref{thm:isometry_group} and \ref{thm:isometriesqt}. We shall show that the action of the identity component of the isometry groups of $\spin{7}/\su{3}$ and $\spin{8}/\gg$ is equivalent to the natural action of $\spin{7}$ and $\spin{8}$, respectively. For this we use the extension theory of Onishchik. Knowing the identity component of the isometry group, we reduce the problem of determining which connected Lie groups act freely and isometrically to classifying which connected Lie subgroups $L$ of $\spin{7}$ and $\spin{8}$ act freely on the corresponding space. The quotients $L\backslash \spin{7}/\su{3}$ and $L\backslash \spin{8}/\gg$ are typically called biquotients.

Second, we classify all biquotients of the form $L\backslash \spin{7}/\su{3}$ and $L\backslash \spin{8}/\gg$. To do that, the first observation is that the rank of $L$ must be $1$. Then, in order to study the actions of subgroups $L$, we describe the $\spin{7}$ and $\spin{8}$-actions on the homogeneous space as linear actions on the corresponding product of spheres. This allows us to find a maximal torus in $\spin{7}$ and $\spin{8}$ and hence to compute which circle subgroups act freely. Then, we study which circle actions extend to an action by $\so{3}$ or $\su{2}$.

Third, we compute the cohomology ring and the Pontryagin classes of the quotients. The cohomology calculation uses classical tools like the Gysin sequence and Poincaré duality. For Pontryagin classes, we use the well-known diffeomorphism $\sph^7\times \sph^7\cong (\u{4}\times \u{4})/(\u{3}\times \u{3})$ and write each $\sg^1$ or $\su{2}$-quotient as a biquotient of $\u{4}\times \u{4}$.  As all the relevant Lie groups have torsion-free cohomology, we can use machinery developed by Eschenburg \cite{Es3} and Singhoff \cite{Si1} to compute the characteristic classes. Unfortunately, for $\sph^6\times \sph^7$ there is no homogeneous description of $\sph^6$ using Lie groups with torsion-free cohomology, so the standard machinery does not work without modification. Instead, we identify each quotient of $\sph^6\times \sph^7$ as a codimension one submanifold of a quotient of $\sph^7\times \sph^7$, hence allowing us to relate their tangent bundles and characteristic classes. The Pontryagin classes are very relevant in that $p_1$ is a homeomorphism invariant and  $p_1\mod{24}$ is a homotopy invariant.

We can summarize  the results in the following two theorems. They follow from the classification of free circle and $\su{2}$-actions contained in Theorems~\ref{thm:free} and~\ref{thm:summary}, along with the rank restriction in Lemma~\ref{lem:rank_restriction}, and from the topological results stated in Section~\ref{SEC:topology}.

\begin{theo}\label{th:biquotientSpin8G2} Suppose $L$ is a connected Lie group acting freely and isometrically on \linebreak $\spin{8}/\gg\cong  \sph^7 \times \sph^7$. Then, $L$ is isomorphic to $\sg^1$ or to $\su{2}$.
	In addition:
	\begin{itemize}
		\item When $L \cong \sg^1$, there are infinitely many free actions. The  corresponding $\sg^1$-quotients of $\sph^7\times \sph^7$ all have the cohomology ring of $\sph^7\times \CP^3$ but a different homotopy type. The first Pontryagin classes vary among infinitely many possibilities. In particular, there are infinitely many distinct quotients up to homeomorphism, so the corresponding actions are pairwise inequivalent.
		
		\item When $L \cong \su{2}$, there is, up to equivalence, a unique such action. The quotient has the cohomology ring of $\sph^7\times \sph^4$ but a different homotopy type.
	\end{itemize}
\end{theo}

\begin{theo} \label{th:biquotientSpin7SU3}
	Suppose $L$ is a connected Lie group acting freely and isometrically on $\spin{7}/\su{3}\cong \sph^6\times \sph^7$. Then, $L$ is isomorphic to $\sg^1$ or to $\su{2}$. In addition:
	\begin{itemize}
		\item When $L \cong \sg^1$, there is, up to equivalence, a unique such action.  The quotient has the cohomology ring of $\sph^6\times \CP^3$ but a different homotopy type.
		
		\item When $L \cong \su{2}$, there is, up to equivalence, a unique such action.  The quotient has the cohomology ring of $\sph^6\times \sph^4$ but a different homotopy type.
	\end{itemize}
\end{theo}

Theorems~\ref{th:biquotientSpin8G2} and \ref{th:biquotientSpin7SU3}, in combination with the Gray-O'Neill formula for Riemannian submersions, prove the remaining items in Theorem~\ref{THM:Main_THM}.

Let us mention that the products $\sph^7\times \sph^7$ and $\sph^6\times \sph^7$ and some of their quotients from Theorems~\ref{th:biquotientSpin8G2} and \ref{th:biquotientSpin7SU3} admit Riemannian submersions to a round $\sph^4$, as we show in Proposition~\ref{prop:S4}. This should be compared to the Petersen-Wilhelm conjecture, which  suggests that the dimension of the fiber in a Riemannian submersion from a compact $\sec >0$ manifold will be less than the dimension of the base. In \cite[Section~1.3]{DGM}, it was observed that the conjecture does not hold if the assumption $\sec>0$ is relaxed to $\Ric_2>0$, due to the example $\sph^3 \times \sph^3\to\sph^2$ given by Wilking's metric. In the case of the submersion $\sph^7 \times \sph^7\to\sph^4$, the fiber dimension is even bigger in terms of the base dimension.

Finally we highlight that the previously known simply connected manifolds that admit metrics of $\Ric_2>0$ but not of $\sec>0$, namely $\sph^3\times\sph^3$, $\sph^2\times \sph^3$ and $\sph^2\times\sph^2$, can be recovered from some of our higher dimensional examples as totally geodesic submanifolds. Specifically, we will show in Proposition~\ref{prop:totallygeo} that $\sph^7\times\sph^7$ and $\sph^6\times\sph^7$ carry isometric circle actions whose fixed point sets are diffeomorphic to $\sph^3\times\sph^3$ and $\sph^2\times\sph^3$, respectively. Thus, these submanifolds are totally geodesic and hence inherit a $\Ric_2>0$ metric. Similarly, the $\sg^1$-quotient of $\sph^6\times\sph^7$ from Theorem~\ref{th:biquotientSpin7SU3} has a totally geodesic submanifold diffeomorphic to $\sph^2\times\sph^2$ and with an induced $\Ric_2>0$ metric, see Remark~\ref{rem:tg_quotient}.

\subsection{Non-simply connected quotients}\label{SS:non_simply}

Here we consider free isometric actions on $\sph^7\times \sph^7$ and $\sph^6\times \sph^7$ by finite groups, instead of connected Lie groups. The quotients inherit metrics of $\Ric_2>0$. Moreover, for some of them we can use Synge's Theorem to conclude that they cannot admit $\sec>0$ metrics. Theorems~\ref{thm:rp7} and~\ref{thm:rp6} in Section~\ref{SEC:non_simply} imply the following result.

\begin{theo}\label{THM:class_non_simply_connected_quotients}
 For each integer $d\geq 1$ both $\sph^6\times \sph^7$ and $\sph^7\times \sph^7$ admit isometric quotients of $\Ric_2>0$ with fundamental group $\mathbb{Z}_2\oplus \mathbb{Z}_d$ which cannot admit a Riemannian metric of $\sec > 0$.  When $d = 1,2$, one can take these isometric quotients to be $\RP^6\times \RP^7$, $\RP^6\times \sph^7$, $\RP^7\times \RP^7$ and $\RP^7\times \sph^7$.
\end{theo}

We will also discuss finite isometric quotients of Wilking's metric on  $\sph^n\times \sph^m$, with $n,m\in\{2,3\}$. We provide infinitely many of them in dimensions 5 and 6 which cannot admit $\sec>0$ metrics, see Proposition~\ref{prop:quotients_S2S3}. Together with Theorem~\ref{THM:class_non_simply_connected_quotients}, this proves Theorem~\ref{THM:non_simply_connected}.

Moreover, our metrics on $\RP^6\times \RP^7$ and $\RP^7\times \RP^7$ inherit transitive isometric actions by $\spin{7}$ and $\spin{8}$, respectively, see Remark~\ref{rem:nonsimply}. Thus, their symmetry rank is equal to $3$ and $4$, respectively.  Also, as we observe in Proposition~\ref{prop:quotients_S2S3}, $\RP^2\times \RP^3$ admits a homogeneous $\Ric_2>0$ metric that is invariant under $\su{2}\times\su{2}$. This should be compared to \cite[Theorem~D]{KM24}, where Kennard and Mouillé show that a $\Ric_2>0$ manifold of dimension $5$, $13$ or $14$ with non-trivial fundamental group and symmetry rank equal to $3$, $7$ or $7$, respectively, must be homotopy equivalent to a spherical space form. We conclude that the rank assumption in \cite[Theorem~D]{KM24} is optimal in dimension 5.

\subsection*{Organization of the paper}

Section~\ref{SEC:b_and_f} contains the definitions of the integers $b$ and $f$ and discusses some of their properties. Section~\ref{SEC:submersion_metrics} deals with submersion metrics and contains the proof of Theorem~\ref{THM:metric_construction}. Section~\ref{sec:sometriples} computes the $b$ and $f$ of some pairs and triples of Lie algebras. Section~\ref{SEC:classification_triples} contains the proof of Theorem~\ref{THM:classification_triples}. Section~\ref{SEC:isometry_group} is dedicated to study the isometry group of our $\Ric_2>0$ metrics on the homogeneous spaces from Theorem~\ref{THM:homogeneous_Ric2}. All free biquotient actions on these homogeneous spaces by connected groups are classified in Section~\ref{SEC:free_actions}, and Section~\ref{SEC:topology} computes the cohomology rings and Pontryagin classes of the corresponding quotients (Theorems~\ref{th:biquotientSpin8G2} and~\ref{th:biquotientSpin7SU3}). Finally, we construct non-simply connected examples in Section~\ref{SEC:non_simply}, in particular proving Theorem~\ref{THM:class_non_simply_connected_quotients}.

\subsection*{Acknowledgements} 

We are grateful to Christoph B\"ohm, Lee Kennard and Philipp Reiser for useful comments on a preliminary version of this article.

J.\ DeV.\ was supported by the National Science Foundation via the grants DMS 2105556 and DMS 2405266.

M.\ D.-V.\ was supported by the grant PID2022-138988NB-I00 funded by MICIU/AEI/ 10.13039/501100011033 and by ERDF, EU, and project ED431C 2023/31 (Xunta de Galicia, Spain), and acknowledges the Mathematisches Forschungsinstitut Oberwolfach (MFO) for providing excellent research conditions during a stay as part of the Oberwolfach Research Fellows programme.

D.\ G.-\'A.\ was supported by grants PID2021-124195NB-C31 and PID2021-124195NB-C32 from the Agencia Estatal de Investigación and the Ministerio de Ciencia e Innovación (Spain).

A.\ R.-V.\ was supported by the grant PID2022-138988NB-I00 funded by MICIU/AEI/ 10.13039/501100011033 and by ERDF, EU, by the FWO postdoctoral grant with project number 1262324N and by the Horizon Europe research and innovation programme under Marie Sklodowska Curie Actions with grant agreement 101149711 - HOLYFLOW, and also acknowledges the Mathematisches Forschungsinstitut Oberwolfach (MFO) for excellent research conditions during a stay as part of the Oberwolfach Research Fellows programme.

\section{Centralizers in pairs and triples}\label{SEC:b_and_f}

This section contains the relevant notations and definitions for nested pairs and triples of Lie algebras in relation to the dimensions of certain centralizers, together with various related results. The main new tool is the notion of fatness coindex $f$ that we define for nested triples of compact Lie algebras.

Given a Lie algebra $\g{g}$, a vector subspace $\g{a}\subset\g{g}$, and a vector $x\in\g{g}$ (not necessarily in $\g{a}$) we shall denote the centralizer of $x$ in $\g{a}$ by $Z_\g{a}(x)$. More precisely, 
$$Z_\g{a}(x)=\{y\in\g{a}:[x,y]=0\}.$$
Note that $x$ does not have to belong to $Z_\g{a}(x)$; indeed, $x\in Z_\g{a}(x)$ if and only if $x\in\g{a}$.

Let $G$ be a compact Lie group with Lie algebra $\g{g}$. For the rest of the section we fix a background bi-invariant metric on $G$ (or, equivalently, an $\Ad(G)$-invariant inner product on $\g{g}$). This induces an $\ad(\g{g})$-invariant inner product $\langle\cdot,\cdot\rangle$ on the compact Lie algebra $\g{g}$. Thus, we will frequently make the abuse of terminology of calling such an $\ad(\g{g})$-invariant inner product on $\g{g}$ a bi-invariant metric on $\g{g}$. Given a subalgebra $\g{h}\subset \g{g}$, we denote by $\g{h}^\bot$ the orthogonal complement with respect to this inner product. Also, we will consider $\g{h}$ equipped with the bi-invariant metric obtained by restriction of the bi-invariant metric on $\g{g}$.

We start with a definition for a Lie algebra pair which measures the intermediate Ricci curvature of the corresponding normal homogeneous space.

\begin{definition}
Given a Lie algebra $\g{g}$ and a proper subalgebra $\g{h}\subset\g{g}$, we define $b(\g{h}\subset\g{g})$~as:
\begin{equation}\label{eq:characterization_b}
b(\g{h}\subset\g{g})=\max_{x\in{\g{h}^\perp}\setminus\{0\}}\dim Z_{\g{h}^\perp}(x).
\end{equation}
\end{definition}

To justify the geometric interpretation of $b(\g{h}\subset\g{g})$, let $G/H$ be a compact homogeneous space of dimension at least $2$ with finite fundamental group.  In \cite[Section~1.1]{DGM}, the invariant $b(G/H)$ is defined as the minimum integer~$k$ for which a normal homogeneous metric on $G/H$ satisfies $\Ric_k>0$. Since any normal homogeneous metric on $G/H$ has $\Ric>0$, the integer $b(G/H)$ is well defined and satisfies the inequality $1\leq b(G/H)\leq\dim G/H -1$. The two equality cases are fully understood. Spaces with $b(G/H)=1$ correspond to those admitting a normal metric of positive sectional curvature, see Table \ref{TABLE:b1} in Section~\ref{SEC:classification_triples} for more information. Spaces with $b(G/H)=\dim G/H-1$ have been recently characterized in \cite[Theorem~C]{DGM} as those which locally split off an $\mathbb{S}^2$-factor.

The tangent space $G/H$ at $eH$ is canonically identified with $\g{h}^\perp\subset\g{g}$, where $\g{h}$ and $\g{g}$ denote the Lie algebras of $H$ and $G$ and $e$ denotes the identity of $G$. The standard formulas for the curvature of a normal homogeneous metric on $G/H$ imply that the sectional curvature of any plane spanned by $x,y\in\g{h}^\perp$ is non-negative, and it is positive if and only if $[x,y]\neq 0$. Using this, it is shown in \cite[Proposition~3.1]{DGM} that $b(G/H)$ equals $b(\g{h}\subset\g{g})$ as defined above.  This algebraic characterization is used in \cite[Corollary~3.3]{DGM} to show that $b(G/H)$, and hence $b(\g{h}\subset\g{g})$, does not depend on the bi-invariant metric on $G$.

If instead $G/H$ has infinite fundamental group, then by Bonnet-Myers Theorem $G/H$ cannot admit a metric of $\Ric>0$. In this case one can define $b(G/H)$ to be $\dim G/H$, and again $b(G/H)$ clearly agrees with $b(\g{h}\subset\g{g})$.

Now we introduce a definition which is the main new ingredient in this work. Let $\g{h}\subset\g{k}\subset\g{g}$ be a triple of compact Lie algebras. We use the bi-invariant metric on $G$ to define:
\[
\mathfrak{m}:=\mathfrak{h}^\perp \cap\mathfrak{k}, \qquad \mathfrak{p}:= \mathfrak{k}^\perp.
\]
We will use this notation throughout the paper.

\begin{definition}\label{DEF:F_triple}
Given a nested triple of Lie algebras $\g{h}\subsetneq\g{k}\subsetneq\g{g}$, we define the fatness coindex $f(\g{h}\subset\g{k}\subset\g{g})$ of the triple as:
$$
f(\g{h}\subset\g{k}\subset\g{g})=\max\{ \max_{x\in\g{p}\setminus\{0\}}\dim Z_{\g{m}}(x) , \max_{y\in\g{m}\setminus\{0\}}\dim Z_{\g{p}}(y)\}.
$$
We may  simply write $f$ when there is no risk of confusion.
\end{definition}

Let us outline the geometric interpretation of Definition~\ref{DEF:F_triple}. If $H<K<G$ is a nested triple of compact Lie groups with Lie algebras $\g{h}\subset\g{k}\subset\g{g}$, then there is an associated homogeneous bundle $K/H\to G/H\to G/K$. We call $\g{h}\subset\g{g}$, $\g{k}\subset\g{g}$ and $\g{h}\subset\g{k}$ the total space pair, the base pair, and the fiber pair, respectively. Their tangent spaces can be identified as follows:
		\[
			T_{eH} K/H\simeq \mathfrak{m},\qquad T_{eK} G/K\simeq\mathfrak{p}  ,\qquad T_{eH} G/H\simeq\mathfrak{h}^\perp=\mathfrak{m}\oplus\mathfrak{p}.
		\]
The invariant $f$ ensures the positivity of the intermediate curvatures $\sum_{i=1}^{f+1}\sec(u,v_i)$ with respect to a normal metric on $G/H$, where either $u\in\g{p}$ and $v_i\in\g{m}$ or $u\in\g{m}$ and $v_i\in\g{p}$, and where $\{v_i\}_{i=1}^{f+1}$ are linearly independent. As we explain in Section~\ref{SEC:submersion_metrics}, the fatness coindex $f$ is crucial in determining the intermediate Ricci curvature of certain non-normal homogeneous metrics on $G/H$.		
In Lemma~\ref{lem:independence} below we prove that, as in the case of $b$ for pairs, $f$ does not depend on the bi-invariant metric chosen.

Note that it follows from the definition that $0\leq f\leq \max\{\dim\g{p},\dim\g{m}\}$. When $f=0$, the triple is called fat in the literature. Fat triples were classified by B\'erard-Bergery in \cite{B75}. On the other side, triples with $f=\max\{\dim\g{p},\dim\g{m}\}$ include those where $\g{m}$ and $\g{p}$ are contained in (not necessarily simple) ideals $\g{g}_1$ and $\g{g}_2$ of $\g{g}$, respectively, with $\g{g}_1\cap\g{g}_2=\{0\}$. The following lemma is a direct consequence of the previous definitions.

\begin{lemma}\label{LEM:f_and_b}
For any triple $\g{h}\subset\g{k}\subset\g{g}$ the following hold:
\begin{enumerate}[label = \rm(\arabic*)]
\item $b(\g{h}\subset\g{g})\geq f(\g{h}\subset\g{k}\subset\g{g}) +1$,
\item $b(\g{h}\subset\g{g})\geq \max\{b(\g{k}\subset \g{g}),b(\g{h}\subset \g{k})\}$.
\end{enumerate}
\end{lemma}

\begin{lemma}\label{lem:independence}
The fatness coindex $f$ is independent of the chosen bi-invariant metric on~$G$.
\end{lemma}
\begin{proof}
	Let us consider $\g{g}=\bigoplus_{i=0}^k\g{g}_i$, where $\g{g}_0$ is the abelian factor, and $\g{g}_i$ is a simple ideal for each $i\in\{1,\ldots,k\}$.   Given an $\Ad(G)$-invariant inner product $\llangle\cdot,\cdot\rrangle$ on $\g{g}$,  Schur's lemma implies that any other $\Ad(G)$-invariant inner product $\langle\cdot,\cdot\rangle$ on $\g{g}$ satisfies $\langle\cdot,\cdot\rangle\rvert_{\g{g}_i\times\g{g}_i}=s_i \llangle\cdot,\cdot\rrangle\rvert_{\g{g}_i\times\g{g}_i}$, for some $s_i>0$ and each $i\geq1$, and $\langle\cdot,\cdot\rangle\rvert_{\g{g}_0\times\g{g}_0}=\llangle A \,\cdot,\cdot\rrangle$, for some linear automorphism $A$ of $\g{g}_0$. Because each $\g{g}_i$ for $i\geq 1$ is a simple Lie algebra, it follows that  $[\g{g}_i,\g{g}_i]=\g{g}_i$. This property, in conjunction with the $\Ad(G)$-invariance, leads to the conclusion that $\langle\cdot,\cdot\rangle\rvert_{\g{g}_i\times\g{g}_j}=0$ for all $i$ and $j$ with $i\neq j$.
	
Denote by $\g{p}$ and $\g{m}$ (respectively $\g{p}'$ and $\g{m}'$) the subspaces $ \g{k}^\perp$ and $\g{h}^\perp \cap\g{k}$ with respect to $\llangle\cdot,\cdot\rrangle$ (respectively $\langle\cdot,\cdot\rangle$). Consider the linear isomorphism $\varphi\colon \g{g}\to \g{g}$ given by $\varphi(\sum_i x_i)=A^{-1}x_0+\sum_{i\geq 1} s_i^{-1}x_i$, where $x_i\in\g{g}_i$. Note that for any $T=\sum_i T_i\in\g{g}$, $T_i\in\g{g}_i$, we have $\langle \varphi(\sum_i x_i), T\rangle=\langle A^{-1}x_0+\sum_{i\geq 1} s_i^{-1}x_i, \sum_i T_i\rangle=\llangle\sum_ix_i,T\rrangle$.  Taking $T\in \g{k}$ and $x\in \g{p}$, this implies that $\varphi$ maps $\g{p}$ bijectively into $\g{p}'$.   Likewise, taking $T\in \g{h}$ and $x\in \g{m}$, we see that $\varphi$ maps $\g{m}$ bijectively into $\g{m}'$.

Let $x=\sum_i x_i\in\g{p}$, $y=\sum_i y_i\in\g{m}$, where $x_i,y_i\in\g{g}_i$. Then $y\in Z_\g{m}(x)$ if and only if $0=[y,x]=\sum_i[y_i,x_i]$, which is equivalent to $[y_i,x_i]=0$ for all $i$. Similarly, $\varphi(y)\in Z_{\g{m}'}(\varphi(x))$ if and only if $0=[A^{-1}y_0+\sum_{i\geq 1} s_i^{-1}y_i,A^{-1}x_0+\sum_{i\geq 1} s_i^{-1}x_i]=\sum_{i\geq 1} s_i^{-2} [y_i,x_i]$, which again is equivalent to $[y_i,x_i]=0$ for all $i$. Therefore, $y\in Z_\g{m}(x)$ if and only if $\varphi(y)\in Z_{\g{m}'}(\varphi(x))$. Since $\varphi$ is an isomorphism, it follows that $\dim Z_\g{m}(x)=\dim Z_{\g{m}'}(\varphi(x))$. Similarly, we see that $\dim Z_\g{p}(y)=\dim Z_{\g{p}'}(\varphi(y))$. Consequently, the claim follows.
\end{proof}

Next, we discuss restrictions for the existence of triples $\g{h}\subset\g{k}\subset\g{g}$ with $f=0$ or $f=1$. As mentioned above, all triples with $f=0$ were classified; let us recall a couple of properties of these triples. First, by \cite[Lemma~4]{B75} it holds that $\rank \g{g}=\rank\g{k}$, which implies that $\dim \g{p}$ is even. Second, it holds that $\dim\g{m}\leq\dim\g{p}-1$, see e.g. \cite[Propositions~2.4 and~2.5]{Zi99} or \cite[Lemma~2.8.1]{GW09} for a proof in the more general context of fat Riemannian submersions. We show that analogous properties hold for the case $f=1$.

\begin{lemma}\label{LEM:f_1_dimensions}
Let $\g{h}\subset\g{k}\subset\g{g}$ be a triple.
\begin{enumerate}[label = \rm(\arabic*)]
\item If $f=0$, then $\dim\g{p}$ is even and $\dim\g{m}\leq\dim\g{p}-1$.
\item If $f=1$, then $\dim\g{p}$ is odd and $\dim\g{m}\leq\dim\g{p}$.
\end{enumerate}
\end{lemma}

\begin{proof}
The statements in the case $f=0$ are well known and the corresponding references have been given above. So let us assume for the rest of the proof that $f=1$.

Fix a background bi-invariant metric $\langle\cdot,\cdot\rangle$ on $\g{g}$. First note that there exists $y\in\g{m}\setminus\{0\}$ with $\dim Z_{\g{p}}(y)=1$. To show it, we argue by contradiction. Suppose $\dim Z_{\g{p}}(y)\neq 1$ for all $y\in\g{m}\setminus\{0\}$. Because $f=1$, it follows that $\dim Z_{\g{p}}(y)=0$ for all $y\in\g{m}\setminus\{0\}$. But the latter implies that $\dim Z_{\g{m}}(x)=0$ for all $x\in\g{p}\setminus\{0\}$. Altogether we get that $f=0$, a contradiction. Similarly, there exists $x\in\g{p}\setminus\{0\}$ with $\dim Z_{\g{m}}(x)=1$.

Take $y\in\g{m}\setminus\{0\}$ with $\dim Z_{\g{p}}(y)=1$ and define $V=Z_{\g{p}}(y)^\perp\cap\g{p}$. In particular, $\dim V=\dim\g{p}-1$. Consider the skew-symmetric bilinear form $\Omega_y\colon V\times V\to\R$ given by the rule
$$
\Omega_y(u,v)=\langle[u,v],y\rangle.
$$
Because $\langle\cdot,\cdot\rangle$ is bi-invariant, it follows that $\langle[u,v],y\rangle=\langle u,[v,y]\rangle$. We claim that, for each non-zero $v\in V$, $[v,y]\neq 0$ has a non-trivial orthogonal projection to $V$, which implies that $\Omega_y$ is non-degenerate and hence $V$ is even-dimensional. We prove the claim by contradiction. Suppose $[v,y]\in V^\perp \cap\g{p}= Z_{\g{p}}(y)$. Then $[[v,y],y]=0$, and it follows that $\langle [[v,y],y],v\rangle =0$. On the other hand, we compute
$$
\langle [[v,y],y],v\rangle =- \langle [v,y],[v,y]\rangle <0,
$$
a contradiction. This proves the first part, i.e., that $\dim\g{p}$ is odd.

For the second part, take $x\in\g{p}\setminus\{0\}$ with $\dim Z_{\g{m}}(x)=1$. Consider the linear map $T_x\colon\g{m}\to\g{p}$ defined by $T_x(y)=[y,x]$. The kernel is $1$-dimensional by construction. Since $\langle [y,x],x\rangle=\langle y,[x,x]\rangle=0$, it follows that the image of $T_x$ is contained in the orthogonal subspace to the line generated by $x$. Thus, the dimension of the image of $T_x$ is less or equal to $\dim\g{p}-1$.  The rank-nullity theorem applied to the linear map $T_x$ implies that $\dim\g{m}\leq\dim\g{p}$.
\end{proof}

We finish by recalling from \cite{KT14} some notions of equivalence for pairs and triples.

\begin{definition}\label{DEF:pair_triple_isomorphic}
A pair $\g{k}\subset\g{g}$ (resp.\  a triple $\g{h}\subset\g{k}\subset\g{g}$) is called pair-isomorphic (resp.\  triple-isomorphic) to another pair $\g{k}'\subset\g{g}'$ (resp.\  triple $\g{h}'\subset\g{k}'\subset\g{g}'$) if there is a Lie algebra isomorphism $\Phi\colon\g{g}\to\g{g}'$ such that $\Phi(\g{k})=\g{k}'$ (resp.\  $\Phi(\g{k})=\g{k}'$ and $\Phi(\g{h})=\g{h}'$).
\end{definition}

Clearly, isomorphic pairs (resp.\ isomorphic triples) have the same $b$ (resp.\ $f$). Unless otherwise stated, in this work pairs (resp.\ triples) will be considered up to pair-isomorphism (triple-isomorphism). We warn the reader that there can be non-triple-isomorphic triples $\g{h}\subset\g{k}\subset\g{g}$ and $\g{h}'\subset\g{k}\subset\g{g}$ such that $\g{h}\subset\g{k}$ is pair-isomorphic to $\g{h}'\subset\g{k}$. This can only happen if the pair isomorphism is an outer automorphism of $\g{k}$.

Next, note that given a pair $\g{k}\subset\g{g}$ or a triple $\g{h}\subset\g{k}\subset\g{g}$, one can construct new pairs and triples as $\g{k}\oplus\g{a}\subset\g{g}\oplus\g{a}$ or $\g{h}\oplus\g{a}\subset\g{k}\oplus\g{a}\subset\g{g}\oplus\g{a}$, where $\g{a}$ is any compact Lie algebra. These new pairs and triples, however, all have the same $b$ and $f$ as the original one, and the corresponding homogeneous spaces are diffeomorphic to those of the original pairs and triples. In this sense, we recall the following definition from \cite[Definition~4.2]{KT14} (note that other articles like \cite[p.~47]{On:transitive} or \cite{WZ85} use  the terminology of effective pairs or triples to refer to the same notion).

\begin{definition}\label{DEF:reduced_pair_triple}
A pair $\g{k}\subset\g{g}$ (resp.\  a triple $\g{h}\subset\g{k}\subset\g{g}$) is called reduced if the two (resp.\  three) Lie algebras do not share a common non-trivial ideal.
\end{definition}

In this paper we shall only consider reduced triples. However reduced triples might have non-reduced base or fiber pairs.

\section{Submersion metrics}\label{SEC:submersion_metrics}

In this section we first discuss a family of metrics on compact Lie groups $G$ induced by a subgroup $K$ and recall a result that determines their flat planes when the pair $K<G$ satisfies a certain property. Then, for a triple $H<K<G$, we estimate the intermediate Ricci curvature of the corresponding submersion metrics on $G/H$ in terms of the integers $b$ and $f$ defined in Section~\ref{SEC:b_and_f}.

\subsection{Metrics on Lie groups}\label{sec:MetricsOnLieGroups}

Let $K<G$ be a pair of compact Lie groups. Let $\langle\cdot,\cdot\rangle$ be a bi-invariant metric on $G$, and let $\g{p}$ denote the orthogonal complement of $\g{k}$ in $\g{g}$. Choose an arbitrary real number $t>0$, and take the bi-invariant product metric $(K\times G,t\langle\cdot,\cdot\rangle|_\mathfrak{k} + \langle\cdot,\cdot\rangle)$. Let us consider the action by right multiplication of the diagonal subgroup $\Delta K< K\times G$, \[  k_1(k_2,g)=(k_2k_1^{-1},gk_1^{-1}), \quad\text{where $g\in G$ and $k_1, k_2\in K$}. \] This action leads to a normal homogeneous metric of the quotient $(K\times G)/\Delta K$, which we denote by $\langle\cdot,\cdot\rangle_t$. The natural diffeomorphism that maps $(K\times G)/\Delta K$ to $G$ via $[k,g]\mapsto gk^{-1}$ transforms the action by left multiplication $(k_1,g_1)[k_2,g_2]=[k_1k_2,g_1g_2]$ into the $(K\times G)$-action on $G$ defined by $(k_1,g_1)g_2=g_1g_2k_1^{-1}$.  In particular, the push-forward of the left $(K\times G)$-invariant metric $\langle\cdot,\cdot\rangle_t$ on $(K\times G)/\Delta K$ gives a metric on $G$, which we denote also by $\langle\cdot,\cdot\rangle_t$ for convenience. This metric is called a Cheeger deformation, and was first introduced in \cite{Cheeger73}. The metric is left $G$-invariant and right $K$-invariant, and moreover it is even right $N_G(K)$-invariant, where $N_G(K)$ denotes the normalizer of $K$ in $G$  \cite[Proposition~3.1]{DVN}. To sum up we have the following:
\begin{equation}\label{eq:metric_G_deformation}		
			(K\times G,t\langle\cdot,\cdot\rangle|_\mathfrak{k} + \langle\cdot,\cdot\rangle) \to ((K\times G)/\Delta K,\langle\cdot,\cdot\rangle_t) \cong (G,\langle\cdot,\cdot\rangle_t),
\end{equation}	
where the first map is a Riemannian submersion. Observe that $\langle\cdot,\cdot\rangle_t$ is of $\sec\geq 0$ by O'Neill's formula \cite{ON}.

Eschenburg proved that when $(G,K)$ is a symmetric pair,  two linearly independent vectors $x,y\in\mathfrak{g}$ span a zero-curvature plane with respect to $\langle\cdot,\cdot\rangle_t$ if and only if the following is satisfied:
			\[	[x,y]=[x_\mathfrak{k},y_\mathfrak{k}]=[x_\mathfrak{p},y_\mathfrak{p}]=0,
			\]
where the subscripts denote the $\g{k}$ and $\g{p}$-components of a vector. 
In particular, $\langle\cdot,\cdot\rangle_t$ has no more zero-curvature planes than the bi-invariant metric $\langle\cdot,\cdot\rangle$; see \cite[Satz~231]{Es84} or \cite[Lemma 4.2]{Zi07}. Recently, the first author and Nance observed that the same conclusion holds for other pairs $K<G$ which are not symmetric but satisfy the following property.

\begin{definition}\label{def:pprime}
Let $G$ be a compact Lie group with bi-invariant metric $\langle\cdot,\cdot\rangle$. Let $\g{g}$ be its Lie algebra and $\g{k}\subset\g{g}$ a proper subalgebra. We say that the metric $\langle\cdot,\cdot\rangle$ satisfies Property~(P') if for any $x,y\in\g{p}=\g{k}^\perp$ with $[x,y]_\g{k}=0$ it also holds that $[x,y]=0$.
\end{definition}		

The main source for examples satisfying (P') are symmetric pairs, since they are characterized by the property $[\mathfrak{p},\mathfrak{p}]\subset\mathfrak{k}$, but we will show that there are non-symmetric pairs which satisfy (P').
 
With this terminology, the result of the first author and Nance can be stated as follows (see \cite[Lemmas~3.2 and 3.4]{DVN}).

\begin{theorem}\label{THM:DeVito_Nance}
Let $\g{k}\subset\g{g}$ be a pair and $\langle\cdot,\cdot\rangle$ a bi-invariant metric on $G$ satisfying (P'). Then, for any $t>0$, the vectors $x,y\in\mathfrak{g}$ span a zero-curvature plane with respect to $\langle\cdot,\cdot\rangle_t$ if and only if the following conditions are satisfied:
			\[	[x,y]=[x_\mathfrak{k},y_\mathfrak{k}]=[x_\mathfrak{p},y_\mathfrak{p}]=0.
			\]
\end{theorem}

In later applications we shall consider Property (P') only on pairs $\g{k}\subset\g{g}$ with $b=1$. In this case we have:

\begin{lemma}\label{lem:independence_prop_P_b_1}
Let $\g{k}\subset\g{g}$ be a pair with $b=1$. If a bi-invariant metric on $G$ satisfies Property (P'), then all bi-invariant metrics on $G$ satisfy Property (P').
\end{lemma}

This lemma suggests the following definition, which only depends on the pair.

\begin{definition}\label{def:p}
Let $\g{k}\subset\g{g}$ be a pair with $b=1$. We say that $\g{k}\subset\g{g}$ satisfies Property (P) if all bi-invariant metrics on $G$ satisfy Property (P').
\end{definition}

The proof of Lemma~\ref{lem:independence_prop_P_b_1} uses the well-known classification of pairs with $b=1$, which is discussed in Section~\ref{SEC:classification_triples}. Therefore, we postpone the proof of Lemma~\ref{lem:independence_prop_P_b_1} to Section~\ref{SEC:classification_triples}. Moreover, in Proposition~\ref{PROP:pairs_with_P} we shall determine which pairs with $b=1$ satisfy (P). 

\begin{remark}
It can be shown that Lemma~\ref{lem:independence_prop_P_b_1} holds for all pairs, i.e., without the assumption $b=1$. The proof is trivial in some particular cases, for example when $\g{g}$ is simple. However, the proof that we have found for the general case is rather extensive, and hence, we omit it here.
\end{remark}

\subsection{Metrics on homogeneous spaces}\label{SS:submersion_metrics}

Let $H<K<G$ be a triple of compact Lie groups. Let $\langle\cdot,\cdot\rangle_t$ be the metric on $G$ from Equation~\eqref{eq:metric_G_deformation}. The action by right multiplication by $H$ on $(G,\langle\cdot,\cdot\rangle_t)$ is by isometries since $H<K$, hence the quotient $G/H$ inherits a metric $q_t$ for which the quotient map
\begin{equation}\label{eq:qt}
	(G,\langle\cdot,\cdot\rangle_t) \to (G/H, q_t)
\end{equation}
is a Riemannian submersion. We call this metric $q_t$ a \emph{submersion metric} on $G/H$ induced by the triple $H<K<G$. Observe that $q_t$ is of $\sec\geq 0$ by O'Neill's formula. The metric $q_t$ is by construction homogeneous (i.e., invariant under the left action by $G$) and it is also invariant under the right $N_K(H)$-action $k\ast gH=gk^{-1}H$, with $k\in N_K(H)$ and $gH\in G/H$, where $N_K(H)$ denotes the normalizer of $H$ in $K$.

Now we discuss the intermediate Ricci curvature of $q_t$ under the assumption that the base space $G/K$ satisfies Property~(P), thus proving one of the key results of this paper.

\begin{theorem}\label{THM:Wallach_generalization}
	Let $H < K < G$ be a triple with fatness coindex $f$ and such that $G/K$ satisfies (P). Define
	\[
	k:=b(K/H)+b(G/K)+f-1.
	\]
	If $k<\dim G/H$, then any homogeneous metric $q_t$ defined on $G/H$ as in Equation~\eqref{eq:qt} has $\Ric_k>0$. 
\end{theorem}

\begin{proof}
	We argue by contradiction. Suppose that there is a subset of orthonormal vectors $\{x,y^1,\dots,y^k\}\subset\mathfrak{m}\oplus\mathfrak{p} \cong T_{eH} G/H$ satisfying $\sum_{i=1}^k\sec_{q_t}(x,y^i)\leq 0$. Since $q_t$ is of $\sec\geq 0$ it follows that $\sec_{q_t}(x,y^i)=0$ for every $i\in\{1,\dots, k\}$. 
	
Because $(G,\langle\cdot,\cdot\rangle_t)$ has non-negative sectional curvature and $(G,\langle\cdot,\cdot\rangle_t)\to (G/H,q_t)$ is a Riemannian submersion, it follows that, for each $i\in\{1,\dots, k\}$, the $2$-plane spanned by $x,y^i\in \mathfrak{m}\oplus\mathfrak{p}\subset\mathfrak{g}$ is also flat with respect to $\langle\cdot,\cdot\rangle_t$. Since $G/K$ satisfies (P), it follows from Theorem~\ref{THM:DeVito_Nance} that, for every $i\in\{1,\dots, k\}$,
	\begin{equation}\label{EQ:brackets}
		[x,y^i]=[x_\mathfrak{k},y^i_\mathfrak{k}]=[x_\mathfrak{p},y^i_\mathfrak{p}]=0.
	\end{equation}
	On the one hand, the identities $[x_\mathfrak{p},y^i_\mathfrak{p}]=0$ imply that $x_\mathfrak{p}=0$ or $\dim \spann\{x_\g{p},y_\g{p}^1,\dots, y_\g{p}^k\}\leq b(G/K)$. On the other hand, since $x,y^{i}\in\mathfrak{m}\oplus\mathfrak{p}$ and $\mathfrak{m}=\mathfrak{h}^\perp \cap\mathfrak{k}$,  we deduce that the identity $[x_\mathfrak{k},y^{i}_\mathfrak{k}]=0$ is nothing but $[x_\mathfrak{m},y^{i}_\mathfrak{m}]=0$. Hence, analogously as before, we have that $x_\mathfrak{m}=0$ or $\dim \spann\{x_\g{m},y_\g{m}^1,\dots, y_\g{m}^k\}\leq b(K/H)$.
	
	Now we divide the discussion into three cases. Suppose first that $x_{\g{p}}\neq 0$ and $x_{\g{m}}\neq 0$. Then, by the previous considerations,
	\begin{align*}
		k+1&=\dim\spann\{x,y^1,\dots,y^k\}
		%\dim\spann\{x,y^1_\g{p},\dots,y^k_\g{p},y^1_\g{m},\dots,y^k_\g{m}\}
		\leq \dim \spann\{x_\g{p},x_\g{m},y^1_\g{p},\dots,y^k_\g{p},y^1_\g{m},\dots,y^k_\g{m}\}
		\\ &\leq b(G/K)+b(K/H)=k-f+1,
	\end{align*}
	which yields a contradiction if $f>0$. If $f=0$, then by~\cite[Lemma~6]{B75}, $b(K/H)=1$, and hence $k=b(G/K)$. Therefore, since $x_\g{p}\neq 0$ and $\dim \spann\{x_\g{p},y_\g{p}^1,\dots, y_\g{p}^k\}\leq b(G/K)=k$, there exists a non-zero $\bar{y}\in\spann\{y^1,\dots,y^k\}$ such that $\bar{y}_\g{p}=\lambda x_\g{p}$, for some $\lambda\in\R$. As $x_\g{m}\neq 0$, $b(K/H)=1$, and $[x_\g{m},y_\g{m}^i]=[x_\g{k},y_\g{k}^i]=0$ because of Equation~\eqref{EQ:brackets} we also have $\bar{y}_\g{m}=\mu x_\g{m}$, for some $\mu\in\R$. Since the identities $[x,y^i]=0$ from Equation~\eqref{EQ:brackets} imply $[x,\bar{y}]=0$, we get $0=(\mu-\lambda)[x_\g{p},x_\g{m}]$, and hence $\lambda=\mu$ (because $f=0$, $x_\g{p}\neq 0 \neq x_\g{m}$), which implies that $\bar{y}$ is proportional to $x$, a contradiction.
	
	Suppose next that $x_{\g{p}}= 0$, and hence $x\in\g{m}$ and $\dim \spann\{x_\g{m},y_\g{m}^1,\dots, y_\g{m}^k\}\leq b(K/H)$. For each $i\in\{1,\dots,k\}$, Equation~\eqref{EQ:brackets} gives $0 = [x,y^i] = [x,y^i_{\g{m}}] + [x, y^i_{\g{p}}]$, so $[x,y^i_{\g{p}}] = -[x,y^i_{\g{m}}]$.  Since $x = x_{\g{m}} = x_{\g{k}}$, we find $[x,y^i_{\g{m}}] = [x_{\g{m}}, y^i_{\g{m}}] = [x_{\g{k}}, y^i_{\g{k}}] = 0$, again by Equation~\eqref{EQ:brackets}.  Thus $[x,y^i_{\g{p}}] = -[x,y^i_{\g{m}}] = 0$, for each $i=1,\dots, k$. But then, by definition of $f$ we have $\dim \spann\{y^1_\g{p},\dots,y^k_\g{p}\}\leq f$. Hence,
	\begin{align*}
		k+1&=\dim\spann\{x,y^1,\dots, y^k\}\leq \dim \spann\{x,y^1_\g{p},\dots, y^k_\g{p},y^1_\g{m},\dots,y^k_\g{m}\}
		\\
		&\leq f+b(K/H)=k-b(G/K)+1,
	\end{align*}
	which is a contradiction.
	
	Suppose finally that $x_{\g{m}}= 0$, so that $x\in\g{p}$ and $\dim \spann\{x_\g{p},y_\g{p}^1,\dots, y_\g{p}^k\}\leq b(G/K)$. Analogously as in the previous paragraph, we have $0=[x,y^i_\g{m}]$, for each $i=1,\dots,k$, and then $\dim\spann\{y_\g{m}^1,\dots,y_\g{m}^k\}\leq f$.  Hence, 
	\begin{align*}
		k+1&=\dim\spann\{x,y^1,\dots, y^k\}\leq \dim \spann\{x,y^1_\g{p},\dots, y^k_\g{p},y^1_\g{m},\dots,y^k_\g{m}\}
		\\
		&\leq f+b(G/K)=k-b(K/H)+1,
	\end{align*}
	which yields again a contradiction.
\end{proof}

\begin{remark}
	Theorem~\ref{THM:Wallach_generalization} generalizes \cite[Theorem~F]{DGM} for the case when $f=0$ and $G/K$ is symmetric. This, in turn, recovers the classical construction of Wallach~\cite{Wallach} of positively curved metrics on the total spaces of fat homogeneous bundles. 
\end{remark}

In this paper, our interest in Theorem~\ref{THM:Wallach_generalization} focuses in its application to construct metrics with $\Ric_2>0$, as stated in the following result, which follows directly from Theorem~\ref{THM:Wallach_generalization}.

\begin{corollary}\label{COR:Wallach_generalization}
	Suppose that $H < K < G$ satisfies the following properties:
	\begin{enumerate}[label = \rm(\arabic*)]
		\item $G/K$ satisfies (P),
		\item $b(G/K)=1$,
		\item $b(K/H)=1$,
		\item $f\leq 1$.
	\end{enumerate}
	Then, for any $t>0$, any homogeneous metric $q_t$ defined on $G/H$ as in Equation~\eqref{eq:qt} has $\Ric_2>0$.
\end{corollary}

\section{Some pairs and triples}\label{sec:sometriples}
In this section we perform explicit computations to calculate or estimate $f$ and $b$ for some particular triples and pairs.

Given a vector $\vec z=( z_1,\dots,z_n)\in\R^{n}$, we define the matrix $M(\vec z)\in\g{so}_{n+1}$ as:
		\[
M(\vec z)=M(z_1,\dots,z_n):=	 \left(\begin{array}{ccc|c}
			 &  & & z_1
			\\
			 & 0 &  & \vdots
			\\
			 &  &  & z_n
			\\ \hline
			-z_1 & \dots & -z_n & 0
		\end{array}\right)\in \g{so}_{n+1}.
		\]

\begin{lemma}\label{LEM:SOn_not_almost_fat}
The triple $\g{so}_{n-1}\subset\g{so}_{n}\subset\g{so}_{n+1}$, induced by block diagonal embeddings $x\mapsto\diag(x,0)$, has $f>1$ if $n\geq 4$.
\end{lemma}

\begin{proof}
Note that for any bi-invariant metric on $\so{n+1}$ we have $\g{m}=\{\diag(M(\vec w),0): \vec w\in\R^{n-1}\}$ and $\g{p}=\{M(\vec z):\vec z\in\R^{n}\}$. Take $y=\diag(M(1,0,\dots,0),0)\in\g{m}$ and an arbitrary $x=M(\vec z)\in\g{p}$. Then $$[y,x]=M(z_n,0,\dots,0,-z_1).$$
It follows that $\dim Z_{\g{p}}(y)=n-2$, thus $f>1$ if $n-2> 1$.
\end{proof}

\begin{lemma}\label{LEM:SO4_almost_fat_}
Let $\g{so}_3\subset\g{so}_4$ be the inclusion given by the block diagonal embedding $x\mapsto\diag(x,0)$. The following triples have $f=1$:
\begin{itemize}
\item $\{0\}\subset\g{so}_{3}\subset\g{so}_{4}$ and $\g{so}_{2}\subset\g{so}_{3}\subset\g{so}_{4}$,
\item $\g{u}_{1} \subset \g{u}_2 \cong \g{u}_1\oplus\g{so}_3\subset\g{u}_1\oplus\g{so}_4$, where the Lie subalgebra $\g{u}_{1} \subset \g{u}_2$ has non-trivial projection into the $\g{so}_3$-factor of $\g{u}_2$,
\item $\g{so}_3 \subset \g{so}_4\cong\g{so}_3\oplus\g{so}_3\subset\g{so}_3\oplus\g{so}_4$.
\end{itemize}
\end{lemma}

\begin{proof}
Let $\g{f}\in\{\{0\},\g{u}_1,\g{so}_3\}$. Except for the second triple in the statement, the remaining three triples fit into the following pattern:
\[
\g{h}=\{(w,\varphi(w)):w\in\g{f}\} \quad \subset\quad \g{k}=\g{f}\oplus\g{so}_3\quad\subset \quad \g{g}=\g{f}\oplus\g{so}_4,
\]
where $\varphi\colon \g{f}\to\g{so}_3$ is any injective Lie algebra homomorphism, and where we consider the standard inclusion $\g{so}_3\subset\g{so}_4$ by adding a final row and column of zeros. 

Fix a bi-invariant metric on $\g{g}$, and denote by $\psi$ the opposite of the adjoint of $\varphi$ with respect to the induced metrics on $\g{f}$ and $\g{so}_3$. Note that $\psi$ is surjective and  $(\psi(A),A)$ is orthogonal to $(w,\varphi(w))$ for all $A\in\g{so}_3$ and all $w\in\g{f}$. We find that
\[
\g{p}=\{(0,M(\vec z)):\vec z\in\R^3\} \qquad\text{and}\qquad \g{m}=\{(\psi(A),A):A\in \g{so}_3\subset\g{so}_4\}.
\]
Let $x=(0, M(\vec z))\in\g{p}\setminus\{0\}$ and $y=(\psi(A),A)\in\g{m}\setminus\{0\}$. Then $[x,y]=0$ if and only if $A\vec z=0$ (regarding $A\in\g{so}_3$ as a $3\times 3$ matrix). On the one hand, since $\ker A$ has dimension~$1$, for every $A\in\g{so}_3$, $A\neq 0$, we get that $\dim Z_{\g{p}}(y)=1$. On the other hand, the Lie algebra of matrices that have $\vec z$ in their kernel, $\{B\in\g{so}_3:B\vec z=0\}$, is isomorphic to $\g{so}_2$. Hence, $\dim Z_\g{m}(x)=1$. 

Finally, note that, since the subspace $\g{m}$ for the triple $\g{so}_{2}\subset\g{so}_{3}\subset\g{so}_{4}$ is a subset of the $\g{m}$ corresponding to $\{0\}\subset\g{so}_{3}\subset\g{so}_{4}$, it follows that $\g{so}_{2}\subset\g{so}_{3}\subset\g{so}_{4}$ has $f\leq 1$. Since $\dim\g{p}$ is odd, Lemma~\ref{LEM:f_1_dimensions} implies that $f=1$.
\end{proof}

Now we discuss pairs and triples where one of the algebras equals $\g{g}_2$. We extract the following information from \cite{Kerin11} and \cite[Section~2]{DVN}. 

The subalgebra $\g{g}_2\subset\g{so}_7$ consists of all matrices of the form
\begin{equation}\label{EQ:g2_so7}
\begin{pmatrix}
0 & x_1+x_2 & y_1+y_2 & x_3+x_4 & y_3+y_4 & x_5 + x_6 & y_5 + y_6\\
-(x_1+x_2) & 0 & z_1 & -y_5 & x_5 & -y_3 & x_3 \\
-(y_1+y_2) & -z_1 & 0 & x_6 & y_6 & -x_4 & -y_4 \\
-(x_3+x_4) & y_5 & - x_6 & 0 & z_2 & y_1 & -x_1\\
-(y_3+y_4) & -x_5 & -y_6 & -z_2 & 0 & x_2 & y_2 \\
-(x_5 + x_6) & y_3 & x_4 & -y_1 & -x_2 & 0 & z_1 + z_2 \\
-(y_5 + y_6) & -x_3 & y_4 & x_1 & -y_2 & -(z_1 + z_2) & 0 \\
\end{pmatrix}
\end{equation}
with $x_1,\dots,x_6,y_1,\dots,y_6,z_1,z_2\in\R$. 

For any bi-invariant metric on $\so{7}$ the subspace $(\g{g}_2)^\perp\cap\g{so}_7$ consists of all matrices $A(\vec v)=A(v_1,\dots,v_7)$ defined as:
$$
A(\vec v)=
\begin{pmatrix}
0 & v_1 & v_2 & v_3 & v_4 & v_5 & v_6\\
-v_1 & 0 & v_7 & v_6 & -v_5 & v_4 & -v_3 \\
-v_2 & -v_7 & 0 & -v_5 & -v_6 & v_3 & v_4 \\
-v_3 & -v_6 & v_5 & 0 & v_7 & -v_2 & v_1\\
-v_4 & v_5 & v_6 & -v_7 & 0 & -v_1 & -v_2 \\
-v_5 & -v_4 & -v_3 & v_2 & v_1 & 0 & -v_7 \\
-v_6 & v_3 & -v_4 & -v_1 & v_2 & v_7 & 0 \\
\end{pmatrix}, \qquad \vec v=(v_1,\dots,v_7)\in\R^7.
$$

The following lemma was essentially proved by Kerin in~\cite[Lemma~4.2 and discussion below]{Kerin11}, but we give a slightly different proof here for completeness.

\begin{lemma}\label{LEM:G2_SO7_SO8}
Consider the triple $\g{g}_2\subset\g{so}_7\subset\g{so}_{8}$, where $\g{g}_2\subset\g{so}_7$  is given by Equation~\eqref{EQ:g2_so7} and $\g{so}_7\subset\g{so}_{8}$ is given by the block diagonal embedding $x\mapsto\diag(x,0)$. Then, the total space pair $\g{g}_2\subset\g{so}_{8}$ has $b=2$.
\end{lemma}

\begin{proof}
Consider $\g{h}\subset\g{k}\subset\g{g}$, with $\g{h}=\g{g}_2$, $\g{k}=\g{so}_7$ and $\g{g}=\g{so}_{8}$. For any bi-invariant metric on $\so{8}$ we have $\g{m}=\{ \diag (A(\vec v),0) : \vec v\in\R^7\}$ and $\g{p}=\{M(\vec z): \vec z\in\R^7\}$. By definition, the subspace $\g{g}_2^\perp\cap\g{so}_8$ equals $\g{m}\oplus\g{p}$. It is well known that the group $H=\gg$ acts (by conjugation) transitively on pairs of orthonormal vectors both in $\g{m}$ and in $\g{p}$ (since $\gg/\su{2}\cong T^1\sph^{6}$).  Hence, by computing the brackets $[\diag(A(\vec e_1),0),\diag(A(\vec e_2),0)]$ and $[M(\vec e_1), M(\vec e_2)]$, where $\vec e_i$ is the $i$-th vector of the canonical basis of $\R^7$, it follows that the bracket of any two matrices in $\g{m}$ has rank equal to $6$ or $0$, whereas the bracket of any two matrices in $\g{p}$ has rank $2$ or $0$.

Let $x,y\in\g{m}\oplus\g{p}$ with $[x,y]=0$. Then the vanishing of the $\g{k}$-component of $[x,y]$ together with the fact that the pair $\g{k}\subset\g{g}$ is symmetric implies $[x_\g{m},y_\g{m}]=-[x_\g{p},y_\g{p}]$. By the observation in the previous paragraph, these brackets must have rank 0 and thus be zero. Since $b(\g{g}_2\subset \g{so}_7)=1$ and $b(\g{so}_{7}\subset\g{so}_8)=1$, it follows that $x_\g{m}$ and $y_\g{m}$ are proportional, and so are $x_\g{p}$ and $y_\g{p}$. Hence, both $x$ and $y$ belong to the subspace $\mathrm{span}\{x_\g{m},y_\g{m},x_\g{p},y_\g{p}\}$, which is at most $2$-dimensional. Thus, $b(\g{g}_2\subset\g{so}_{8})\leq 2$. By taking $x=M(\vec e_3)\in\g{p}$ and $y=\diag(A(\vec e_1),0)\in\g{m}$, we have $[x,y]=0$, which actually proves that $b(\g{g}_2\subset\g{so}_{8})= 2$.
\end{proof}

Now we recall the embedding $\g{su}_3\subset\g{g}_2$, see~\cite{DVN}. The subalgebra $\g{su}_3$ consists of all matrices as in Equation~\eqref{EQ:g2_so7} whose first row and column vanish.

For any bi-invariant metric on $\gg$, the orthogonal subspace $\g{su}_{3}^\perp\cap\g{g}_2$ consists of all matrices $C(\vec z)=C(z_1,z_2,z_3,z_4,z_5,z_6)$ of the form:
$$
C(\vec z)=
\frac{1}{2}\begin{pmatrix}
0 & 2z_1 & 2z_2 & 2z_3 & 2z_4 & 2z_5 & 2z_6\\
-2z_1 & 0 & 0 & -z_6 & z_5 & -z_4 & z_3 \\
-2z_2 & 0 & 0 & z_5 & z_6 & -z_3 & -z_4 \\
-2z_3 & z_6 & -z_5 & 0 & 0 & z_2 & -z_1\\
-2z_4 & -z_5 & -z_6 & 0 & 0 & z_1 & z_2 \\
-2z_5 & z_4 & z_3 & -z_2 & -z_1 & 0 & 0 \\
-2z_6 & -z_3 & z_4 & z_1 & -z_2 & 0 & 0 \\
\end{pmatrix},\qquad \vec z=(z_1,\dots,z_6)\in\R^6.
$$

\begin{lemma}\label{lem:su3_g2_so7}
The triple $\g{su}_3\subset\g{g}_2\subset\g{so}_7$ has $f=1$. The total space pair satisfies $b(\g{su}_3\subset\g{so}_7)\geq 3$.
\end{lemma}

\begin{proof}

Fix any bi-invariant metric on $\so{7}$. First we show that $\dim Z_{\g{p}}(y)=1$ for all $y\in\g{m}=\g{su}_{3}^\perp\cap\g{g}_2$. We note that the isotropy action of $\su{3}$ on $\g{m}$ is of cohomogeneity one (or, equivalently, transitive on the set of unit vectors in $\g{m}$), since it is equivalent to the standard irreducible representation of $\su{3}$ in $\mathbb{C}^3$.

Hence,  it is sufficient to show that $\dim Z_{\g{p}}(y)=1$ for some non-zero $y\in\g{m}$. Take $y=C(1,0\dots,0)\in\g{m}$ and an arbitrary $x=A(v_1,\dots,v_7)\in\g{p}=\g{g}_2^\perp\cap\g{so}_7$. We compute:
$$
[y,x]=\frac{1}{2}A(0,2v_7,v_6,-v_5,v_4,-v_3,-2v_2).
$$
We clearly see that $\dim Z_{\g{p}}(y)=1$.

Now we show that $\dim Z_{\g{m}}(x)\leq 1$ for all non-zero $x\in\g{p}$. The isotropy action of $\mathsf{G}_2$ on $\g{p}$ is the $7$-dimensional irreducible representation of $\mathsf{G}_2$. Moreover, $\su{3}$ is the stabilizer for this action at $A(0,\ldots,0,1)$. Indeed, the restriction of the isotropy representation of $\spin{7}/\gg$ to $\su{3}$ splits as the trivial submodule spanned by $A(0,\ldots,0,1)$ and its orthogonal complement in $\g{p}$, which we denote by $\g{a}$. Clearly, $\g{a}$ is equivalent to the standard representation of $\su{3}$ in $\mathbb{C}^3$, and therefore $\su{3}$ acts with cohomogeneity one on $\g{a}$. It follows that the set of elements of the form $A(v_1,0,\dots,0,v_7)$ intersects all $\su{3}$-orbits in $\g{p}$. Thus, it is sufficient to show that $\dim Z_{\g{m}}(x)\leq 1$ for all $x\in\g{p}\setminus\{0\}$ of the form $x=A(v_1,0,\dots,0,v_7)$. Take an arbitrary $y=C(z_1,z_2,z_3,z_4,z_5,z_6)$. We compute:
$$
[y,x]=\frac{1}{2} A(-2v_7z_2,2v_7z_1,-v_1z_6-2v_7z_4,v_1z_5+2v_7z_3,-v_1z_4+2v_7z_6,v_1z_3-2v_7z_5,2v_1z_2).
$$
We split the discussion in two cases. If $v_1=0$ and $v_7\neq 0$, then $[y,x]=0$ if and only if $z_i=0$ for all $1\leq i\leq 6$, so $\dim Z_{\g{m}}(x)=0$. Now suppose that $v_1\neq0$. We can assume $v_1=1$ without loss of generality. If $[y,x]=0$, then the following equations hold:
$$
z_2=0,\qquad \begin{cases}
  -2v_7z_4-z_6=0,\\
  -z_4+2v_7z_6=0,
\end{cases}\qquad \begin{cases}
  2v_7z_3+z_5=0,      \\
  z_3-2v_7z_5=0.
\end{cases}
$$ 
The determinant of both $2\times 2$-systems equals $-(1+4v_7^2)$, and hence the only solutions are $z_4=z_6=0$ and $z_3=z_5=0$ respectively. This shows that $\dim Z_{\g{m}}(x)= 1$. Altogether, we conclude that $f=1$.

To prove that $b(\g{su}_3\subset\g{so}_7)\geq 3$, note that there is an intermediate subalgebra $\g{su}_3\subset\g{so}_6\subset\g{so}_7$ by taking all matrices of $\g{so}_7$ whose first row and column vanish. This new triple satisfies $\dim\g{p}=6$ and $\dim\g{m}=7$, so $f>1$ by Lemma~\ref{LEM:f_1_dimensions}. It follows from Lemma~\ref{LEM:f_and_b} that the total space pair $\g{su}_3\subset\g{so}_7$ has $b\geq 3$.
\end{proof}

We conclude this section by proving that the pair $\g{g}_2\subset\g{so}_7$ satisfies Property~(P).

\begin{lemma}\label{LEM:g2_so7_satisfies_P}
For any bi-invariant metric on $\gg$ and any $x,y\in\g{g}_2^\perp\cap\g{so}_7$, it holds that $[x,y]=0$ if and only if $[x,y]_{\g{g}_2}=0$. In other words, the pair $\g{g}_2\subset\g{so}_7$ satisfies (P).
\end{lemma}

\begin{proof}
We shall prove that $[x,y]_{\g{g}_2}\neq 0$ for any $x,y\in\g{g}_2^\perp\cap\g{so}_7$ orthogonal and of unit length, from which the statement follows immediately. Recall from the proof of Lemma~\ref{LEM:G2_SO7_SO8} that $\gg$ acts transitively on the space $\{(z,t)\in\g{p}\oplus\g{p}: |z|=|t|=1, \langle z,t\rangle=0\}$, where $\g{p}=\g{g}_2^\perp\cap\g{so}_7$. Thus, it is sufficient to prove the statement for $x=A(\vec e_1)$ and $y=A(\vec e_2)$. Let $T\in\g{g}_2$ as in Equation~\eqref{EQ:g2_so7} with all coefficients equal to zero except $z_1=1$. Then it is immediate to compute that $\langle [x,y], T\rangle=-6$, which implies $[x,y]_{\g{g}_2}\neq 0$.
\end{proof}

\section{The classification theorem and the proof}\label{SEC:classification_triples}

The goal of this section is to prove the following theorem, which classifies all triples of Lie algebras satisfying the assumptions in Corollary~\ref{COR:Wallach_generalization}. 

\begin{theorem}\label{THM:submersion_triples}
Let $\g{h}\subset\g{k}\subset\g{g}$ be a reduced triple with $b(\g{k}\subset\g{g})=b(\g{h}\subset\g{k})=1$, and with base pair $\g{k}\subset\g{g}$ satisfying Property (P). Then the triple $\g{h}\subset\g{k}\subset\g{g}$ satisfies $f=1$ if and only if it is triple-isomorphic to one of the following:
\begin{enumerate}[label = \rm(\arabic*)]
\item $\g{so}_{2}\subset\g{so}_{3}\subset\g{so}_{4}$,
\item $\{0\}\subset\g{so}_{3}\subset\g{so}_{4}$,
\item $\g{u}_{1} \subset \g{u}_2 \cong \g{u}_1\oplus\g{so}_3\subset\g{u}_1\oplus\g{so}_4$, where the Lie subalgebra $\g{u}_{1} \subset \g{u}_2$ has non-trivial projection into the $\g{so}_3$-factor of $\g{u}_2$,
\item $\g{so}_3 \subset \g{so}_4\cong\g{so}_3\oplus\g{so}_3\subset\g{so}_3\oplus\g{so}_4$,
\item $\g{su}_{3} \subset \g{g}_2 \subset \g{so}_{7}$, or
\item $\g{g}_2\subset\g{so}_7\subset\g{so}_{8}$.
\end{enumerate}
\end{theorem}

\begin{remark}
	The inclusion $\g{so}_3\subset\g{so}_4$ in Items (1), (2), (3) and (4) in Theorem~\ref{THM:submersion_triples} is the standard one, i.e., conjugate to the $3\times 3$-block inclusion $x\mapsto\diag(x,0)$ or, equivalently, conjugate to the diagonal inclusion $\Delta\g{so}_3\subset\g{so}_3\oplus\g{so}_3\cong\g{so}_4$. The inclusions in Items (5) and (6) are those considered in Lemmas~\ref{LEM:G2_SO7_SO8} and \ref{lem:su3_g2_so7}.
\end{remark}

The proof of Theorem~\ref{THM:submersion_triples} consists of a case by case analysis. We consider all triples satisfying the assumptions in the statement and derive whether $f=1$. In order to analyze all triples, first we need to recall all pairs with $b=1$ and determine which of them satisfy Property (P).

In Table~\ref{TABLE:b1} the reader can find all reduced pairs $\g{k}\subset\g{g}$ with $b=1$, where $n$ can be taken to be any integer $n\geq 1$ in all possible cases. All such pairs were classified by Berger~\cite{Berger:normal} with an omission corrected by Wilking \cite{wilking:pams}, see also~\cite{WZ:crelle}. Note that the first line with $n=1$ in Table~\ref{TABLE:b1} equals the abelian algebra of dimension one, whose corresponding homogeneous space is the circle $\sg^1$.

The table also contains all intermediate subalgebras, which is helpful for various reasons, and a $\checkmark$ for each pair that satisfies (P). Except for the information about Property (P), which is extracted from Proposition~\ref{PROP:pairs_with_P} below, the table is exactly the same as in the article~\cite{KT14}. The left column ``Type'' will serve to refer to each triple. We now briefly explain the embeddings.

For Type 1, $\g{so}_n\subset\g{so}_{n+1}$ denotes the standard block inclusion. Let us remark that this is the only possible inclusion of $\g{so}_n$ in $\g{so}_{n+1}$ up to conjugacy, except for two cases. The first exception occurs when $n=3$, when one can also include $\g{so}_3$ into $\g{so}_4\cong \g{so}_3\oplus \g{so}_3$ as a factor. In this case the corresponding pair is not reduced (in the sense of Definition~\ref{DEF:reduced_pair_triple}), and the corresponding reduced pair $\{0\}\subset\g{so}_3$ corresponds to Type 2 for $n=1$. The other exception occurs when $n=7$, where there are two spin representations of $\spin{7}$ on $\R^8$, which induce two other embeddings $\g{so}_7\subset\g{so}_8$. The corresponding three pairs of the form $\g{so}_7\subset \g{so}_8$ are all isomorphic, but any isomorphism is necessarily outer, so no two of such $\g{so}_7$ are conjugate in $\g{so}_8$.

For Types 2, 3, 5, 6, 7, 8, 9, 10, and 12 the embedding is unique up to conjugacy, and is given by the standard inclusion.  For Type 4, the notation $\g{spin}^+_7$ refers to either of the spin embeddings, which are conjugate in $\g{spin}_9$.  For Type 11, the notation $\g{so}_3^{\text{max}}$ refers to the embedding induced via the unique irreducible $5$-dimensional representation of $\g{so}_3$; the image in $\g{so}_5$ is unique up to conjugacy.  For Type 13, the $\Delta \g{u}_2$ is diagonally included into $\g{su}_3\oplus \g{so}_3$, via the standard inclusion in $\g{su}_3$ and via the projection $\g{u}_2\cong\g{su}_2\oplus \g{u}_1\cong \g{so}_3\oplus \g{u}_1\rightarrow \g{so}_3$.  The image is, again, unique up to conjugacy.  For Type 14, if we write $\g{u}_n\cong \g{u}_1\oplus \g{su}_n$, then the $\g{su}_n$ factor is embedded up to conjugacy in the usual block fashion in $\g{su}_{n+1}\subset \g{u}_1\oplus \g{su}_{n+1}\cong \g{u}_{n+1}$.  The embedding of $\g{u}_1$ into $\g{u}_{n+1}$ is any of the infinitely many options that have non-trivial projections to both $\g{su}_{n+1}$ (otherwise the pair would not be reduced) and $\g{u}_1$ (otherwise the pair would have $b>1$).  Type 15 is similar to Type 14: the $\g{sp}_n$-factor is embedded into $\g{sp}_{n+1}$ via the standard inclusion, up to conjugacy.  The $\Delta \g{u}_1$ factor must have non-trivial projection to both the $\g{sp}_{n+1}$ and $\g{u}_1$-factors, which implies that there is a one-parameter family of pairs up to conjugacy.  Finally, for Type 16, the $\g{sp}_n$ is, up to conjugacy, embedded in the usual way into $\g{sp}_{n+1}$ while the $\g{sp}_1$-factor must have non-trivial projection to both $\g{sp}_{n+1}$ and $\g{sp}_1$. Hence, an embedding of Type 16 is unique up to conjugacy.

Even though we will only deal with reduced triples, note that the base pair or the fiber pair of a reduced triple can be non-reduced. Thus we fix the following terminology.

\begin{definition}
Let $X$ be an integer with $1\leq X\leq 16$. We say that a pair $\g{k}'\subset\g{g}'$ is of Type~$X$ if $\g{k}'\subset\g{g}'$ is pair-isomorphic to $\g{k}\oplus\g{a}\subset\g{g}\oplus\g{a}$ where $\g{k}\subset\g{g}$ is pair-isomorphic to the pair of Type $X$ in Table~\ref{TABLE:b1} and $\g{a}$ is some (possibly trivial) Lie algebra. We may specify a concrete parameter $n\geq 1$ for Types $X$ which consist of a family of pairs depending on $n$.
\end{definition}

\begin{center}
\begin{table}
\caption{Reduced pairs with $b=1$.}
\begin{tabular}{c l l c} 
 \hline
Type & $G/K$ & $\g{k}\subset$ any intermediate subalgebras $\subset\g{g}$ & (P)\\\hline
1 & $\sph^n$ & $\g{so}_n\subset\g{so}_{n+1}$ &  $\checkmark$\\ 
2 & $\sph^{2n+1}$ & $\g{su}_n\subset\g{u}_n\subset\g{su}_{n+1}$ & \\ 
3 & $\sph^{4n+3}$ & $\g{sp}_n\subset\g{u}_1\oplus\g{sp}_n\subset\g{sp}_1\oplus\g{sp}_n\subset\g{sp}_{n+1}$ & \\ 
4 & $\sph^{15}$ & $\g{spin}_7^{+}\subset\g{spin}_8\subset\g{spin}_9$ & \\ 
5 & $\sph^{7}$ & $\g{g}_2\subset\g{spin}_7$ & $\checkmark$ \\ 
6 & $\sph^{6}$ & $\g{su}_3\subset\g{g}_2$ & $\checkmark$ \\ 
7 & $\CP^{n}$ & $\g{u}_n\subset\g{su}_{n+1}$ & $\checkmark$\\ 
8 & $\CP^{2n+1}$ & $\g{sp}_n\oplus\g{u}_1\subset\g{sp}_n\oplus\g{sp}_1\subset\g{sp}_{n+1}$ & \\ 
9 & $\mathbb{HP}^{n}$ & $\g{sp}_n\oplus\g{sp}_1\subset\g{sp}_{n+1}$ & $\checkmark$\\
10 & $\mathbb{OP}^{2}$ & $\g{spin}_9\subset\g{f}_{4}$ & $\checkmark$\\
11 & $B^7$ & $\g{so}_3^{\max}\subset\g{so}_5$ & \\
12 & $B^{13}$ & $\g{sp}_2\oplus\g{u}_1\subset\g{u}_4\subset\g{su}_5$ & \\
13 & $W_{1,1}^7$ & $\Delta\g{u}_2\subset\g{u}_2\oplus\g{so}_3\subset\g{su}_3\oplus\g{so}_3$ & \\
14 & $\sph^{2n+1}$ & $\g{u}_n\subset\g{u}_1\oplus\g{u}_n\subset\g{u}_{n+1}$ & \\ 
15 & $\sph^{4n+3}$ & $\g{sp}_n\oplus\Delta\g{u}_1\subset\g{sp}_n\oplus\g{u}_1\oplus\g{u}_1\subset\g{sp}_n\oplus\g{sp}_1\oplus\g{u}_1\subset\g{sp}_{n+1}\oplus\g{u}_1$ & \\ 
16 & $\sph^{4n+3}$ & $\g{sp}_n\oplus\Delta\g{sp}_1\subset\g{sp}_n\oplus\g{sp}_1\oplus\g{sp}_1\subset\g{sp}_{n+1}\oplus\g{sp}_1$ & \\ 
 \hline
\end{tabular}\label{TABLE:b1}
\end{table}
\end{center}

Our next goal is to determine which pairs with $b=1$ satisfy Property (P), which is achieved in Proposition~\ref{PROP:pairs_with_P}. First, we need to provide the proof of Lemma~\ref{lem:independence_prop_P_b_1}. We will use the following observation.

\begin{lemma}\label{LEM:no_property_P}
Let $\g{k}\subset\g{g}$ be a pair with $b=1$ and suppose there is an intermediate algebra $\mathfrak{k}\subset\mathfrak{f}\subset\mathfrak{g}$. Then no bi-invariant metric on $\g{g}$ satisfies Property (P').
\end{lemma}

\begin{proof}
Fix any bi-invariant metric on $\g{g}$, and take any non-zero vectors $x\in\g{f}^\perp$ and $y\in\g{k}^\perp\cap\g{f}$. Then both $x,y$ belong to $\g{k}^\perp$. Because $\g{f}\oplus\g{f}^\perp$ is a reductive decomposition, it follows that $[\g{f},\g{f}^\perp]\subset\g{f}^\perp$ and hence $[x,y]_{\g{f}}=0$, which in particular implies $[x,y]_{\g{k}}=0$. On the other hand, the assumption $b=1$ implies that $[x,y]\neq 0$. Thus (P') is not satisfied.
\end{proof}

Now we are ready to give the proof of Lemma~\ref{lem:independence_prop_P_b_1}.

\begin{proof}[Proof of Lemma~\ref{lem:independence_prop_P_b_1}]
First note that it is enough to prove the lemma for all reduced pairs. So we assume that $\g{k}\subset\g{g}$ is one of the pairs in Table~\ref{TABLE:b1}. If $\g{g}$ is simple, then there is just one bi-invariant metric up to rescaling, and $\g{p}=\g{k}^\perp$ does not depend on the chosen metric. Thus, Lemma~\ref{lem:independence_prop_P_b_1} clearly holds for all pairs with $b=1$ and with $\g{g}$ simple.  We must therefore consider the cases where $\g{g}$ is not simple, consisting of Type 1 with $n=3$, as well as Types 13, 14, 15 and 16.

We discuss first the pair $\g{so}_3 \subset \g{so}_4$, which is isomorphic to  $\Delta\g{so}_3\subset\g{so}_3\oplus\g{so}_3$. Clearly, $\g{k}=\{(x,x): x\in\g{so}_3\}$.  All bi-invariant metrics on $\g{so}_3\oplus\g{so}_3$ are, up to scaling, of the form $Q_t((x_1,y_1),(x_2,y_2))=\langle x_1 ,y_1\rangle + t\langle x_2 ,y_2\rangle$ with $t>0$, where $\langle\cdot ,\cdot\rangle$ is a fixed bi-invariant metric on $\g{so}_3$. For the metric $Q_t$, we have $\g{p}=\{(-tx,x): x\in\g{so}_3\}$. Now, we take two arbitrary elements $X=(-tx,x)$, $Y=(-ty,y)$ in $\g{p}$. We have 
$$[X,Y]=(t^2[x,y],[x,y])= \underbrace{t([x,y],[x,y])}_{\in \g{k}} + \underbrace{ (1-t)(-t[x,y], [x,y])}_{\in \g{p}},$$ 
so $[X,Y]_{\g{k}}=t([x,y],[x,y])$. If the latter vanishes, then $[x,y]=0$ and hence $[X,Y]=0$. In other words, all metrics $Q_t$ satisfy Property (P').

Finally, the existence of intermediate subalgebras $\mathfrak{k}\subset\mathfrak{f}\subset\mathfrak{g}$ for pairs of Type 13, 14, 15 and 16 imply that they do not satisfy Property (P') for any bi-invariant metric, by Lemma~\ref{LEM:no_property_P}. 
\end{proof}

\begin{proposition}\label{PROP:pairs_with_P}
A pair $\g{k}\subset\g{g}$ with $b=1$ satisfies (P) if and only if it is of Type 1, 5, 6, 7, 9 or 10.
\end{proposition}

\begin{proof}

The pairs $\g{k}\subset\g{g}$ of Type 1 (with $n\neq 3$), 7, 9 and 10 are symmetric (i.e.\ $[\g{p},\g{p}]\subset\g{k}$) for any bi-invariant metric on $\g{g}$, so they clearly satisfy (P). As for Type 1 with $n=3$, in the proof of Lemma~\ref{lem:independence_prop_P_b_1} we have shown that it satisfies (P).

The pairs with intermediate subalgebras (i.e.\ Types 2, 3, 4, 8, 12, 13, 14, 15, 16) do not satisfy (P) by Lemma~\ref{LEM:no_property_P}. It remains to consider Types 5, 6, 11. It has been shown by the first author and Nance in \cite[Lemma~3.3]{DVN} that Type 6 satisfies (P), and we have shown in Lemma~\ref{LEM:g2_so7_satisfies_P} that Type 5 satisfies (P).

As for Type 11, we claim that for any non-zero $x\in \g{p}$ there is a $y\in \g{p}$ with $[x,y]_{\g{k}} =0$ and $[x,y]\neq 0$. To show it, note that $[x,z]\neq 0$ for any $z\in\operatorname{span}\{x\}^{\bot}\cap\g{p}$ because $b=1$. On the other hand, the linear map $\operatorname{span}\{x\}^{\bot}\cap\g{p}\rightarrow \g{k}$ defined by $z\mapsto [x,z]_{\g{k}}$ is a map from a $6$-dimensional space to a $3$-dimensional space so it has non-trivial kernel.  Selecting $y\neq 0$ in the kernel, then $[x,y]_{\g{k}} = 0$ but $[x,y]\neq 0$.
\end{proof}

Now we have all the needed ingredients for proving Theorem~\ref{THM:submersion_triples}.

\begin{proof}[Proof of Theorem~\ref{THM:submersion_triples}]
The assumption $f=1$ implies that the dimension of $\g{p}$ is odd, see Lemma~\ref{LEM:f_1_dimensions}. Since the base pair $\g{k}\subset\g{g}$ satisfies Property (P), Proposition~\ref{PROP:pairs_with_P} implies that $\g{k}\subset\g{g}$ must be of Type 1 (with $n$ odd) or 5. 

Suppose that the base pair is of Type 5. Then the triple is of the form $\g{h}\subset\g{a}\oplus\g{g}_2\subset\g{a}\oplus\g{so}_7$ for some Lie algebra $\g{a}$, where $\g{h}$ does not share the ideal $\g{a}$ with $\g{a}\oplus\g{g}_2$. The only fiber pair $\g{h}\subset\g{a}\oplus\g{g}_2$ with $b=1$ occurs when $\g{a}$ is trivial and $\g{h}=\g{su}_3$, see~Table~\ref{TABLE:b1}. The corresponding triple has $f=1$ by Lemma~\ref{lem:su3_g2_so7}.

Suppose that the base pair is $\g{so}_n\subset\g{so}_{n+1}$ with $n$ odd. Then the triple is of the form $\g{h}\subset\g{a}\oplus\g{so}_n\subset\g{a}\oplus\g{so}_{n+1}$ for some Lie algebra $\g{a}$, where $\g{h}$ does not share the ideal $\g{a}$ with $\g{a}\oplus\g{so}_n$. There are several possibilities for a fiber pair $\g{h}\subset\g{a}\oplus\g{so}_n$ with $b=1$. By inspecting  Table~\ref{TABLE:b1} we can list all of them, indicating by F$X$ that the fiber pair is of Type $X$. The possibilities with $\g{a}$ trivial are:
\begin{enumerate}[leftmargin=2.5cm,style=multiline]
\item[(F1)] $\g{so}_{n-1}\subset\g{so}_{n}\subset\g{so}_{n+1}$, with $n$ odd. This triple has $f>1$ if $n>3$ and $f= 1$ if $n=3$ by Lemmas~\ref{LEM:SOn_not_almost_fat} and \ref{LEM:SO4_almost_fat_} respectively.

\item[(F2, $n=1$)] $\{0\}\subset\g{so}_{3}\subset\g{so}_{4}$. This triple has $f=1$ by Lemma~\ref{LEM:SO4_almost_fat_}.

\item[(F4)] $\g{spin}_7^+\subset\g{spin}_9\subset\g{spin}_{10}$. Here $\dim \g{m}=15$ and $\dim \g{p}=9$ so $f>1$ by Lemma~\ref{LEM:f_1_dimensions}.

\item[(F5)] $\g{g}_2\subset\g{spin}_7\subset\g{spin}_{8}$. Here $b(\g{g}_2\subset\g{spin}_{8})=2$ by Lemma~\ref{LEM:G2_SO7_SO8}, hence $f=1$ by Lemmas~\ref{LEM:f_and_b} and \ref{LEM:f_1_dimensions}.

\item[(F11)] $\g{so}_3^{\max}\subset\g{so}_5\subset\g{so}_{6}$. Here $\dim\g{m}=7$ and $\dim\g{p}=5$ so $f>1$ by Lemma~\ref{LEM:f_1_dimensions}.

\item[(F3, F8, F9, $n=1$)] $\g{h}\subset \g{sp}_2 \cong\g{so}_{5}\subset\g{so}_{6}$ with $\g{h}$ equal to $\g{sp}_1$, $\g{sp}_1\oplus\g{u}_1$ or $\g{sp}_1\oplus\g{sp}_1$. In the first two cases we have $\dim\g{m}=7$ and $6$ respectively and $\dim\g{p}=5$, thus $f>1$ by Lemma~\ref{LEM:f_1_dimensions}. In the third case the triple is triple-isomorphic to $\g{so}_{4}\subset\g{so}_{5}\subset\g{so}_{6}$, which has $f>1$ by Lemma~\ref{LEM:SOn_not_almost_fat}.
\end{enumerate}
The possibilities with $\g{a}$ non-trivial are:
\begin{enumerate}[leftmargin=2.5cm,style=multiline]
\item[(F1, $n=3$)] $\g{so}_3 \subset \g{so}_4\cong\g{so}_3\oplus\g{so}_3\subset\g{so}_3\oplus\g{so}_4$. This triple has $f=1$ by Lemma~\ref{LEM:SO4_almost_fat_}.	

\item[(F2, $n=1$)] $\{0\}\oplus\g{so}_3\subset \g{su}_2\oplus\g{so}_3\cong\g{so}_3\oplus\g{so}_3\subset\g{so}_3\oplus\g{so}_4$. In this case $\g{m}$ and $\g{p}$ belong to different ideals of $\g{so}_3\oplus\g{so}_4$, so $[\g{m},\g{p}]=0$ and $f=3$.
	
\item[(F13)] $\g{u}_2 \subset \g{su}_3\oplus\g{so}_3 \subset \g{su}_3\oplus\g{so}_4$. Here $\dim \g{p}=3$ and $\dim \g{m}=7$ so $f>1$ by Lemma~\ref{LEM:f_1_dimensions}.

\item[(F14, $n=1$)] $\g{u}_{1} \subset \g{u}_2 \cong \g{u}_1\oplus\g{so}_3\subset\g{u}_1\oplus\g{so}_4$, where $\g{u}_{1} \subset \g{u}_2$ has non-trivial projection into the $\g{so}_3$-factor of $\g{u}_2$. This triple has $f=1$ by Lemma~\ref{LEM:SO4_almost_fat_}.

\item[(F15, $n=1$)] $\g{sp}_1\oplus\Delta\g{u}_1 \subset  \g{sp}_2\oplus\g{u}_1 \cong \g{u}_1\oplus\g{so}_5 \subset \g{u}_1\oplus\g{so}_6$. Here $\dim\g{m}=7$ and $\dim\g{p}=5$ so $f>1$ by Lemma~\ref{LEM:f_1_dimensions}.

\item[(F16, $n=1$)] $\g{sp}_1\oplus\Delta\g{sp}_1 \subset  \g{sp}_2\oplus\g{sp}_1 \cong\g{sp}_1\oplus\g{so}_5\subset \g{sp}_1\oplus\g{so}_6$. Here $\dim\g{m}=7$ and $\dim\g{p}=5$ so $f>1$ by Lemma~\ref{LEM:f_1_dimensions}.

\item[(F16)] $\g{sp}_{n-1}\oplus\Delta\g{sp}_1 \subset  \g{sp}_n\oplus\g{sp}_1\cong\g{sp}_n\oplus\g{so}_3 \subset\g{sp}_n \oplus\g{so}_4$, with $n\geq 2$. Here $\dim\g{m}=4n-1$ and $\dim\g{p}=3$ so $f>1$ for any $n\geq 2$ by Lemma~\ref{LEM:f_1_dimensions}.\qedhere
\end{enumerate}
\end{proof}

\section{The isometry groups}\label{SEC:isometry_group}

Let us consider the homogeneous spaces $\spin{8}/\gg\cong \sph^7\times \sph^7$ and $\spin{7}/\su{3}\cong \sph^6\times \sph^7$; we refer to Subsection~\ref{subsec:difeo_product_spheres} for a discussion of the corresponding diffeomorphisms. By  Theorem~\ref{THM:Wallach_generalization} and Theorem~\ref{THM:submersion_triples}, any (non-normal) homogeneous metric on $\spin{8}/\gg$ or $\spin{7}/\su{3}$ of the form $q_t$ with $t>0$ has $\Ric_2 > 0$ (see Section~\ref{SEC:submersion_metrics} for the precise definition of $q_t$). Also, by Lemma~\ref{LEM:G2_SO7_SO8}, any normal homogeneous metric on $\spin{8}/\gg$ is of $\Ric_2>0$. Since $\spin{8}$ is simple, the normal homogeneous metrics of $\spin{8}/\gg$ are unique up to scalar multiplication. We denote any such metric by $q_\infty$ since it is the limit of the corresponding submersion metric $q_t$ when $t\rightarrow\infty$.  The purpose of this section is to determine the identity component of the isometry group of the Riemannian manifolds  $(\spin{7}/\su{3},q_t)$, $(\spin{8}/\mathsf{G}_2,q_t)$ and $(\spin{8}/\mathsf{G}_2,q_{\infty})$. This is achieved in Theorems~\ref{thm:isometry_group} and \ref{thm:isometriesqt}.

For $(\spin{8}/\mathsf{G}_2,q_{\infty})$, Wang and Ziller \cite[p.~625]{WZ85} have shown that the identity component of the isometry group is isomorphic to $\spin{8}$. Thus, we are left with the cases of non-normal metrics, namely $(\spin{7}/\su{3},q_t)$ and $(\spin{8}/\mathsf{G}_2,q_t)$. Let $M$ and $G$ denote either $M=\spin{7}/\su{3}\cong \sph^6\times \sph^7$ and $G=\spin{7}$, or $M=\spin{8}/\gg\cong \sph^7\times\sph^7$ and $G=\spin{8}$. Since $q_t$ is $G$-invariant and $G$ acts transitively on $M$, the action of the (identity component of the) isometry group must extend the $G$-action and in particular must be transitive. With this in mind, our strategy to determine the identity component of the isometry group consists of three steps:

\begin{enumerate}
\item we find the Lie algebras of the connected Lie groups that act transitively and almost effectively (but not necessarily by isometries) on $M$,
\item we determine which of the groups in Item (1) act by isometries with respect to $q_t$, up to covering,
\item we derive which Lie groups in Item (2) act effectively on $M$.
\end{enumerate}

In order to implement step (1) we need some preliminaries. The problem of extending transitive actions was studied by Onishchik. Let us briefly summarize his so-called \emph{extension theory}. For more details, one can see \cite{On:transitive} or \cite[Chapter 5]{WZ85}.

A triple of compact Lie algebras $(\g{g}',\g{g},\g{k})$ is said to be a \textit{decomposition} if $\g{g}'=\g{g}+\g{k}$. Let $G'$ be the simply connected Lie group corresponding to $\g{g}'$, and let  $G$, $K$ be the connected Lie subgroups of $G'$ with Lie algebras $\g{g}$ and $\g{k}$, respectively. Then, $(\g{g}',\g{g},\g{k})$ is a decomposition if and only if $G$ acts transitively on $G'/K$. If $H$ denotes the isotropy group at the identity coset of the $G$-action on $G'/K$, then we have a canonical diffeomorphism $G/H\cong G'/K$.  Thus, via this diffeomorphism, $G'$ acts smoothly and transitively on $G/H$. (Note, however, that the $G'$-action on $G/H$ is not necessarily isometric with respect to all $G$-invariant metrics on $G/H$.) Given a decomposition $(\g{g}',\g{g},\g{k})$, the pair $(\g{g}',\g{k})$ is called an \textit{extension} of the pair $(\g{g},\g{h})$ if $\g{h}=\g{g}\cap\g{k}$. Moreover, an extension is \textit{effective} if $\g{g}'$ and $\g{k}$ have no ideals in common. 

Let us recall the definition of the different types of extensions.  
\begin{itemize}
	\item \textit{Type I extensions:} $(\g{g}',\g{k})$ is a type I extension of the pair $(\g{g},\g{h})$ if there is a subalgebra $\g{a}\neq0$ of $\g{g}$ such that $\g{h}\oplus\g{a}\subset\g{g}$, and $\g{g}'\simeq\g{g}\oplus\g{a}$, $\g{k}\simeq\g{h}\oplus\g{a}$, where the inclusion of $\g{k}$ in $\g{g}'$ restricted to $\g{a}$ is given by the diagonal embedding.  Observe that since $\g{h}\oplus \g{a}\subset \g{g}$, it follows that $\g{a}$ is contained in the centralizer of $\g{h}$.  In particular, if this type occurs then $Z_{\g{g}}(\g{h})$ is non-zero. Notice the abuse of notation made to denote by $\g{a}$ both a Lie subalgebra of $\g{g}$ and a copy of it that   gives rise to the extension of $\g{g}$ (as an ideal of $\g{g'}$).
	\item \textit{Type II extensions:} $(\g{g}',\g{k})$ is a type II extension of the pair $(\g{g},\g{h})$ if $\g{g}'$ is simple. All these extensions are listed in \cite[Table~7]{On:transitive}.
	\item \textit{Type III extensions:} Let $(\g{m}, \g{m}', \g{m}'')$ be a decomposition with $\g{m}$ simple. Let $\g{a}$ be a simple Lie algebra such that $\g{m}''\subsetneq \g{a}\subset \g{m}$. Define $\g{g}' = \g{m} \oplus \g{a}$, $\g{k} = \g{m}' \oplus \g{m}''$  with $\g{m}'\subset \g{m}$ and $\g{m}''\subset \g{a}$, $\g{g}=\Delta\g{a}$ and $\g{h}=\g{k}\cap\Delta\g{a}$. In this situation, we say that $(\g{g}',\g{k})$ is a type III extension of the pair $(\g{g},\g{h})$. In this case $G'/K$ is diffeomorphic to the product $M/M'\times A/M''$. In most cases $\g{a}=\g{m}$ except for the 5 examples  listed in \cite[p.~624]{WZ85}.
\end{itemize}
Onishchik proved that the combination of these extensions gives all the possible ones. More concretely, he proved the following:
\begin{theorem}[cf.~{\cite[Theorem 6.2]{On:transitive}}]
	\label{th:onishchik}
	Let $(\g{g}, \g{h})$ be an effective pair of compact Lie algebras with $\g{g}$ simple. Then any effective compact extension of $(\g{g}, \g{h})$ is either a type I extension,  a type II extension, a type III extension, or a type I extension of an extension of type II or type III.
\end{theorem}

We are now ready to prove the following  two theorems, which are the main results of this section.

\begin{theorem}\label{thm:isometry_group} 
	The connected component containing the identity of the isometry group 
	$(\spin{7}/\su{3},q_t)$, $t>0$, is isomorphic to  
	$\spin{7}$.
\end{theorem}
\begin{proof} 	
Let us consider $(M,q_t):=(\spin{7}/\su{3},q_t)$. We follow the strategy explained at the beginning of the present section. We deal with steps (1) and (2) simultaneously, that is, we shall search for all groups (up to covering) extending the $\spin{7}$-action on $M$ and for each such group we determine whether it acts by isometries. According to Theorem~\ref{th:onishchik}, we need to discuss extensions of type I, II and III.

We begin with type I extensions of~$(\g{spin}_7,\g{su}_3)$. Assume that the Lie algebra of the isometry group contains a subalgebra of the form $\g{spin}_7\oplus \g{a}$.  We will show that $\g{a}$ is trivial.  To that end, denote $G=\spin{7}$, $H=\su{3}$ and observe that any element of $\g{a}$ can be exponentiated to a map $\varphi\colon G/H\to G/H$ which is in the identity component of the isometry group of $(G/H,q_t)$. Since $\g{a}$ commutes with $\g{spin}_7$, we have that $\varphi$ commutes with any isometry $g\in G$ of $G/H$. Our goal is to show that $\varphi$ is trivial.

We first claim that there is an element $d$ in the identity component of $N_G(H)$ for which $\varphi$ has the form $\varphi(gH) = gdH$, for each $g\in G$.  To see this, observe that $\varphi(eH)\in G/H$, where $e\in G$ denotes the identity element, so we may write $\varphi(eH) = dH$ for some $d\in G$.  For any $g\in \spin{7}$, since $\varphi$ and left multiplication by $g$ commute, we find $\varphi(gH) = g\varphi(eH) = gdH$.  In other words, $\varphi$ is given by right multiplication by the element $d$. Notice that if $h\in H$, then $hdH = \varphi(hH) = \varphi(eH) = dH$.  Thus, we conclude that for every $h\in H$ there is an element $h'\in H$ with $hd = dh'$.  Thus, $d\in N_G(H)$.  Since $\varphi$ is in the identity component of the isometry group, we conclude that $d$ is in the identity component of $N_G(H)$.

We now consider conjugation by $d$, denoted by $C_d\colon G/H\rightarrow G/H$, $gH\mapsto dgd^{-1}H$.  This is a composition of isometries, so it is an isometry as well.  Conjugation fixes the identity coset, so the differential $\Ad(d)$ must act by isometries on $\g{m}\oplus\g{p}$, where $\g{m}=\g{g}_2\cap\g{su}_3^\perp$ and $\g{p}=\g{spin}_7\cap\g{g}_2^\perp$.

We will use that the metric $q_t$, which is constructed out of a fixed bi-invariant metric $\langle\cdot,\cdot\rangle$ on $G$, is characterized by the following rule (see e.g.\ \cite[Equation~(1), p.~220]{Kerin11}):
\[q_t(X,Y)=\frac{t}{t+1}\langle X_{\g{m}}, Y_{\g{m}}\rangle + \langle X_{\g{p}},Y_{\g{p}}\rangle, \quad \text{for all $X,Y\in\g{m}\oplus\g{p}$}.
\]
By defining the map $\psi\colon\g{m}\oplus \g{p}\to \g{m}\oplus \g{p}$ as $\psi(X_\g{m}+X_\g{p}) = \frac{t}{t+1}X_\g{m} + X_\g{p}$, we can write $q_t(X,Y) = \langle \psi(X),Y\rangle$.

We claim that $\Ad(d) \circ \psi = \psi \circ \Ad(d)$.  Indeed, for any $X,Y\in \g{m}\oplus \g{p}$,  we have 
$$q_t(X,Y) =  q_t( \Ad(d) X, \Ad(d) Y) = \langle \psi(\Ad(d) X), \Ad(d) Y\rangle = \langle \Ad(d^{-1}) \psi(\Ad(d) X), Y\rangle,$$
where in the first identity we have used that $\Ad(d)$ is an isometry. Since on the other hand $q_t(X,Y) = \langle \psi(X),Y\rangle$, we conclude that $\langle \psi(X),Y\rangle=\langle \Ad(d^{-1}) \psi(\Ad(d) X), Y\rangle$ for all $Y$. It follows that $\Ad(d^{-1})\psi(\Ad(d) X) = \psi(X)$ for all $X$, which clearly implies $\Ad(d) \circ \psi = \psi \circ \Ad(d)$.

Because $\Ad(d)$ and $\psi$ commute, for each eigenvalue $\lambda$ of $\psi$, $\Ad(d)$ must map the $\lambda$-eigenspace of $\psi$ to itself.  Since $\psi$ has two such eigenspaces, $\g{m}$ with $\lambda =\frac{t}{t+1}\neq 1$ and $\g{p}$ with $\lambda = 1$, $\Ad(d)$ must preserve both $\g{m}$ and $\g{p}$ individually.  Since $\Ad(d)$ preserves $\g{m}$ and $\g{h}$, $\Ad(d)$ preserves $\g{h}\oplus \g{m}$, which is the Lie algebra of $\gg < \spin{7}$.  Exponentiating, we find that $C_d$ preserves $\gg< \spin{7}$.  That is, $d\in N_{\spin{7}}(\gg)$. As above, we observe that $d$ must be in the identity component $N_{\spin{7}}^0(\gg)$ of $N_{\spin{7}}(\gg)$ the normalizer of $\gg$ in $\spin{7}$. But the fact that the isotropy representation of $\spin{7}/\gg$ is irreducible implies  $N_{\g{spin}_7}(\g{g}_2)=\g{g}_2$, and hence $d\in N^0_{\spin{7}}(\gg)=\gg$. As we have already seen that $d\in N_G(H)=N_{\spin{7}}(\su{3})$, and again by a continuity argument, it follows that $d$ lies in the identity component $N_{\gg}^0(\su{3})$ of the normalizer of $\su{3}$ in $\gg$. Since the isotropy representation of $\gg/\su{3}$ is irreducible, similarly as above we conclude that $d\in N_{\gg}^0(\su{3})=\su{3}=H$, so that $\varphi(gH) = gdH = gH$, for every $g\in G$. That is, $\varphi$ is the identity function.  Thus, we conclude that $\g{a} = 0$ as claimed.

\medskip

We continue with type II extensions. By inspection of \cite[Table~7]{On:transitive} the pair $(\g{spin}_{7},\g{su}_{3})$ only admits the following type~II extensions:
\begin{enumerate}
	\item[i)] $(\g{so}_{8},\g{so}_{6})$, where $\g{so}_6$ is embedded by the standard block embedding in $\g{so}_8$.
	\item[ii)] $(\g{so}_{8},\g{su}_{4})$, where $\g{su}_4$ is given by the chain of maximal inclusions  $\g{su}_4\subset\g{u}_4\subset\g{so}_8$.
\end{enumerate}
However, the subalgebras $\g{su}_4$ and $\g{so}_6$ of $\g{so}_8$ differ by an automorphism of $\g{so}_8$, i.e., $(\g{so}_8,\g{su}_4)$ and $(\g{so}_8,\g{so}_6)$ are pair-isomorphic in the sense of Definition~\ref{DEF:pair_triple_isomorphic}. This follows, for example, from \cite[Chapter X, Proposition 1.4, Theorem 6.1, and entry (viii) on p.~519] {Helgason}. Hence, the corresponding homogeneous spaces have the same invariant metrics. Thus, we will work with the pair $(\g{so}_8,\g{so}_6)$. We shall show that $\g{so}_{8}$ is not the Lie algebra of the isometry group of $(M,q_t)$ by proving that no $\so{8}$-invariant metric on  $\spin{8}/\spin{6} = \so{8}/\so{6}$ has $\Ric_2>0$.

Consider the pair $(\g{so}_8,\g{so}_6)$. Its isotropy representation splits as $\g{p} \cong \mathbb{R}^6\oplus \mathbb{R}^6 \oplus \mathbb{R}$, where $\mathbb{R}^6$ denotes the standard representation of $\g{so}_6$ and $\mathbb{R}$ denotes the trivial representation of $\g{so}_6$.  Because the $\mathbb{R}^6$ summand appears with multiplicity, the space of invariant metrics is $4$-dimensional, paramaterized by the choice of a metric on each factor (which is unique up to scaling), as well as the angle between the two $\mathbb{R}^6$ factors.  However, as shown by Kerr in \cite[Section~4, pp.~120-121]{Ke98}, every invariant metric is isometric to a \textit{diagonal} metric, namely one for which the two $\mathbb{R}^6$ factors are orthogonal.  Hence, for our purposes it is sufficient to show that no diagonal metric on $\so{8}/\so{6}$ has $\Ric_2>0$.  In fact, we will show the stronger statement that no diagonal metric has $\Ric_3>0$.  We will do this by showing that $\so{8}/\so{6}$ admits a $6$-dimensional totally geodesic submanifold $\Sigma$, arising as the fixed point set of an isometry $\sigma$, whose induced metric does not have $\Ric_3>0$. Then, the result follows from the fact that any $m$-dimensional totally geodesic submanifold of a manifold of $\Ric_k>0$, with $k<m$, must have $\Ric_k>0$.

To that end, let us consider $\sigma=\diag(-1,1,I_3,-I_3)\in \mathsf{SO}_8$, and assume that $\so{6}$ is given by the lower block embedding. Then $\sigma$ acts on $\so{8}/\so{6}$ via conjugation: $g\so{6}\mapsto \sigma g \sigma^{-1} \so{6}$.  This action is well-defined since $\sigma$ normalizes $\so{6}$, and one easily checks that it is an isometry with respect to any diagonal metric.

The base point $p=e\so{6}$, where $e=I_8$, is a fixed point of $\sigma$. Let $\Sigma$ denote the connected component of the fixed point set of the isometry $\sigma$ passing through $p$. Define $G_\sigma$ as the connected component of the identity of the subgroup $\{g\in\mathsf{SO}_8: g\sigma=\sigma g \}$.   By  \cite[Lemma~A.1]{Pu99}, we have $G_\sigma\cdot p=\Sigma$. In our case, one can  check that $G_\sigma$ is isomorphic to $\so{4}\times\so{4}$, and then the isotropy of $G_\sigma$ at $p$ turns out to be isomorphic to $\so{3} \times \so{3}$, where the inclusion $\so{3} \times \so{3} <\so{4}\times\so{4}$ is factorwise.
This implies that $\Sigma$ is the homogeneous space $(\so{4}\times\so{4})/(\so{3}\times\so{3})=(\so{4}/\so{3})\times(\so{4}/\so{3})$. This is a product of isotropy-irreducible homogeneous spaces, and hence, by Schur's lemma, $\Sigma$ is isometric to a product of two round spheres of dimension three. Thus, $\Sigma$ has $\Ric_4>0$ but not $\Ric_3>0$, see \cite[Remark~1.2]{DGM}.  Because $\Sigma$, as a connected component of the fixed point set of an isometry, is a totally geodesic submanifold of  $\so{8}/\so{6}$, no diagonal metric on $\so{8}/\so{6}$ can have $\Ric_3>0$.

\medskip

 Finally, we address type III extensions.  For type III extensions of~$(\g{spin}_7,\g{su}_3)$, the only possible extensions yield the homogeneous spaces $\spin{7}/\mathsf{G}_2\times\so{7}/\so{6}$ (case $\g{a}=\g{m}$), or $\spin{8}/\spin{7}\times\spin{7}/\spin{6}$ (exception (e) in \cite[p.~624]{WZ85}). These are products of isotropy-irreducible homogeneous spaces. Hence, any homogeneous metric is a product metric on $\mathbb{S}^6\times\mathbb{S}^7$, so it cannot have $\Ric_k>0$ for any $k\leq 7$, see~\cite[Proposition~3.5]{DGM}.

Summarizing, we have determined all groups (up to covering) which extend the $\spin{7}$-action on $M$ and arise as extensions of types I, II or III, and we have shown that none of them acts by isometries on $(M,q_t)$. Thus, no type I extension of a type II or III extension can yield an isometric extended action either. According to Theorem~\ref{th:onishchik}, it follows that the identity component of the isometry group of $q_t$ is finitely covered by $\spin{7}$. 
Furthermore, observe that the center of $\spin{7}$ has order $2$ while the center of $\su{3}$ has order $3$, see \cite[Table, p.~48]{On:topology}. Hence, these centers intersect only at the identity. In particular, $\spin{7}$ and $\su{3}$ share no non-trivial normal subgroup in common, which implies that the $\spin{7}$-action on $\spin{7}/\su{3}$ is effective.  Thus, the identity component of the isometry group of $q_t$ is isomorphic to $\spin{7}$.
\end{proof}

\begin{remark}\label{rem:spin7a_normal}
	Notice that the identity connected component of the isometry group of $\spin{7}/\su{3}$, equipped with the normal homogeneous metric (which corresponds to the limit metric $q_\infty$), is $\spin{7}\times\u{1}$, which is a type I extension of $\spin{7}$. This follows from~\cite[Corollary~1.3]{Reggiani}.
\end{remark}

\begin{theorem}\label{thm:isometriesqt}  
The identity component of the isometry groups of $(\spin{8}/\gg, q_t)$, $t>0$, and of $(\spin{8}/\gg, q_\infty)$ are isomorphic to $\spin{8}$.
\end{theorem}

\begin{proof}
By \cite[p.~625]{WZ85}, the connected component containing the identity of the Riemannian manifold $(\spin{8}/\gg,q_\infty)$ is isomorphic to $\spin{8}$.	
	
 To deal with $(\spin{8}/\gg, q_t)$, we will proceed as in the proof of Theorem \ref{thm:isometry_group}, by showing that $(\spin{8}/\gg, q_t)$ does not admit an extension of type I, II, nor III.

For type I, as mentioned above, such an extension can only occur if $Z_{\g{spin}_{8}}(\g{g}_2)$ is non-trivial.  But we claim that $Z_{\g{spin}_{8}}(\g{g}_2)= 0$.  To see this, simply observe that as a representation of $\g{g}_2$, we have a splitting of $\g{spin}_8$ into irreducible subrepresentations as $\g{spin}_8  = \g{g}_2\oplus \mathbb{R}^7\oplus \mathbb{R}^7$.  In particular, there is no subrepresentation where the action by $\g{g}_2$ is trivial.

For type II, there is no such extension listed in \cite[Table 7]{On:transitive}.

For type III, we have $\g{a} \cong \g{spin}_8$.  We note that from inspection of the 5 examples listed in \cite[p.~624]{WZ85}, there is no case where $\g{a} \subsetneq \g{m}$ and $\g{a}\cong \g{spin}_{8}$.  Thus, we only need to consider the case $\g{a} = \g{m}$.

If $\g{spin}_8 \cong \g{a} = \g{m}$, then we must have $\spin{8}/M'\times \spin{8}/M'' \cong \sph^7\times \sph^7$.  From here we see that $M'$ and $M''$ must both be isomorphic to $\spin{7}$.  Moreover, the isotropy actions of $M'\times M''$ on the tangent space to each factor of $\spin{8}/M'\times \spin{8}/M''$ are clearly inequivalent, so the only $(\spin{8}\times \spin{8})$-invariant metrics on $\spin{8}/M' \times \spin{8}/M''$ are product metrics.  In particular, such a metric cannot have $\Ric_k>0$ for any $k\leq 7$.

From Theorem \ref{th:onishchik}, it follows that the identity component of the isometry group of either $q_t$ or $q_\infty$ is finitely covered by $\spin{8}$.  Furthermore, from \cite[Table, p.~48]{On:topology}, $\gg$ has trivial center.  It follows that $\spin{8}$ and $\gg$ share no non-trivial normal subgroup in common, which implies that the $\spin{8}$-action on $\spin{8}/\gg$ is effective.  Thus, the identity component of the isometry group of $q_t$ and $q_\infty$ is isomorphic to $\spin{8}$.
\end{proof}

\section{Free biquotient actions}\label{SEC:free_actions}

Consider the homogeneous spaces $\spin{8}/\gg\cong \sph^7\times \sph^7$ and $\spin{7}/\su{3}\cong \sph^6\times \sph^7$. In this section (Subsections~\ref{subsec:rank} to~\ref{subsec:freeso3su2}) we determine all connected subgroups of $\spin{8}$ and $\spin{7}$ acting freely on the corresponding spaces. By Theorems~\ref{thm:isometry_group} and \ref{thm:isometriesqt} from Section~\ref{SEC:isometry_group}, this is equivalent to determining all connected groups acting freely and isometrically on $(\spin{7}/\su{3},q_t)$, $(\spin{8}/\mathsf{G}_2,q_t)$ and $(\spin{8}/\mathsf{G}_2,q_{\infty})$. Finally, in Subsections~\ref{subsec:S4} and~\ref{subsec:tg} we will study certain Riemannian submersions and totally geodesic embeddings that are naturally associated with our $\Ric_2>0$ metrics on $\sph^7\times \sph^7$, $\sph^6\times \sph^7$ and some of their quotients manifolds by the free isometric actions analyzed in the first part of this section.

We will now outline the contents of this section in more detail. 

Several of our results will be proved using the theory of biquotients. We recall that, by definition, a \textit{biquotient} is any manifold obtained as the quotient of a homogeneous space $G/H$ by a free action by a subgroup $L< G$.   Eschenburg systematically studied biquotients in his habilitation thesis~\cite{Es84}. Biquotients have an alternative description which is often useful.  Given any closed subgroup $U< G\times G$ where $G$ is a compact Lie group, one obtains an action by $U$ on $G$ given by $(u_1,u_2)\ast g = u_1 g u_2^{-1}$.  When this action is free, the quotient space is denoted $G/\!\!/U$.  In the special case where $U = L\times H$ with $L,H< G$, the quotient space is also sometimes denoted $L\backslash G/H$.  In this context, one can view $L\backslash G/H$ as the quotient by an $L$-action on $G/H$, or as the quotient by an $H$-action on $L\backslash G$.

The main result of this section is that the only connected subgroups $L<\spin{8}$ (resp. $L<\spin{7}$) acting freely on $\spin{8}/\gg$ (resp.\ $\spin{7}/\su{3}$) are isomorphic to $\sg^1$ and $\su{2}$; moreover, up to equivalence of actions, there are infinitely many free circle actions and precisely one free $\su{2}$-action on $\spin{8}/\gg$, and there is a unique free circle action and a unique free $\su{2}$-action on $\spin{7}/\su{3}$. The precise statements are given in Theorems~\ref{thm:free} and~\ref{thm:summary}. We recall that the actions of $G$ and $G'$ on $M$ and $M'$, respectively, are said to be \emph{equivalent} if there is an isomorphism  $\phi\colon G\to G'$ and a diffeomorphism $f\colon M\rightarrow M'$ with the property that $f( g * m) = \phi(g)\ast f(m)$ for any $g\in G$ and $m\in M$.

Let us sketch the proof of the determination of the free actions on $\spin{8}/\gg$ and $\spin{7}/\su{3}$. As a preliminary step we show that any connected Lie subgroup $L<\spin{8}$ (resp. $L<\spin{7}$) acting freely on $\spin{8}/\gg$ (resp. $\spin{7}/\su{3}$) must have rank $1$, see Lemma~\ref{lem:rank_restriction}.

Then, we will first consider the case $L  \cong \sg^{1}$ and later determine which of these $\sg^{1}$-actions extend to $\so{3}$ or $\su{2}$. To that end, notice that if we conjugate $L$ within either $\spin{8}$ or $\spin{7}$, the resulting actions are equivalent.  Thus, it is sufficient to consider $L\cong\sg^1$ in a fixed maximal torus.
To actually classify all such freely acting $\sg^1$ we use Kerr's  homogeneous description of $\sph^7\times \sph^7$, see \cite[Section~5]{Ke98}. In more detail, there are explicit embeddings $\spin{7}< \spin{8} < \so{8}\times \so{8}$ such that the restriction of the natural linear action of $\so{8}\times \so{8}$ on $\sph^7\times \sph^7$ to $\spin{8}$ is transitive with point stabilizer isomorphic to $\gg$. In addition, further restricting this action to $\spin{7}$, Kerr identifies an orbit which is diffeomorphic to $\sph^6\times \sph^7$.

These identifications allow us to translate left multiplication by $\spin{8}$ on $\spin{8}/\gg$ to the linear action of $\spin{8} < \so{8}\times \so{8}$ on $\sph^7\times \sph^7$. Similarly, we translate left multiplication by $\spin{7}$ on $\spin{7}/\su{3}$ to the linear action of $\spin{7}< \spin{8} < \so{8}\times \so{8}$ on $\sph^6\times \sph^7$. Then, in Proposition \ref{prop:maxtorus}, we determine a description of a maximal torus $\T^4$ (resp.\ $\T^3$) of $\spin{8}$ (resp.\ $\spin{7}$). Since the actions under consideration are linear, it is straightforward to determine when a circle subgroup $\sg^1<\T^4$ (resp. $\sg^1<\T^3$) acts freely; we do this in Theorem \ref{thm:free}.

We then determine which $\sg^1$-actions extend to either an $\so{3}$-action or an $\su{2}$-action. For it we use the well-known representation theory of $\su{2}$ to first classify all possible $8$-dimensional real representations of $\su{2}$.  Then, for each representation, we consider the action by the maximal torus, and determine if this action is free using Theorem~\ref{thm:free}.  

We will conclude this section by discussing certain Riemannian submersions and totally geodesic submanifolds related to our examples. In particular, in Subsection~\ref{subsec:S4} we will prove that $\sph^7\times \sph^7$, $\sph^6\times \sph^7$ and some of the quotient manifolds associated with the previously studied free isometric actions admit Riemannian submersions onto a round $4$-sphere. Finally, in Subsection~\ref{subsec:tg} we will show that the previously known simply connected manifolds admitting $\Ric_2>0$ metrics (namely, $\sph^k\times\sph^\ell$, with $k,\ell\in\{2,3\}$) arise as totally geodesic submanifolds of our higher dimensional spaces $\sph^7\times \sph^7$, $\sph^6\times \sph^7$ and some of their quotients studied in this section.

\subsection{Rank restrictions} \label{subsec:rank}

Here we prove the following lemma.

\begin{lemma}\label{lem:rank_restriction}
Suppose $L<\spin{8}$ or $L<\spin{7}$ acts freely on $\spin{8}/\gg$ or $\spin{7}/\su{3}$, respectively. Then $\rank L\leq 1$.
\end{lemma}
\begin{proof}

Recall that the action by left multiplication of a closed subgroup $L<G$ on a homogeneous space $G/H$ is free if and only if the corresponding $L\times H$ biquotient action is free. In this situation it is well known that $\rank L\times H\leq\rank G$. In the case of $\spin{7}/\su{3}$ we conclude that $\rank L\leq 1$.

In the case of $\spin{8}/\gg$ the same argument implies that $\rank L\leq 2$. However, it cannot have rank $2$.  For if it does, we may restrict the free action of $L\times \gg$ on $\spin{8}$ to a maximal torus of $L\times \gg$ to obtain a free biquotient action by $\T^4\cong \T^2\times \T^2$ on $\spin{8}$, where $\T^2\times \{e\}$ acts by left multiplication and $\{e\}\times \T^2$ acts by right multiplication by inverses. However, there are no such free $\T^4$-actions on $\spin{8}$, as it follows from the classification by Eschenburg of free biquotient actions by a full rank torus on simple Lie groups (see~\cite[Satz 75, p.~119]{Es84} or \cite[p.~8]{Zi:Eschenburg}).
\end{proof}

\subsection{\texorpdfstring{Identifying $\spin{8}/\gg$ with $\sph^7\times \sph^7$ and $\spin{7}/\su{3}$ with $\sph^6\times \sph^7$}{Identifying Spin8/G2 with S7xS7 and Spin7/SU3 with S6xS7}}\label{subsec:difeo_product_spheres}

All the claims in this subsection can be found in \cite{DS} together with their proofs; see also \cite{Ke98} for a similar exposition. We let $\mathbb{O}$ denote the octonions, defined as the algebra obtained from the quaternions upon applying the Cayley-Dickson process. We endow $\mathbb{O}$ with the standard inner product and consider the orthonormal basis $\{1,i,j,k,\ell, i\ell, j\ell, k\ell\}$, where the multiplication rules are given in Table \ref{table:cayleymult} in the order (row)(column).

\begin{table}[ht]

\begin{center}
	\renewcommand{\arraystretch}{1.0}

\begin{tabular}{cccccccc}
 
\multicolumn{1}{c}{} & $\bm{i}$ & $\bm{j}$ & $\bm{k}$ & $\bm{\ell}$ & $\bm{i\ell}$ & $\bm{j\ell}$ & $\bm{k\ell}$ \rule{0pt}{4ex} \\

\rule{0pt}{4ex} 

$\bm{i}$ & $-1$ & $k$ & $-j$ & $i\ell$ & $-\ell$ & $-k\ell$ & $j\ell$  \\ %\cline{2-8}
\rule{0pt}{4ex}

$\bm{j}$ & $-k$ & $-1$ & $i$ & $j\ell$ & $k\ell$ & $-\ell$ & $-i\ell$ \\ %\cline{2-8}
\rule{0pt}{4ex}

$\bm{k}$ & $j$ & $-i$ & $-1$ & $k\ell$ & $-j\ell$ & $i\ell$ & $-\ell$\\ %\cline{2-8}
\rule{0pt}{4ex}

$\bm{\ell}$ & $-i\ell$ & $-j\ell$ & $-k\ell$ & $-1$ & $i$ & $j$ & $k$\\ %\cline{2-8}
\rule{0pt}{4ex}

$\bm{i\ell}$ & $\ell$ & $-k\ell$ & $j\ell$ & $-i$ & $-1$ & $-k$ & $j$\\ %\cline{2-8}
\rule{0pt}{4ex}

$\bm{j\ell}$ & $k\ell$ & $\ell$ & $-i\ell$ & $-j$ & $k$ & $-1$ & $-i$\\ %\cline{2-8}
\rule{0pt}{4ex}

$\bm{k\ell}$ & $-j\ell$ & $i\ell$ & $\ell$ & $-k$ & $-j$ & $i$ & $-1$\\ 
\rule{0pt}{0ex}
\vspace{-1ex}
\end{tabular}
\caption{Multiplication table for Cayley numbers}\label{table:cayleymult}

\end{center}

\end{table}

 The octonions are well-known to be a non-associative normed division algebra.  In addition, the octonions satisfy several weak forms of associativity: they are alternative (meaning any subalgebra generated by two elements is associative), and satisfy the following Moufang identities for all $a,b,c\in \mathbb{O}$:

\begin{enumerate}[label = (\roman*)]
\item $(ab) (ca) = a (bc)a$,

\item  $((a b) a ) c = a (b (a c))$,

\item $ ((a b) c) b = a (b (c b))$.

\end{enumerate}

We identify the underlying additive structure of the octonions with $\mathbb{R}^8$ in the usual way. 
We also identify $\sph^7\times \sph^7$ as the pairs of unit vectors in $\mathbb{O}^2$, and we identify $\sph^6\times \sph^7$ as $\{(x,y)\in \sph^7\times \sph^7: \mathrm{Re}(x) = 0\}$.
For $u\in \sph^7$, we define the left multiplication map $L_u\colon \mathbb{O}\rightarrow \mathbb{O}$ by $L_u(x) = u x$.  We similarly define $R_u$ by $R_u(x) = xu$.  Both $L_u$ and $R_u$ are isometries of $\mathbb O$.  We caution that, due to non-associativity of the octonions, in general $L_{u_1}\circ L_{u_2} \neq L_{u_1 u_2}$.

Given matrices $A,B,C\in \so{8}$, we consider the triality equation \begin{equation}\label{eqn:triality}A(x) B(y) = C(x y) \text{ for all } x,y\in \mathbb{O}.
\end{equation}

\begin{lemma}\label{lemma:moufang}
	Let $u\in \sph^6\subset \operatorname{Im}\mathbb{O}$. Then, the triples $(A,B,C) = (-L_u, R_u, L_u\circ R_{\overline{u}})$ and $(A,B,C) = (L_u\circ R_{\overline{u}}, L_u, L_u)$ satisfy the triality relation from Equation~\eqref{eqn:triality}. 
\end{lemma}
\begin{proof}
	Taking into account that for $u\in \sph^6\subset \operatorname{Im}\mathbb{O}$ we have $\bar{u}=-u$, the result follows from the first and second Moufang identities, respectively; see also \cite[Lemma~5]{DS}.
\end{proof}

\begin{theorem}[{\cite[p.~152, Corollary]{DS}}]  Given any $A\in \so{8}$, there are matrices $B,C\in \so{8}$ for which Equation~\eqref{eqn:triality} holds.  The pair $(B,C)$ is almost unique, with $(-B,-C)$ being the only other pair for which Equation~\eqref{eqn:triality} holds.  Analogous statements hold given $B$ or given $C$.
\end{theorem}

We now let $G< \so{8}\times \so{8}$ be defined by
\begin{equation}\label{eq:definition_spin8}
G = \{(A,B)\in \so{8}\times \so{8}: \exists C\in \so{8} \text{ satisfying } \eqref{eqn:triality}\}.
\end{equation}
Then $G$ is isomorphic to $\spin{8}$ \cite[p.~151, Theorem]{DS}.  Any of the three projections $(A,B)\rightarrow A$, $(A,B)\rightarrow B$, and $(A,B)\rightarrow C$ are double covers. We define the following subgroups of $G$
$$
K=\{(A,B)\in G : B=C \},\qquad H=\{(A,B)\in G : A=B=C \}< K.
$$
It is known that $K$ is isomorphic to $\spin{7}$ \cite[Lemma 9]{DS} (observe that Kerr denotes it by $\spin{7}^+$ in \cite[Section~5]{Ke98}) and that $H$ is isomorphic to $\gg$ \cite[p.~151]{DS}. For the rest of the present section the notation $H< K< G$ will be used for these groups.

Notice that as $G$ is a subgroup of $\so{8}\times \so{8}$, it naturally acts on $\sph^7\times \sph^7$. Kerr proves the following in \cite[Sections~5~and~6]{Ke98}.

\begin{proposition}\label{prop:diffeo}  This $G$-action on $\sph^7\times \sph^7$ is transitive with stabilizer isomorphic to $H\cong\gg$.  The $K$-orbit through the point $(i,1)$ is $\{(x,y)\in \sph^7\times \sph^7: \mathrm{Re}(x)=0\},$ which is diffeomorphic to $\sph^6\times \sph^7$.  It has stabilizer isomorphic to $\su{3}< H$.
\end{proposition}

\begin{remark}
For the rest of Section~\ref{SEC:free_actions}, the notation $\spin{8}$ will always refer to the group $G<\so{8}\times \so{8}$ defined in Equation~\eqref{eq:definition_spin8}, and similarly $\spin{7}$ will always refer to the group $K<G\cong\spin{8}$.
\end{remark}

\subsection{\texorpdfstring{Identifying a maximal torus of $\spin{8}$ and $\spin{7}$}{Identifying a maximal torus of Spin8 and Spin7}}

Our next goal is to study circle actions on $\spin{8}/\gg$ and $\spin{7}/\su{3}$. Notice that any $\sg^1$-subgroup of $\spin{8}< \so{8}\times \so{8}$ or $\spin{7}<\spin{8}$ is conjugate to a circle in a given maximal torus.  As conjugate subgroups yield equivalent actions, we need only determine the induced action by circles in a fixed maximal torus.  To that end, we now determine such a maximal torus.

Let $R(\theta)\in\so{2}$ be the standard rotation matrix
$$
R(\theta) =\left(\!\!\begin{array}{r@{\;}r} \cos(\theta)  & -\sin(\theta)\\ \sin(\theta) & \cos(\theta)\end{array}\!\!\right).
$$

Then the standard maximal torus of $\so{8}$ is given by the subset of block diagonal matrices $\diag(R(\theta_1), R(\theta_2), R(\theta_3), R(\theta_4))$, with $\theta_r\in\R$, $r=1,2,3,4$.  We will determine the lift via the double covering $\pi\colon \spin{8}\rightarrow \so{8}$ of such torus to $\spin{8}< \so{8}\times \so{8}$ which lifts the identity of $\so{8}$ to the identity of $\spin{8}$. Recall that $\pi\colon \spin{8}\rightarrow \so{8}$ is given by $\pi(A,B) = C$, where $(A,B,C)$ satisfies the triality Equation~\eqref{eqn:triality}.

In order to calculate the lift of the matrices $C=\diag(R(\theta_1), R(\theta_2), R(\theta_3), R(\theta_4))$ in a standard maximal torus $\mathsf{T}^4$ of $\so{8}$,  we will proceed separately by lifting those matrices with only one non-zero $\theta_r$, first with $r\in\{2,3,4\}$ and then with $r=1$. Since $\pi$ is a Lie group homomorphism, the composition of the corresponding lifts will yield the lift of a general element $C=\diag(R(\theta_1), R(\theta_2), R(\theta_3), R(\theta_4))$. To fix notation, we make the following convention. When we refer to a rotation of the $\{x,y\}$-plane, where $\{x,y\}$ is an ordered orthonormal set, we mean the linear transformation of $\mathbb{R}^8\cong \mathbb{O}$ which acts as the identity on the space $\operatorname{span}\{x,y\}^\bot$ and which rotates $x$ towards $y$ and $y$ towards $-x$.  With this convention in place, let us make the following observation:

\begin{lemma}\label{LEM:preimage_rotation} 
Let $u\in \operatorname{span}\{1,i\}^\bot\subset \mathbb{O}$ of norm $1$ and let $C$ be a rotation of angle $2\theta$ in the $\{u,iu\}$-plane. Then, the two preimages of $C$ under $\pi$ are
$$\pm (A,B):= \pm ( L_{ue^{-i\theta}}\circ L_u, R_{ue^{-i\theta}} \circ R_u) \in \spin{8}.$$
\end{lemma}

\begin{proof}
We will view such a rotation as a composition of two reflections.  It is easy to see that $L_u \circ R_{\overline{u}}$ fixes $u$ and $1$, but acts as $-1$ on the subspace $\operatorname{span}\{1,u\}^\perp$. Indeed, it is well known that $ u v  = - vu$ if $v \in \operatorname{span}\{1,u\}^\bot$, see e.g. \cite[Fact 6) in page 186]{GWZ}, so $L_u(R_{\overline{u}} v) = uv\overline{u} = -v|u|^2 = -v$.

Then, the composition 
\begin{equation}\label{eq:C}
	C := (L_{ue^{-i\theta}}\circ R_{\overline{ue^{-i\theta}}})\circ (L_u \circ R_{\overline{u}})
\end{equation} 
fixes every vector in $\spann\{u,iu\}^\perp$ for every $\theta\in[0,2\pi]$. Moreover, since
$(L_u\circ R_{\overline{u}})(u) = u$, we have 
$$(L_{ue^{-i\theta}}\circ  R_{\overline{ue^{-i\theta}}} \circ L_u \circ R_{\overline{u}})(u) = (ue^{-i\theta}) (u (e^{i\theta} \overline{u}))= (e^{i\theta} u^2)( e^{i\theta} (-u)) = e^{2i\theta} u, $$
where we have used the alternativity of the octonions, $u^2=-1$, and $ue^{\pm i\theta} = e^{\mp i\theta} u$ since $u\in \operatorname{span}\{1,i\}^\bot$. Thus, $C$ describes a rotation of the $\{u, iu\}$-plane as in the statement of the lemma.

Moreover, by Lemma~\ref{lemma:moufang}, $(A,B,C) = (-L_z, R_z, L_z\circ R_{\overline{z}})$ satisfies Equation~\eqref{eqn:triality} for any unit length purely imaginary $z\in \mathbb{O}$. Thus, we find that 
\begin{equation}\label{eq:AB}
	\pm (A,B):= \pm ( (-L_{ue^{-i\theta}}) \circ(-L_u), R_{ue^{-i\theta}} \circ R_u) \in \spin{8}
\end{equation}
and that $\pi(\pm(A,B)) = C$, for any unit length $u\in \operatorname{span}\{1,i\}^\bot\subset \mathbb{O}$. Since $(-L_{ue^{-i\theta}}) \circ(-L_u)=L_{ue^{-i\theta}}\circ L_u$, we get the expressions in the statement.
\end{proof} 

This observation allows us to prove the following:

\begin{proposition}\label{prop:theta234}
The two preimages of $\diag(R(0), R(2\theta_2), R(0),R(0))$ under $\pi$ are:
	$$\pm\bigl(\diag(R(\theta_2), R(\theta_2),R(-\theta_2),R(\theta_2)), \diag(R(-\theta_2),R(\theta_2), R(-\theta_2),R(\theta_2))\bigr).$$

The two preimages of $\diag(R(0), R(0), R(2\theta_3), R(0))$ under $\pi$ are: $$\pm\bigl(\diag(R(\theta_3), R(-\theta_3), R(\theta_3), R(\theta_3)), \diag(R(-\theta_3), R(-\theta_3), R(\theta_3), R(\theta_3))\bigr).$$

The two preimages of $\diag(R(0),R(0),R(0),R(2\theta_4))$ under $\pi$ are: $$\pm\bigl(\diag(R(-\theta_4), R(\theta_4), R(\theta_4), R(\theta_4)), \diag(R(\theta_4), R(\theta_4), R(\theta_4), R(\theta_4))\bigr).$$
\end{proposition}

\begin{proof}
Let us start with the case of $\theta_3\neq 0$, and set $C=\diag(R(0), R(0), R(2\theta_3), R(0))$. Clearly, $C$ is a rotation of angle $2\theta_3$ in the $\{u,iu\}$-plane for $u=\ell$. Hence the preimages $\pm (A,B)$ of $C$ by $\pi$ are then given by Lemma~\ref{LEM:preimage_rotation}. Let us compute the matrix expression of both $A$ and $B$ in this case.

For simplicity, we will use $\theta$ instead of $\theta_3$. In this case, $\ell e^{-i\theta} = \cos(\theta)\ell +\sin(\theta) i\ell$.  Then, using Table \ref{table:cayleymult} and the fact that imaginary quaternions anti-commute with $\ell$, we can compute the image of each element of the canonical basis of $\mathbb{O}$ under the transformation $A=L_{\ell e^{-i\theta}}\circ L_\ell$ in Equation~\eqref{eq:AB}:
	\begin{alignat*}{3} A(1) &= (\cos(\theta)\ell +\sin(\theta) i\ell)(\ell 1) &&=  -\cos(\theta)  - \sin(\theta)i\\
			A(i) &= (\cos(\theta)\ell +\sin(\theta) i\ell)(\ell i) &&= \phantom{-}\sin(\theta) - \cos(\theta)i\\
			A(j) &=(\cos(\theta)\ell +\sin(\theta) i\ell)(\ell j) &&= -\cos(\theta)j +\sin(\theta)k\\
			A(k) &=(\cos(\theta)\ell + \sin(\theta) i\ell)(\ell k) &&= - \sin(\theta)j - \cos(\theta)k\\
			A(\ell) &= (\cos(\theta)\ell +\sin(\theta) i\ell)(\ell^2) &&= -\cos(\theta)\ell - \sin(\theta)i\ell\\
			A(i\ell) &= (\cos(\theta)\ell +\sin(\theta) i\ell)(\ell(i\ell)) &&= \phantom{-}\sin(\theta)\ell - \cos(\theta) i\ell\\
			A(j\ell) &= (\cos(\theta)\ell +\sin(\theta) i\ell)(\ell(j\ell)) &&= -\cos(\theta) j\ell -\sin(\theta) k\ell\\
			A(k\ell) &= (\cos(\theta)\ell +\sin(\theta) i\ell)(\ell(k\ell)) &&= \phantom{-} \sin(\theta)j\ell -\cos(\theta)k\ell.\end{alignat*}
		
		Thus, $A = -\diag(R(\theta), R(-\theta), R(\theta), R(\theta))$.   In a similar fashion, one computes $B$ using Equation~\eqref{eq:AB}, obtaining $B =R_{\ell e^{-i\theta}}\circ R_\ell= -\diag(R(-\theta), R(-\theta), R(\theta), R(\theta))$.  So, the  two lifts of $C = \diag(R(0), R(0), R(2\theta), R(0))$ are $\pm(A,B)$, as stated in Item~(b) of the statement.

One can complete the proof by repeating this procedure for $u = j$ and $u =j \ell$. We warn the reader that, for $u=j\ell$, a rotation of $2\theta_4$ with respect to the basis $\{u,iu\}=\{j\ell, -k\ell\}$ is a rotation $R(-2\theta_4)$ with respect to the two last elements of the basis of $\mathbb{O}$, i.e., $\{j\ell, k\ell\}$.
\end{proof}

In view of Proposition~\ref{prop:theta234}, it remains to determine the lift of $\diag(R(2\theta_1), R(0), R(0), R(0))$.

\begin{proposition}\label{prop:theta1}
The preimages of $\diag(R(2\theta_1), R(0), R(0), R(0))$ under $\pi$ are
\[
\pm\bigl(\diag(R(-\theta_1) , R(-\theta_1) , R(-\theta_1) ,R(\theta_1)),\;\diag( R(-\theta_1),  R(\theta_1), R(\theta_1), R(-\theta_1))\bigr).
\]
\end{proposition}
\begin{proof}
On the one hand, we note that 
\begin{equation}\label{eq:conj_diag}
	\begin{split} \diag(R(-2\theta_1), R(0), R(0), R(0)) &= L_{\overline{j}} \,\diag(R(0), R(2\theta_1), R(0), R(0)) L_j
	\\
	&=L_{\bar{j}}\circ (L_{je^{-i\theta_1}}\circ R_{\overline{je^{-i\theta_1}}})\circ (L_j \circ R_{\overline{j}})\circ L_j.
	\end{split}
\end{equation}
Indeed, the first equality can be verified by applying the different elements of the canonical basis of $\mathbb{O}$, whereas the second one follows from the fact that the transformation $\diag(R(0), R(2\theta_1), R(0), R(0))$, which applies a rotation by $2\theta_1$ in the plane $\spann\{j,k\}$ and leaves $\spann\{j,k\}^\perp$ fixed, coincides with Equation~\eqref{eq:C} for $u=j$. 

On the other hand, by Lemma~\ref{lemma:moufang}, the triples $(-L_u, R_u, L_u\circ R_{\overline{u}})$, with $u=je^{-i\theta_1}$ and $u=j$, and $(L_u \circ R_{\overline{u}}, L_u, L_u)$, with $u=\overline{j}=-j$ and $u=j$, satisfy Equation~\eqref{eqn:triality}. By composing them appropriately, we see that
\begin{equation}\label{eq:AB_long}
(A,B) = \bigl((L_{\overline{j}} \circ R_j) \circ L_{je^{-i\theta_1}} \circ L_j \circ (L_j \circ R_{\overline{j}}),\; L_{\overline{j}} \circ R_{je^{-i\theta_1}} \circ R_j \circ L_j\bigr) \in \spin{8}
\end{equation}
and that 
\[
\pi(A,B) = L_{\bar{j}}\circ (L_{je^{-i\theta_1}}\circ R_{\overline{je^{-i\theta_1}}})\circ (L_j \circ R_{\overline{j}})\circ L_j=\diag(R(-2\theta_1), R(0), R(0), R(0)),
\]
where the second equality follows from Equation~\eqref{eq:conj_diag}. By applying both components in Equation~\eqref{eq:AB_long} to the canonical basis of $\mathbb{O}$, we get that 
\[
(A,B)=-\bigl(\diag(R(-\theta_1) , R(-\theta_1) , R(-\theta_1) ,R(\theta_1)),\;\diag( R(-\theta_1),  R(\theta_1), R(\theta_1), R(-\theta_1))\bigr),
\]
from where the proposition follows.
\end{proof}

By combining Propositions~\ref{prop:theta234} and~\ref{prop:theta1} to obtain the preimages under $\pi\colon \spin{8}\to\so{8}$ of an arbitrary element $C=\diag(R(2\theta_1), R(2\theta_2), R(2\theta_3), R(2\theta_4))$ in a standard maximal torus $\mathsf{T}^4$ of $\so{8}$, and by considering the particular lift of $C$ for which the identity $I_8\in \so{8}$ lifts to the identity in $\spin{8}$, we can summarize the results of this subsection as follows.

\begin{proposition}\label{prop:maxtorus}
A maximal torus in $\spin{8}< \so{8}\times \so{8}$ consists of all $(A,B)\in \spin{8}$ such that $A$ has the form
$$\diag( R(-\theta_1 + \theta_2 + \theta_3 - \theta_4) , R(-\theta_1+\theta_2 - \theta_3+\theta_4 ) , R(-\theta_1-\theta_2 + \theta_3+\theta_4) , R(\theta_1+\theta_2+\theta_3+\theta_4)),$$
and $B$ has the form
$$\diag(R(-\theta_1-\theta_2 - \theta_3 + \theta_4) , R(\theta_1+\theta_2 - \theta_3 + \theta_4) , R(\theta_1 -\theta_2 + \theta_3 + \theta_4) , R(-\theta_1 + \theta_2 + \theta_3 + \theta_4)),$$
where $\theta_r\in[0,2\pi)$ for each $r\in\{1,2,3,4\}$. A  maximal torus in $\spin{7}$ is the sub-torus for which $-\theta_1 + \theta_2 +\theta_3 - \theta_4 = 0$.

\end{proposition}

\begin{proof}  We have already proven the statement for $\spin{8}$.  For $\spin{7}$, we note that the sub-torus $\T^3$ with $-\theta_1 +\theta_2 + \theta_3-\theta_4 =0$ has rank $3$, so it is sufficient to show that $\T^3$ is a subgroup of $\spin{7}$.
To that end, suppose that $(A,B)\in \T^3$.  Notice that $A(1) = 1$.  Then, using Equation~\eqref{eqn:triality} with $x=1$ and $y$ arbitrary, we find that $$B(y) = A(1) B(y) = C(1 y) = C(y),$$ so that $B= C$.  That is, $(A,B)\in\spin{7}$.
\end{proof}

Using Proposition \ref{prop:diffeo}, one easily sees that the induced $\T^4$-action on $\sph^7\times \sph^7$ is given by left multiplication  by the matrices $(A,B)$ of Proposition \ref{prop:maxtorus}, and similarly for the $\T^3$-action on $\sph^6\times \sph^7$. Unfortunately, by regarding $\T^4$ as the abstract torus $(\sg^1)^4$ with coordinates $\theta_r\in[0,2\pi)$, the $\T^4$-action on $\sph^7\times\sph^7$ is not effective; for example, the element $(\theta_1,\theta_2,\theta_3,\theta_4) = (\pi,\pi,\pi,\pi)$ acts trivially.   To solve this issue, consider the parameters $\alpha_i$ defined by:  
\[
\alpha_1 =-\theta_1 + \theta_2 + \theta_3 - \theta_4, \;
\  \alpha_2 = \theta_1 + \theta_2 + \theta_3 + \theta_4,\;
 \alpha_3 = -\theta_1-\theta_2-\theta_3+\theta_4,\;
  \alpha_4 = \theta_1 + \theta_2 -\theta_3 + \theta_4.
\] 
One easily sees that the pair $(A,B)$ of Proposition \ref{prop:maxtorus} now takes the form 
\begin{equation}\label{EQ:maximal_torus}
\begin{aligned}
A &= \diag(R(\alpha_1), R(\alpha_1 + \alpha_3 + \alpha_4), R(\alpha_2 + \alpha_3 - \alpha_4), R(\alpha_2)),\\
B &= \diag( R(\alpha_3), R(\alpha_4), R(-\alpha_1 + \alpha_2 - \alpha_4), R(\alpha_1 + \alpha_2 + \alpha_3)),
\end{aligned}
\end{equation}
with $\alpha_r\in[0,2\pi)$ for each $r\in\{1,2,3,4\}$. From now on, we consider the torus $\T^4=(\sg^1)^4$ with these coordinates, which can be regarded as a maximal torus subgroup of $\spin{8}< \so{8}\times \so{8}$ by the identification $(\alpha_1,\alpha_2,\alpha_3,\alpha_4)\in (\sg^1)^4\mapsto (A,B)\in\spin{8}$, with $A$ and $B$ as before. The induced action is clearly effective.

\subsection{Free circle actions}\label{SS:circle_actions}

All subgroups $\sg^1 <\T^4$ are obtained by setting $\alpha_i = n_i \theta$ for some $n_i\in \mathbb{Z}$ with $\gcd(n_1,n_2,n_3,n_4) = 1$, and where $\theta\in [0,2\pi)$ is the parameter of $\sg^1$.   We define $\ell_i$ and $r_i  \in \mathbb{Z}$ by the resulting coefficient of $\theta$ for the $\sg^1$-action on the two $\sph^7$-factors.  That is, 
\begin{equation}\label{EQ:l_i_r_i}
\begin{aligned} \ell_1 &= n_1, &&& r_1 &= n_3,\\ \ell_2 &= n_1 + n_3 + n_4, &&& r_2 &= n_4,\\ \ell_3 &= n_2+n_3-n_4,  &&& r_3 &= -n_1 + n_2 - n_4,\\ \ell_4 &= n_2, &&&r_4 &= n_1 + n_2 + n_3.
\end{aligned}
\end{equation}
In other words, the pair $(A,B)$ of Proposition \ref{prop:maxtorus} is just given by
\begin{equation}\label{EQ:maximal_torus_ell_i_r_i}
\begin{aligned}
A &= \diag(R(\ell_1\theta), R(\ell_2\theta), R(\ell_3\theta), R(\ell_4\theta)),\\
B &= \diag( R(r_1\theta), R(r_2\theta), R(r_3\theta), R(r_4\theta)),
\end{aligned}\qquad\qquad \theta\in [0,2\pi).
\end{equation}

We are now ready to classify the free actions by $\sg^1$-subgroups of $\spin{8}< \so{8}\times \so{8}$ on $\sph^7\times \sph^7$ and of $\spin{7}$ on $\sph^6\times \sph^7$.

\begin{theorem}\label{thm:free} The circle subgroup $\sg^1 < \T^4 < \spin{8}< \so{8}\times \so{8}$ determined by $(n_1,n_2,n_3,n_4)$ acts freely on $\sph^7\times \sph^7$ if and only if $\gcd(\ell_i, r_j) = 1$ for all $i,j \in \{1,2,3,4\}$.  There are infinitely many such actions up to equivalence of actions. 

Up to equivalence of actions, there is a unique free action by some subgroup $\sg^1 < \T^3 < \spin{7}$ on $\sph^6\times \sph^7$, and the projection of this action to the $\sph^7$-factor is equivalent to the Hopf action.
\end{theorem}

Before proving this, we note the behavior of the $\gcd$ function when some entry is $0$.  As every integer is a divisor of $0$, we have that $\gcd(0,x) = |x|$ for any non-zero integer $x$.  On the other hand, $\gcd(0,0)$ is undefined (and, in particular, is not equal to $1$) as every integer is a divisor of $0$.

\begin{proof}   For $i = 1,..., 8$, let $e_i$ denote the element of $\sph^7$ whose coordinates are all $0$, except that the $i$-th coordinate is $1$. We first observe that if $(A,B)\in \mathsf{T}^4$ fixes a point of $\sph^7\times\sph^7$, then at least one of the four blocks of $A$ and $B$ is the identity $I_2$, each of which fixes the plane $\operatorname{span}\{e_{2i-1},e_{2i}\}$  for some $i\in\{1,2,3,4\}$. Hence, $(A,B)\in \sg^1< \mathsf{T}^4 < \spin{8}$ fixes a point $(x,y)\in \sph^7\times \sph^7$ if and only if it fixes a point of the form $(e_{2i},e_{2j})$, for some $i,j\in\{1,2,3,4\}$, so we may focus solely on these points. 

Suppose first that $\gcd(\ell_i,r_j) = 1$ for all $i,j\in\{1,2,3,4\}$, and suppose some $(A,B)\in \sg^1$ fixes a point in $\sph^7\times \sph^7$.   Then it fixes a point $(e_{2i}, e_{2j})$ for some $i,j\in \{1,2,3,4\}$.  It follows that the $\theta\in [0,2\pi)$ corresponding to $(A,B)$ satisfies $R(\ell_i \theta) = I_2 = R(r_j \theta)$.  This implies that  both $\ell_i \theta, r_j \theta \in 2\pi \mathbb{Z}$.  Since $\gcd(\ell_i,r_j) = 1$, there are integers $p$ and $q$ for which $p\ell_i + q r_j = 1$.  Then $\theta = p\ell_i \theta + q r_j \theta \in 2\pi \mathbb{Z}$, which implies that $(A,B)$ is the identity. Since the action is effective (see Equation~\eqref{EQ:maximal_torus} and below), we conclude that it is free.

On the other hand, assume that $\gcd(\ell_i, r_j) > 1$ or $(\ell_i, r_j)=(0,0)$ for some $i,j\in \{1,2,3,4\}$. If we set $\theta = 2\pi/ \gcd(\ell_i, r_j) \in (0,2\pi)$ in the first case, or we take an arbitrary $\theta\in(0,2\pi)$ in the second case, then the corresponding $(A,B)$ is not the identity because the action is effective, but it fixes the point $(e_{2i}, e_{2j})$. Thus, the action is not free.

One infinite family of inequivalent free actions is provided by taking $(n_1,n_2,n_3,n_4) = (1,1,1,k)$ where $k$ is congruent to $3\pmod{6}$. The inequivalence of these actions for $|k|\neq |k'|$ will follow from Theorem \ref{thm:p1}, which  demonstrates that the first Pontryagin class of the quotient space $(\sph^7\times \sph^7)/\sg^1$ has the form $4(k^2+5)$ times a generator in $H^4$.

\smallskip

Finally, we show that up to equivalence, there is a unique $\sg^1< \T^3$ acting freely on $\sph^6\times \sph^7$.  The key observation is that, since $\T^3< \T^4$ is characterized by having $\alpha_1 = 0$, the projection of this action to the $\sph^6$-factor fixes the point $i\in \mathbb{O}$. Consequently, the projection to the $\sph^7$-factor must act freely, and, because it is a linear action, it must be equivalent to the Hopf action. Now, since $\alpha_1 = 0$, we have $n_1 = 0$. For the Hopf action, we must have $|r_j| = 1$  for every $j\in\{1,2,3,4\}$. It follows that $|n_3| = |n_4| = |n_2 - n_4| = |n_2 + n_3| = 1$.  These clearly imply that $n_2 \in \{0, \pm 2\}$. Moreover, $(n_2,n_3,n_4)$ equals one of
$$
\pm(0,1,1),\quad \pm(0,1,-1),\quad \pm(2,-1,1).
$$
It follows that $A$ is conjugate to $\diag(R(0),R(0),R(0), R(2\theta))$. Thus, up to equivalence, there is a unique free action.
\end{proof}

\subsection{Free $\so{3}$ and $\su{2}$-actions.}\label{subsec:freeso3su2}

We now address which of the effectively free circle actions on $\sph^7\times \sph^7$ or $\sph^6\times \sph^7$ discussed above extend to free actions by $\su{2}$ or $\so{3}$.

Our main result is the following theorem, which will follow directly from Propositions~\ref{prop:step1}, \ref{prop:two_actions_equiv} and~\ref{prop:freeS6}.

\begin{theorem}\label{thm:summary}Up to equivalence of actions, there is precisely one $\su{2}<\spin{8}< \so{8}\times \so{8}$ whose induced action on $\sph^7\times \sph^7$ is free.  Up to equivalence, the action can be described both as the Hopf action on both factors and also via multiplication by $\so{3}$ embedded as $A\mapsto \diag(1,1,1,1,1,A)$ on one factor and the Hopf action on the other.    Up to conjugacy in $\spin{7}$, there is precisely one $\su{2}<\spin{7}$ whose induced action on $\sph^6\times \sph^7$ is free.  This action is equivalent to the restriction of the second-mentioned action on $\sph^7\times \sph^7$ to $\sph^6\times \sph^7$.  There is no $\so{3}<\spin{8}$ (respectively $\so{3}<\spin{7}$) whose induced action on $\sph^7\times \sph^7$ (respectively $\sph^6\times \sph^7)$ is free.

\end{theorem}

\begin{remark}
It follows from \cite{Ol} that there are no free continuous $\so{3}$-actions on $\sph^7\times \sph^7$, but our proof will not use this result.
\end{remark}

We will prove Theorem \ref{thm:summary} first in the case of free actions on $\sph^7\times \sph^7$.  The proof follows three main steps.  First, we will investigate subgroups $H<\spin{8}$ which are isomorphic to either $\su{2}$ or $\so{3}$; in particular we simplify the problem to a finite number of candidates. The second step is to determine which of these induced actions are free.  This turns out to be an easy consequence of Proposition~\ref{prop:maxtorus} and Theorem~\ref{thm:free}, and we obtain two subgroups $H<\spin{8}$ isomorphic to $\su{2}$ acting freely on $\sph^7\times \sph^7$. The last step is to show that these two free actions are actually equivalent.

For the rest of the section, given $H,H'<\spin{8}$ with $H$ and $H'$ isomorphic to either $\su{2}$ or $\so{3}$, we say $H$ and $H'$ are \emph{equivalent} if their induced actions on $\sph^7 \times \sph^7$ are equivalent as actions (in the sense of the usual definition given at the beginning of Section~\ref{SEC:free_actions}). We observe that since the action by $\spin{8}$ on $\sph^7\times \sph^7$ is effective, if $H$ and $H'$ are equivalent, then they are necessarily isomorphic as Lie groups.  Thus, we will always assume that $H$ and $H'$ are both isomorphic to $\su{2}$ or both isomorphic to $\so{3}$.

We now work towards solving the first step.

\begin{lemma}\label{lem:spin8conjugate}
Suppose $H$ and $H'$ are both subgroups of $\spin{8}$ as above.  If they are conjugate in $\spin{8}$, then $H$ and $H'$ are equivalent.
\end{lemma}

\begin{proof}  Suppose $g\in \spin{8}$ conjugates $H$ to $H'$, that is, $H' = g H g^{-1}$.  Then the map $f_g\colon\sph^7\times \sph^7\rightarrow \sph^7\times \sph^7$ given by $f_g(m) = g\ast m$ is a diffeomorphism which intertwines the two actions. Indeed, if $h\in H$, then for any $m\in \sph^7\times \sph^7$, 
\[
ghg^{-1} \ast (f_g(m)) = (ghg^{-1})\ast(g\ast m) = (gh)\ast m = f_g(h\ast m).\qedhere
\]
\end{proof}

The next proposition will allow us to bound from above the number of subgroups  $\su{2}<\spin{8}$ or $\so{3}<\spin{8}$ acting on $\sph^7\times \sph^7$, up to equivalence, by considering the induced representations.

\begin{proposition}\label{prop:reptheory}
Suppose $H$ and $H'$ are subgroups of $\spin{8}$ as above,  and let $\pi\colon\spin{8}\rightarrow \so{8}$ denote the double cover.  If $\pi|_{H}$ and $\pi|_{H'}$ define equivalent real $8$-dimensional representations of $H\cong H'$, then $H$ and $H'$ are equivalent.
\end{proposition}

\begin{remark}
The converse is not true. Namely, there are $H,H'$ which are equivalent but whose induced representations $\pi|_{H},\pi|_{H'}$ are inequivalent, see~Proposition~\ref{prop:two_actions_equiv}.
\end{remark}

\begin{proof}[Proof of Proposition~\ref{prop:reptheory}]  Suppose that $\pi|_{H}$ and $\pi|_{H'}$ define equivalent real $8$-dimensional representations of $H\cong H'$.  From \cite[Theorem~2, p.~17]{Ma}, the images $\pi(H)$ and $\pi(H')$ are conjugate in $\oo{8}$.

Let us assume initially that $\pi(H)$ and $\pi(H')$ are conjugate in $\so{8}$, say $h \pi(H) h^{-1} = \pi(H')$ for some $h\in \so{8}$.  Let $\tilde{h} \in \spin{8}$ denote either of the two preimages of $h$.  Then $\tilde{h} H \tilde{h}^{-1} = H'$.  In particular, $H$ and $H'$ are conjugate in $\spin{8}$, so they induce equivalent actions on $\sph^7\times \sph^7$ by Lemma~\ref{lem:spin8conjugate}.

We now handle the case where $\pi(H)$ and $\pi(H')$ are conjugate in $\oo{8}$ but not in $\so{8}$, say $h\pi(H) h^{-1} = \pi(H')$ with $h\in \oo{8}\setminus \so{8}$.  Let $\kappa\colon\so{8}\rightarrow \so{8}$ denote the map given by conjugation by 
$$\gamma = \diag(1,-1,-1,-1,-1,-1,-1,-1)\in \oo{8}\setminus\so{8}$$
and define $\tilde{\kappa}\colon\so{8}\times \so{8}\rightarrow \oo{8}\times \oo{8}$ by $\tilde{\kappa}(A,B) = (\kappa(B),\kappa(A))$. 

From \cite[Theorem, p.~151]{DS}, $\tilde{\kappa}$ maps $\spin{8}$ to itself.  In addition, recall that the projection $\pi\colon\spin{8}\rightarrow \so{8}$ is given by $\pi(A,B) = C$ where $(A,B,C)$ satisfies Equation~\eqref{eqn:triality}. De Sapio also shows that if $(A,B,C)$ satisfies Equation~\eqref{eqn:triality}, then  $(\kappa(B),\kappa(A),\kappa(C))$  satisfies Equation~\eqref{eqn:triality} as well.  In particular, it follows that $\pi \circ \tilde{\kappa} = \kappa \circ \pi$.

We now observe that $$ \pi(\tilde{\kappa}(H'))  = \kappa(\pi(H')) = \kappa( h \pi(H) h^{-1}) = (\gamma h) \pi(H) (\gamma h)^{-1}.$$ Since $\det(\gamma h) = \det(\gamma)\det(h) = 1$, it follows that $\pi(\tilde{\kappa}(H'))$ and $\pi(H)$ are conjugate in $\so{8}$.  Thus, from the first case above, $H$ and $\tilde{\kappa}(H')$ are conjugate in $\spin{8}$.   It then follows from Lemma~\ref{lem:spin8conjugate} that $H$ and $\tilde{\kappa}(H')$ are equivalent.

The proof will be complete if we can show that $\tilde{\kappa}(H')$ and $H'$ are equivalent.  To show this,  consider the function $f\colon\sph^7\times \sph^7\rightarrow \sph^7\times \sph^7$ given by $f(x,y) = (\gamma y, \gamma x)$.  We claim that $f$ is an equivariant diffeomorphism between the action by $H'$ and the action by $\tilde{\kappa}(H')$.  Indeed, let $(A,B)\in H'<\spin{8}<\so{8}\times \so{8}$, so that $\tilde{\kappa}(A,B) = (\gamma B \gamma^{-1}, \gamma A \gamma^{-1})\in \tilde{\kappa}(H')$.  Then we have \begin{align*} f((A,B)\ast(x,y)) &= f(Ax, By)  
= (\gamma By, \gamma Ax)
= (\gamma B \gamma^{-1}, \gamma A\gamma^{-1})\ast (\gamma y, \gamma x)\\
&= (\gamma B \gamma^{-1}, \gamma A\gamma^{-1}) \ast f(x,y).\qedhere\end{align*}
\end{proof}

The next result reduces the problem of determining whether a group acts freely to the induced action by a maximal torus.

\begin{lemma}\label{lem:freetori} Suppose a connected Lie group $G$ acts on a set $M$ and let $T< G$ denote a maximal torus.  Then the $G$-action on $M$ is free if and only if the restriction of the action to $T$ is free.
\end{lemma}

\begin{proof}Of course, if the $G$-action is free, then so is the restriction to any subgroup.  So we only need to prove the converse. 
To that end, suppose the $T$-action is free and that $g\in G$ fixes $m\in M$.  By the maximal torus theorem, there is an $h\in G$ for which $hgh^{-1}\in T$.  The element $hgh^{-1}\in G$ fixes $h\ast m\in M$.  Since the $T$-action is free, we deduce that $hgh^{-1}$ is the identity, which immediately implies that $g$ is the identity.
\end{proof}

Now we use Proposition~\ref{prop:reptheory} to bound the number of subgroups, up to equivalence, and we determine which of the corresponding actions are free.

\begin{proposition}\label{prop:step1}  Up to equivalence of (non-trivial) actions on $\sph^7\times \sph^7$, there are at most $9$ subgroups $H<\spin{8}< \so{8}\times \so{8}$ which are isomorphic to $\su{2}$ or $\so{3}$.These subgroups are listed in Table~\ref{table:rep}, in terms of the associated $8$-dimensional real representations given by the projections of such subgroups under $\pi\colon \spin{8}\to\so{8}$.  Among these subgroups, there are only two whose  induced actions on $\sph^7\times \sph^7$ are free.
Both free actions are actions by $\su{2}$ and, in particular, there are no free actions by $\so{3}$.  

For one of the free actions, up to equivalence of actions, $\su{2}$ acts by the Hopf action on each factor. 
For the other free action, up to equivalence, $\su{2}$ acts via the Hopf action on one factor and via the map $\su{2}\rightarrow \so{3}\rightarrow \so{8}$ on the other, where the first map is the double cover and the second map sends $A\in \so{3}$ to $\diag(1,1,1,1,1,A)\in \so{8}$.

\end{proposition}

\begin{proof}
Suppose $H$ and $H'$ are both subgroups of $\spin{8}$ and that their induced actions on $\sph^7\times \sph^7$ are not equivalent.  Then it follows from Proposition \ref{prop:reptheory} that the non-trivial $8$-dimensional real representations defined by $\pi|_{H}$ and $\pi|_{H'}$ are inequivalent.

The strategy of the proof is as follows. First, we consider all non-trivial $8$-dimensional real representations $W$ of $\su{2}$ and $\so{3}$, up to equivalence. Then, for each such representation $W$, we determine whether a subgroup $H<\spin{8}$ isomorphic to $\su{2}$ or $\so{3}$ with $\pi|_{H}$ equivalent to $W$ acts freely on $\sph^7\times \sph^7$ or not.

Note that, since $\so{3} = \su{2}/\{\pm I_2\}$, a representation of $\so{3}$ lifts to a representation of $\su{2}$, so we can focus on representations of $\su{2}$. We emphasize that, given a representation $W$ of $\su{2}$, the method we use for determining whether a subgroup $H$ with $\pi|_{H}$ equivalent to $W$ acts freely on $\sph^7\times \sph^7$ does not require knowing whether $W$ is the lift of an $\so{3}$-representation. In other words, it does not require knowing whether $H$ is isomorphic to $\so{3}$ or $\su{2}$. This will only be discussed for those $H$ which act freely.

To enumerate the $\su{2}$-representations, recall that $8$-dimensional real representations of $\su{2}$ are parameterized by partitions of $8$, i.e., a way of representing $8$ as a sum of positive integers, where every even number appears an even number of times. This is a consequence of the following well-known representation theoretic results:

\begin{itemize}\item  For each positive integer $d$, $\su{2}$ has a unique irreducible representation $\phi_d$ of (complex) dimension $d$.  The weights have the form $d-1, d-3, d-5,..., -(d-1)$.

\item  Each $\phi_d$ is self-conjugate, i.e., $\phi_d$ is isomorphic to its complex conjugate.  When $d$ is odd, $\phi_d$ is real, and when $d$ is even, $\phi_d$ is  quaternionic.

\item  A representation (of any compact Lie group) is real if and only if it is isomorphic to a representation of the form $\bigoplus_i V_i \oplus \bigoplus_j (W_j\oplus \overline{W}_j)$, where each $V_i$ is real and $\overline{W}_j$ denotes the conjugate of the representation $W_j$.

\end{itemize}

From this, one easily sees that the partitions given in Table \ref{table:rep} describe all the non-trivial $8$-dimensional real representations of $\su{2}$.  The second column gives the image $f(\sg^1)$, up to conjugacy, of the maximal torus $\sg^1< \su{2}$ via the map $f\colon \su{2}\to\so{8}$ defining the representation.  Here, we have abbreviated the image of the form $\diag(R(a\theta), R(b\theta), R(c\theta), R(d\theta))$ as $(a,b,c,d)$. This image is easy to fill in based on the information above.

Once the second column is filled in, we use Proposition \ref{prop:maxtorus} to lift $f(\sg^1)<\so{8}$ to $\spin{8}$; this information is contained in the third and fourth rows. Note that this lift to $\spin{8}$ coincides with the maximal torus of $H$ in $\spin{8}$, up to conjugacy.

The fifth column is obvious.  From Theorem \ref{thm:free}, one can use the fifth column to easily determine whether the maximal torus of $H$ acts freely on $\sph^7\times \sph^7$.  By Lemma~\ref{lem:freetori}, $H$ acts freely on $\sph^7\times \sph^7$ if and only if its maximal torus acts freely, so now one can easily fill in the sixth column.

For example, for the first representation $5+1+1+1$, the weights are $\pm 4$, $\pm 2$, with the rest being zero, so $f(\sg^1)$ consists of matrices of the form $ C:= \diag(R(4\theta), R(2\theta), R(0), R(0))$, which is abbreviated $(4,2,0,0)$ in Table \ref{table:rep}. We then lift $C$ to $\spin{8}$ using Proposition~\ref{prop:maxtorus}, getting 
$$(A,B) = \bigl(\diag(R(-\theta), R(-\theta), R(-3\theta), R(3\theta)), (\diag(R(-3\theta), R(3\theta), R(\theta), R(-\theta))\bigr),$$
which we abbreviate to $((-1,-1,-3,3), (-3,3,1,-1))$. In this case $\gcd(\ell_3,r_1) = 3$, so the action is not free.

We finally describe the two free actions in more detail.  We note that both of them have $|r_i| = 1$ for all $i$, and in particular, the image in the right factor $\so{8}$ contains $-I_8\in Z(\so{8})$.  Since $\so{3}$ is centerless, there is no homomorphism $\so{3}\rightarrow \so{8}$ whose image contains $-I_8$.  In particular, the two free actions must be free actions by $\su{2}$, so $\so{3}$ does not act freely.

It remains to describe precisely the free $\su{2}$-actions on $\sph^7\times \sph^7$.  As recorded in Table \ref{table:rep}, we already know how a maximal torus of each $\su{2}$ acts on $\sph^7\times \sph^7$.  Thus, we will first argue that the action is determined up to conjugacy in $\so{8}\times \so{8}$ by the restriction of the action to a maximal torus.  Then, for each freely acting $\su{2}<\spin{8}$, we describe an embedding of $\su{2}<\so{8}\times \so{8}$ having the correct maximal torus.

The first step follows immediately from the well-known result of Dynkin that a homomorphism from $\su{2}$ to a compact Lie group is determined, up to conjugation, by its restriction to a maximal torus of $\su{2}$ (see for example~\cite[Theorem~7]{Vogan}).   Specifically, if there are two copies of $\su{2}<\so{8}\times \so{8}$ sharing a maximal torus, one simply applies Dynkin's result to the inclusion maps.

To accomplish the second step, we need only verify that the actions described in the statement of the proposition, when interpreted as subgroups of $\so{8}\times \so{8}$, have the maximal tori $(A,B)$ listed in Table~\ref{table:rep}. Note that the Hopf action of $\su{2}<\so{8}$ on $\sph^7$ has maximal torus equal to $(1,1,1,1)$, up to signs. Thus, the Hopf action on both factors corresponds to the representation $3+1+1+1+1+1$. 

Now consider the action of $\su{2}<\so{8}$ on $\sph^7$ via the map $\su{2}\rightarrow \so{3}\rightarrow \so{8}$ described in the statement. Note that the standard $\so{3}$-action on $\mathbb{R}^3$ is irreducible (since it is transitive on the unit sphere), so the $\su{2}$-action is equivalent to  $\R^5\oplus\phi_3$, where $\phi_3$ is as defined at the beginning of the present proof and $\R^5$ is the trivial representation. This representation is precisely the one that we have denoted by $1+1+1+1+1+3$, and its maximal torus takes the form $(0,0,0,2)$. It follows that the action on $\sph^7\times\sph^7$ via the map $\su{2}\rightarrow \so{3}\rightarrow \so{8}$ on one factor and via the Hopf action on the other has maximal torus $((0,0,0,2),(1,1,1,1))$ up to signs, and hence, by comparing it with the $(A,B)$ in Table~\ref{table:rep}, it corresponds to the representation $2+2+2+2$.\qedhere

\begin{table}
\small
	\begin{tabular}{c|c|c|c|c|c}
		
		Representation & $f(\sg^1)$ & $A$ & $B$  & $\gcd(\ell_i,r_j)$ & Free? \\
		
		\hline
		
		$5+1+1+1$ & $(4,2,0,0)$ & $ (-1,-1,-3,3) $& $ (-3,3,1,-1) $  & $\gcd(\ell_3,r_1) = 3$ & No\\
		
		$4 + 4$ & $(3,3,1,1)$ & $(0,0,-2,4)$ &  $(-3,3,1,1)$ & $\gcd(\ell_1,r_1) = 3$ & No \\
		
		$3 + 2 + 2 + 1$ & $(2,1,1,0)$  & $(0,-1,-1,2)$ & $(-2,1,1,0)$ & $\gcd(\ell_1,r_1) = 2$ & No \\
		
		$3+1+\overset{5)}{\cdots}+1$ & $(2,0,0,0)$ &$(-1,-1-1,1)$ & $(-1,1,1,-1)$&  $\gcd = 1$ always & Yes\\
		
		$2+2+2+2$ & $(1,1,1,1)$ &  $(0,0,0,2)$ & $(-1,1,1,1)$ & $\gcd = 1$ always & Yes\\
		
		$2+2+1+\!\overset{4)}{\cdots}\!+1$ & $(1,1,0,0)$ & $(0,0,-1,1)$ & $(-1,1,0,0)$& $\gcd(\ell_1,r_3)$\! undef.  & No \\
		
		$7+1$ & $(6,4,2,0)$ & $(0,-2,-4,6)$ & $(-6,4,2,0)$& $\gcd(\ell_1,r_1) = 6 $ & No  \\
		
		$5+3$ & $(4,2,2,0)$ &  $(0,-2,-2,4)$ & $(-4,2,2,0)$ & $\gcd(\ell_1,r_1) = 4 $ & No\\
		
		$3 + 3 + 1 + 1$ & $(2,2,0,0)$ & $(0,0,-2,2)$ & $(-2,2,0,0)$ & $\gcd(\ell_1,r_1) =2 $ & No  \\
		
	\end{tabular}
	\bigskip
	\captionsetup{width=.88\textwidth}
	\caption{Enumeration of subgroups $H<\spin{8}$ isomorphic to $\su{2}$ or $\so{3}$, up to equivalence of action on $\sph^7\times \sph^7$, described in terms of the associated $8$-dimensional representation $f\colon\su{2}\rightarrow \so{8}$. The following additional information is given: the coefficients describing $f(\sg^1)$, where $\sg^1 < \su{2}$ is a maximal torus; the coefficients that determine the lift $(A,B)$ of $f(\sg^1)$ to $\spin{8}$; and the freeness of the $H$-action on $\sph^7\times \sph^7$, determined by the $\gcd$ of the coefficients of $A$ and $B$ via Theorem~\ref{thm:free}.}\label{table:rep}

\end{table}

\end{proof}

To complete the proof of Theorem \ref{thm:summary} for the case of $\sph^7\times \sph^7$, we need to demonstrate that the two free $\su{2}$-actions on $\sph^7\times \sph^7$ from Proposition \ref{prop:step1} are equivalent.

\begin{proposition}\label{prop:two_actions_equiv}
The two free actions of $\su{2}$ on $\sph^7\times \sph^7$ in Proposition \ref{prop:step1} are equivalent.
\end{proposition}

\begin{proof} Consider the action $\ast_1$ by $\su{2}\cong \sp{1}$ on $\sph^7\times \sph^7 \subset \mathbb{H}^2 \oplus \mathbb{H}^2$ given by 
$$q\ast_1 (a,b) = q\ast_1 ((a_1,a_2), (b_1,b_2)) = ((a_1 \overline{q}, q a_2), (q b_1, qb_2)).$$
Note that the action $q(b_1,b_2)=(q b_1, qb_2)$ on the second factor is the Hopf action, while the action $q(a_1,a_2)=(a_1 \overline{q}, q a_2)$ on the first factor is equivalent to the Hopf action via the intertwining map $(a_1,a_2)\mapsto (\overline{a}_1,a_2)$. Thus $\ast_1$ is equivalent to the action in Proposition~\ref{prop:step1} which is the Hopf action on both factors.  We also consider the action $\ast_2$ given by 
$$ q\ast_2(a,b) = ((a_1 , q a_2 \overline{q}), (q b_1, qb_2)).$$
Noting that the conjugation map $a_2 \mapsto q a_2 \overline{q}$ fixes $1\in \mathbb{H}$ and acts via the double cover $\su{2}\rightarrow \so{3}$ on $\operatorname{Im}\mathbb{H}$, we see that $\ast_2$ is equivalent to the second free action from Proposition~\ref{prop:step1}.

We will now show that $\ast_1$ and $\ast_2$ are equivalent. To do this, we see $\sph^7\subset\mathbb O$ as the set of unit octonions and we consider the diffeomorphism $f\colon\sph^7\times \sph^7 \rightarrow \sph^7\times \sph^7$ given by $f(a,b) = (ab,b)$, which has inverse $f^{-1}(a,b) = (ab^{-1},b)$. In order to describe the octonion product, recall that the octonions are formed from the Cayley-Dickson process applied to the quaternions.  In particular, writing octonions $a = (a_1,a_2)$ and $b = (b_1,b_2)$ as pairs of quaternions, the octonionic product is defined by $$ab = (a_1, a_2) (b_1,b_2) = (a_1 b_1 - \overline{b}_2 a_2, b_2a_1 + a_2\overline{b}_1).$$

We claim that $f$ intertwines $\ast_1$ and $\ast_2$.  We compute \begin{align*}f( q\ast_1 (a,b)) &= f((a_1 \overline{q}, q a_2), (q b_1, qb_2))\\
&=(( a_1 \overline{q} qb_1 - \overline{b}_2 \overline{q} q a_2, qb_2 a_1 \overline{q} + qa_2 \overline{b}_1 \overline{q} ) , (qb_1, qb_2))\\
&= ((a_1 b_1 - \overline{b}_2 a_2, q(b_2a_1 + a_2\overline{b}_1) \overline{q}), (qb_1,qb_2))\\
&= q\ast_2f(a,b).\qedhere
\end{align*}
\end{proof}

We now turn attention to classifying the free actions of either $\su{2}<\spin{7}$ or $\so{3}<\spin{7}$ on $\sph^6\times \sph^7$.  The following proposition will complete the proof of Theorem \ref{thm:summary}.

\begin{proposition}\label{prop:freeS6}  There is no $\so{3}<\spin{7}$ whose induced action on $\sph^6\times \sph^7$ is free.  Up to equivalence of actions on $\sph^6\times \sph^7$, there is a unique free action by a subgroup $\su{2}<\spin{7}$ on $\sph^6\times \sph^7$. This action is given by multiplication by $\so{3}$ embedded as $A\mapsto \diag(1,1,1,1,A)\in\so{7}$ on the $\sph^6$-factor and by the Hopf action on the $\sph^7$-factor.

\end{proposition}

\begin{proof}

From Theorem~\ref{thm:free}, there is, up to conjugacy in $\spin{7}$, a unique $\sg^1 < \spin{7}$ whose induced action on $\sph^6\times \sph^7$ is free, and the projection of the action to the $\sph^7$-factor is the Hopf action. Recall also from the proof of Theorem~\ref{thm:free} that $|r_i| = 1$ for all $i$, so the image in the right factor of $\sg^1<\spin{7}<\so{8}\times \so{8}$ contains $-I_8\in Z(\so{8})$. As $\so{3}$ is centerless, there is no homomorphism $\so{3}\rightarrow \so{8}$ whose image contains $-I_8$, so this $\sg^1$-action cannot extend to an $\so{3}$-action.

As in the end of the proof of Proposition \ref{prop:step1}, there is at most one conjugacy class of $\su{2}<\spin{7}$ containing the freely acting $\sg^1<\spin{7}$.  It remains to find one such $\su{2}$.
Note that the free $\su{2}$-action on $\sph^7\times \sph^7$ found in Proposition \ref{prop:step1} given by multiplication by $\so{3}$ embedded as $A\mapsto \diag(1,1,1,1,1,A)$ on one factor and the Hopf action on the other preserves the subset $\{1\}\times \sph^7\subset \sph^7\times \sph^7$.
In particular, it follows that this $\su{2}$ is actually a subset of $\spin{7}<\spin{8}$, and clearly this $\su{2}$ contains the freely acting $\sg^1$. 
\end{proof}

\subsection{Submersions to $\sph^4$}\label{subsec:S4}

In this subsection, we show that several of our examples admit Riemannian submersions to a round $\sph^4$.  We recall that, from Theorem \ref{thm:free} and Propositions~\ref{prop:two_actions_equiv} and \ref{prop:freeS6}, for each $k\in \{6,7\}$, there is an essentially unique free isometric $\su{2}$-action on $\sph^k\times \sph^7$.  In addition, there is an essentially unique free isometric $\sg^1$-action on $\sph^6\times \sph^7$, but there are infinitely many free isometric $\sg^1$-actions on $\sph^7\times \sph^7$.

\begin{proposition} \label{prop:S4}
For each $H\in \{\{\id\}, \su{2}\}$ and $k\in \{6,7\}$, the manifolds $(\sph^k\times \sph^7)/H$ carry a $\Ric_2>0$ metric which admits a Riemannian submersion to a round $\sph^4$.  The same is true for $(\sph^6\times \sph^7)/\sg^1$ and at least one quotient of the form $(\sph^7\times \sph^7)/\sg^1$.
\end{proposition}

\begin{proof}  From \cite[Tabelle 101]{Es84} or \cite[Theorem 1.1]{DV14}, there is a free biquotient action of $\su{2}\times \gg$ on $\spin{7}$ with quotient $\sph^4$, where $\su{2}$ acts by left multiplication and $\gg$ acts by right multiplication by inverses.  By our construction of the metric $\langle\cdot,\cdot\rangle_t$ on $G=\spin{7}$ given in Subsection~\ref{sec:MetricsOnLieGroups}, this $(\su{2}\times\gg)$-action is isometric.  Moreover, the isotropy representation of $\spin{7}/\gg$ is irreducible,  and hence $\spin{7}/\gg$ with the induced metric is a round $\sph^7$.  The  $\su{2}$-factor then acts by the Hopf action, so  the induced metric on $\su{2}\backslash \spin{7}/\gg\cong \sph^4$ is round.  We note that for any closed subgroup $H< \su{2}$, the ($\su{2}\times \gg$)-biquotient action restricts to a free biquotient action by $H\times \su{3}< \su{2}\times \gg$. Moreover, the induced metric on $\spin{7}/\su{3}$ is our $\Ric_2>0$ metric, so each biquotient $H\backslash \spin{7}/\su{3}$ admits a $\Ric_2>0$ metric which submerses onto a round $\sph^4$.   We showed in Theorem \ref{thm:free} and Proposition~\ref{prop:freeS6} that  for each $H\in \{\sg^1,\su{2}\}$, there is, up to equivalence, a unique such action.  In particular,  each of our examples of the form $(\sph^6\times \sph^7)/H$ is diffeomorphic to $H\backslash \spin{7}/\su{3}$, so each admits a Riemannian submersion to a round $\sph^4$.

In a similar fashion, by \cite[Theorem 1.1]{DV14} there are (finitely many, up to conjugacy) free biquotients actions of $\su{2}\times \spin{7}$ on $\spin{8}$ with quotient a round $\sph^4$. So, for any $H< \su{2}$, including $H = \sg^1$, the projection $H\backslash \spin{8}/\gg \rightarrow \sph^4$ is a Riemannian submersion.  As in the previous case, the induced metric on $\spin{8}/\gg$ has $\Ric_2>0$, and thus the same is true on any quotient of $\spin{8}/\gg$.  Since $H\backslash \spin{8}/\gg$ is diffeomorphic to some $(\sph^7\times \sph^7)/H$, these spaces admit Riemannian submersions to a round $\sph^4$.
\end{proof}

\subsection{Totally geodesic embeddings}\label{subsec:tg}

We conclude this section by proving the existence of two totally geodesic embeddings that relate the homogeneous spaces appearing in Theorem~\ref{THM:submersion_triples} induced by the triples (1) and (5); and by (4) and (6), respectively. This will show the existence of a totally geodesic $\su{2}^3/\Delta\su{2}\cong\sph^3\times\sph^3$ inside $\spin{8}/\mathsf{G}_2\cong\sph^7\times\sph^7$ and of a totally geodesic $\su{2}^2/\sg^1\cong\sph^2\times\sph^3$ inside $\spin{7}/\su{3}\cong\sph^6\times\sph^7$, for any of the homogeneous metrics $q_t$, $t\in(0,\infty]$. Recall from the beginning of Section~\ref{SEC:isometry_group} that $q_t$, $t\in(0,+\infty)$, denotes the submersion metric introduced in Section~$\ref{SEC:submersion_metrics}$, and $q_\infty$ denotes a normal homogeneous metric. Except for $q_\infty$ on $\spin{7}/\su{3}$, these are $\Ric_2>0$ metrics and, thus, the totally geodesic submanifolds we identify inherit (extrinsically homogeneous) $\Ric_2>0$ metrics.

\begin{proposition}\label{prop:totallygeo}
	For each $t\in(0,+\infty]$, there exist totally geodesic embeddings
	\[ \sph^2\times\sph^3\hookrightarrow (\spin{7}/\su{3}, q_t) \qquad \text{and}\qquad \sph^3\times\sph^3 \hookrightarrow (\spin{8}/\mathsf{G}_2, q_t), \]
given as fixed point sets of an isometric $\mathsf{S}^1$-action.  Moreover, for each $k\in \{2,3\}$, the $\sg^1$-action commutes with an isometric $\su{2}^k$-action, which induces an isometric transitive $\su{2}^k$-action on $\sph^k\times \sph^3$ with isotropy group equal to $\sg^1$ and $\Delta\su{2}$, respectively.
\end{proposition}

\begin{proof} A circle subgroup of $\so{8}$ is, up to conjugacy, given as the set of matrices of the form $\diag(R(n_1 \theta), R(n_2\theta), R(n_3\theta), R(n_4\theta))$, $\theta\in\R$.  For readability, we will simply denote this as $(n_1,n_2,n_3,n_4)$.

Consider the Lie group $\mathsf{S}^1$ given by the  matrices $(A,B)\in\so{8}\times\so{8}$ of the form $(A,B) = \big((0,0,1,1) , (0,0,1,1)\bigr)$.   In the notation of Proposition~\ref{prop:maxtorus}, this is taking $\theta_1 = \theta_2 = 0$ and $\theta_3 = \theta_4 = \theta/2$.  Hence, this $\mathsf{S}^1$ is a Lie subgroup of $\spin{8}< \so{8}\times \so{8}$.   In fact, since $A(1)  = 1$, Equation~\eqref{eqn:triality} implies $A(1)B(y) = C(y)$ for all $y$, so $B=C$.  Moreover, since $B(1) = 1$, we likewise conclude that $A=C$. Hence, $\sg^1$ is a subgroup of $\gg$.  As $\gg<\spin{7}<\spin{8}$,  Theorems~\ref{thm:isometry_group}, \ref{thm:isometriesqt} and Remark~\ref{rem:spin7a_normal} imply this $\mathsf{S}^1$ is a Lie subgroup of the isometry group of $(\spin{8}/\mathsf{G}_2, q_t)$ or $(\spin{7}/\su{3}, q_t)$, for any $t\in(0,+\infty]$.
	
Recall that the induced $\mathsf{S}^1$-action on $\sph^k\times\sph^7$, $k\in\{6,7\}$, is given by multiplication by $(A,B)$ on $\sph^k\times\sph^7\subset \R^8\times\R^8$.  Take coordinates $(x_i ,y_j)$ with $1\leq i, j \leq 8$ on $\R^8\times\R^8$. Then the multiplication by $(A,B)$ has the fixed point set given by equations $x_5 = x_6 =x_7=x_8 = y_5 = y_6 = y_7 = y_8=0$. In $\sph^7\times\sph^7$, this fixed point set is the product $\sph^3\times\sph^3$, where $\sph^3$ in $\sph^7$ is given by the first four coordinates. Restricting to $\sph^6\times\sph^7$ (where $\sph^6$ is the sphere of unit purely imaginary octonions), we get the fixed point set $\sph^2\times\sph^3$ of $\sph^6\times\sph^7$.

We now show that the centralizer of $\sg^1$ in $\spin{8}$  contains a copy of $\su{2}^3$.  Recall that the double covering $\pi\colon\spin{8}\rightarrow \so{8}$ is given by $\pi(A,B) = C$, where $(A,B,C)$ satisfies Equation~\eqref{eqn:triality}.   As mentioned above, $\sg^1 < \gg$, which implies that $\pi(\sg^1)$ is the circle of the form  $(0,0,1,1)$.  We observe that $\so{4}$, embedded as the top left $4\times 4$-block, centralizes~$\pi(\sg^1)$.

Second, focusing on the $4\times 4$-bottom right $\so{4}$-block  containing $\pi(\sg^1)$, we see that $\pi(\sg^1)$ is precisely the image of the center of $\u{2}$ under the embedding $\u{2}\rightarrow \so{4}$ induced from the canonical map $\mathbb{C}^2\rightarrow \mathbb{R}^4$.  In particular, $\su{2}<\u{2}<\so{4}$ centralizes $\pi(\sg^1)$, as well.  Since the $\so{4}$ we identified lies in the top left $4\times 4$-block while $\su{2}$ lies in the bottom right $4\times 4$-block, it follows that $\so{4}\times \su{2}$ centralizes $\pi(\sg^1)$.  By the lifting criteria for covering maps, $\so{4}\times \su{2}$ lifts to $\su{2}^3<\spin{8}$.  Thus, we have identified a copy of $\su{2}^3$ centralizing~$\sg^1$.

We next identify the induced action of $\su{2}^3$ on $\sph^3\times \sph^3$.  To do this, we consider the three circle subgroups of $\so{4}\times \su{2}$ given by $(1,1,0,0)$, $(1,-1,0,0)$, and $(0,0,1,-1)$.  Observe that each of these circles lies in an $\su{2}$-normal subgroup so that, when lifted to $\spin{8}$, these lift to the circles in each factor of $\su{2}^3$.
We lift the $3$-torus generated by these three circles to $\spin{8}$ via Proposition \ref{prop:maxtorus}.  Thus, we see that a maximal torus $\mathsf{T}^3$ of $\su{2}^3$ consists of the $3$-torus generated by the following three circles  $$((0,0,-1,1), (-1,1,0,0)) , \quad( (-1,-1,0,0),(0,0,1,-1)), \text{ and } ((1,-1,0,0),(-1,-1,0,0)).$$

These weights imply that each $\su{2}$-factor either acts trivially or by a Hopf map (up to equivalence of actions) on each $\sph^3$-factor.  Specifically, the second and third circle factors act by commuting Hopf actions on the first $\sph^3$-factor, while the first and third act by commuting Hopf actions on the second $\sph^3$-factor.
It follows that, up the equivalence of actions, the action of $\su{2}^3\cong\sp{1}^3$ on $\sph^3\times \sph^3\subset \mathbb{H}\oplus \mathbb{H}$ takes the form $(p,q,r)\ast(a,b) = (par^{-1}, q b r^{-1})$.  Thus, the action on $\sph^3\times \sph^3$ is transitive with isotropy group at $(1,1)$ equal to $\Delta \su{2}$.

We conclude by proving the corresponding result for $\sph^2\times \sph^3$.  Observe that in the above subgroup $\su{2}^3<\spin{8}$, the $\su{2}^2$-subgroup consisting of elements with $r=q$ preserves the set $\{1\}\times \sph^3\subset\{1\}\times \sph^7$.  In particular, this $\su{2}^2$ is a subgroup of $\spin{7}$, so acts isometrically on $\sph^6\times \sph^7$ and induces an isometric action on the submanifold $\sph^2\times \sph^3$.  Finally, simply note that the subaction on $\sph^2\times \sph^3$ is clearly transitive, and the isotropy group at $(i,1)$ consists of the diagonal $\sg^1$.
\end{proof}

\begin{remark}\label{rem:tg_quotient}
	 Recall from Proposition \ref{prop:step1} that the $\su{2}$-action on $\sph^7\times \sph^7$ which is given by the Hopf action on each factor is isometric with respect to $q_t$, $t\in(0,\infty]$. This action clearly preserves the totally geodesic $\sph^3\times \sph^3$ from Proposition \ref{prop:totallygeo}, and its restriction to this submanifold can be assumed to be given by the right action $r*(a,b)=(ar^{-1},br^{-1})$, $r\in\su{2}$, $(a,b)\in\sph^3\times \sph^3\subset \mathbb{H}\oplus \mathbb{H}$.    Consider any $\sg^1$-subaction, with $\sg^1< \su{2}$. 
	 Then, the quotient $(\sph^3\times \sph^3)/\sg^1$ is a submanifold of the $\Ric_2 > 0$ quotient $(\sph^7\times \sph^7)/\sg^1$. Moreover, the left $\su{2}^2$-action $(p,q)*(a,b)=(pa,qb)$, $(p,q)\in\sp{1}^2\cong\su{2}^2$, $(a,b)\in\sph^3\times\sph^3$, commutes with the right $\sg^1$-action. Thus, it descends to a transitive action on $(\sph^3\times \sph^3)/\sg^1$. Being a homogeneous quotient of $\sph^3\times \sph^3$, it follows from \cite[p.~274, Table 12]{On:topology} that it is diffeomorphic to $\sph^2\times \sph^3$. Because $\sph^3\times \sph^3$ is totally geodesic in $\sph^7\times \sph^7$, the O'Neill formulas applied to the Riemannian submersion $\sph^7\times \sph^7\to (\sph^7\times \sph^7)/\sg^1$ imply that the submanifold $\sph^2\times \sph^3$ of $(\sph^7\times \sph^7)/\sg^1$ is also totally geodesic (cf.~\cite[Lemma~3.1]{DE}). Since the ambient space $(\sph^7\times \sph^7)/\sg^1$ has $\Ric_2>0$, the induced metric on the submanifold $\sph^2\times \sph^3$ is also of $\Ric_2>0$.

Similarly, the free isometric $\su{2}$-action on $\sph^6\times \sph^7$ from Proposition \ref{prop:freeS6} preserves the totally geodesic $\sph^2\times \sph^3$ we identified in Proposition~\ref{prop:totallygeo}.  Restricting to the $\sg^1$-subaction, it follows that the $\sg^1$-quotient of $\sph^6\times \sph^7$ endowed with the corresponding induced $\Ric_2>0$ metric admits a totally geodesic $(\sph^2\times \sph^3)/\sg^1$. Because the action of $\sg^1$ on $\sph^2\times \sph^3$ is trivial on the $\sph^2$-factor and the Hopf action on the $\sph^3$-factor, it follows that the quotient is diffeomorphic to $\sph^2\times \sph^2$.  Thus, the induced metric on $\sph^2\times \sph^2$ has $\Ric_2>0$. In this case, this $\Ric_2>0$ metric on $\sph^2\times \sph^2$ is a fortiori inhomogeneous (see the last paragraph of the proof of Proposition~\ref{prop:quotients_S2S3}). 
\end{remark}

\begin{remark}
Recall from Subsection~\ref{subsec:difeo_product_spheres} that we regard $\mathbb{S}^6\times\mathbb{S}^7$ as the subset of elements $(x,y)\in\mathbb{S}^7\times\mathbb{S}^7\subset\mathbb{O}\oplus\mathbb{O}$ with $\mathrm{Re}(x)=0$. However, this inclusion $\mathbb{S}^6\times\mathbb{S}^7\subset \mathbb{S}^7\times\mathbb{S}^7$ cannot be totally geodesic with respect our $\Ric_2>0$ metrics. In fact, no $\mathbb{S}^6\times\mathbb{S}^7$ can arise as a totally geodesic submanifold of $\mathbb{S}^7\times\mathbb{S}^7$ endowed with a homogeneous $\Ric_2>0$ metric. By \cite{Nik}, if a compact homogeneous space admits a totally geodesic hypersurface, it must be isometric to a product of spaces with constant curvature. Similarly, there is no totally geodesic $\mathbb{S}^2\times\mathbb{S}^3$ in $\mathbb{S}^3\times\mathbb{S}^3$ with respect to any homogeneous $\Ric_2>0$ metric.
\end{remark}

\section{The topology of the quotients}\label{SEC:topology}

In Sections~\ref{SEC:isometry_group} and \ref{SEC:free_actions} we have classified all connected Lie groups acting freely and isometrically on $(\spin{7}/\su{3},q_t)$, $(\spin{8}/\mathsf{G}_2,q_t)$ and $(\spin{8}/\mathsf{G}_2,q_{\infty})$, up to equivalence. In this section we compute the (integral) cohomology ring and the Pontryagin classes of the corresponding quotient manifolds. Unless otherwise stated, the cohomology groups are always taken with integer coefficients. The main result is that each one of the quotient manifolds has the cohomology ring of some product of two compact rank one symmetric spaces but it is not even homotopy equivalent to such product.

We emphasize that all (free) circle and $\su{2}$-actions considered in this section are exactly those from Theorems~\ref{thm:free} and \ref{thm:summary}, respectively. Let us begin with the cohomology rings.  

\begin{theorem}\label{thm:topology}  Suppose $k\in \{6,7\}$.  All $\sg^1$-quotients of  $\sph^k\times \sph^7$ have cohomology ring isomorphic to that of $\sph^k\times \CP^3$.  The $\su{2}$-quotients of $\sph^k\times \sph^7$ have cohomology ring isomorphic to that of $\sph^k\times \sph^4$.
\end{theorem}

\begin{proof}
The computation of the cohomology ring of both $(\sph^7\times \sph^7)/\sg^1$ and $(\sph^7\times \sph^7)/\su{2}$ can be found in Kerin's thesis \cite[Theorem 5.4.1]{KePhD}.  Thus, we will focus on the quotients of $\sph^6\times \sph^7$.

We first consider the $\su{2}$-quotient. Recall  that the $\su{2}$-action on $\sph^6\times \sph^7$ is the Hopf action on the $\sph^7$-factor.  Thus, the map $(\sph^6\times \sph^7)/\su{2}\to \sph^7/\su{2}$ given by $[x,y]\mapsto [y]$ gives $(\sph^6\times \sph^7)/\su{2}$ the structure of a linear sphere bundle over $\sph^4$ with fiber $\sph^6$.  A simple application of the Gysin sequence then establishes that $(\sph^6\times \sph^7)/\su{2}$ has the same cohomology groups as $\sph^6\times \sph^4$.  The product structure is determined by Poincaré duality.

For the $\sg^1$-quotient of $\sph^6\times \sph^7$ we argue as follows.  The $\sg^1$-action on $\sph^7$ is the Hopf action, so $(\sph^6\times \sph^7)/\sg^1$ is naturally an $\sph^6$-bundle over $\CP^3$.  The Serre spectral sequence of this bundle satisfies the conditions of the Leray-Hirsch theorem, so the cohomology of $(\sph^6\times \sph^7)/\sg^1$ is isomorphic to that of $\sph^6\times \CP^3$ as an $H^\ast(\CP^3)$-module. In order to show that $H^\ast((\sph^6\times \sph^7)/\sg^1)$ is in fact isomorphic to $H^\ast(\sph^6\times \CP^3)$, it remains to find an element $x\in H^6((\sph^6\times \sph^7)/\sg^1)$ with $x^2=0$ and which pulls back to a generator of $H^6(\sph^6)$ via the inclusion $\sph^6\rightarrow (\sph^6\times \sph^7)/\sg^1$.

For this, we note that the natural map $(\sph^6\times \sph^7)/\sg^1\rightarrow (\sph^6\times \sph^7)/\su{2}$ is a fiber bundle with fiber $\sph^2$.  The differentials in the corresponding Serre spectral sequence all vanish for trivial reasons.  It follows that the composition $$H^6((\sph^6\times \sph^7)/\su{2})\rightarrow H^6((\sph^6\times \sph^7)/\sg^1)\rightarrow H^6(\sph^6)$$ is an isomorphism. Let $x\in H^6((\sph^6\times \sph^7)/\sg^1)$ denote the image of a generator of $H^6((\sph^6\times \sph^7)/\su{2}$).  Notice that $x$ pulls back to a generator of $H^6(\sph^6)$. Because the square of any element in $H^6((\sph^6\times \sph^7)/\su{2})$ is $0$ (since it is an element in $H^{12}((\sph^6\times \sph^7)/\su{2}) = 0$), it follows that $x^2 = 0$.\qedhere
\end{proof}

\begin{remark}\label{rem:homotopy_types}It follows from Theorem \ref{thm:topology} that none of the examples of Theorem \ref{THM:Main_THM} is rationally homotopy equivalent to any known example with $\sec > 0$.  Indeed, the rationally cohomology rings of all our examples require two generators, so are distinct from the rational cohomology rings of spheres and projective spaces.  In dimensions $10,11,12, 13 $ and $14$, the only other known examples with $\sec > 0$ are the $12$-dimensional homogeneous space $\sp{3}/ (\sp{1})^3$ and the infinite family of $13$-dimensional Bazaikin spaces.  Using the long exact sequence in homotopy groups associated with the fibration $\sp{1}^3\rightarrow \sp{3} \rightarrow \sp{3}/ (\sp{1}^3)$, one easily sees that $\sp{3}/(\sp{1}^3)$ is $3$-connected.  On the other hand, $H^2((\sph^6\times \sph^7)/ \sg^1; \mathbb{Q}) \cong \mathbb{Q}\neq 0$.  Finally, as Bazaikin showed \cite{Baz}, every Bazaikin space has trivial $H^7$, while $H^7(\sph^6\times \sph^7;\mathbb{Q})$ and $H^7( (\sph^7\times \sph^7)/\sg^1);\mathbb{Q})$ are both non-trivial.
\end{remark}

Next we consider the Pontryagin classes $p_i\in H^{4i}$ of the quotients, where $i\geq 1$. Before computing the first Pontryagin class $p_1$, we note the following:

\begin{proposition}\label{PROP:higher_pont_classes} Suppose $k\in \{6,7\}$. For all $\sg^1$ and $\su{2}$-quotients of $\sph^k\times \sph^7$ the Pontryagin classes $p_i$ with $i>1$ vanish.
\end{proposition}

\begin{proof}
For $i > 3$, the Pontryagin classes vanish because the relevant cohomology groups all vanish.  Moreover, except for the circle quotient of $\sph^6\times \sph^7$, all the spaces under consideration have vanishing $H^8$ and $H^{12}$ by Theorem~\ref{thm:topology}, so $p_2$ and $p_3$ vanish for these examples as well.

For $(\sph^6\times \sph^7)/\sg^1$, we argue as follows.  Recall from the proof of Theorem~\ref{thm:topology} that $(\sph^6\times \sph^7)/\sg^1$ has the structure of a linear $\sph^6$-bundle over $\CP^3$.  We also recall that for a linear sphere bundle $\sph^k\rightarrow E\xrightarrow{\pi} B$, we have a splitting $TE\oplus 1 = \pi^\ast(TB)\oplus \pi^\ast(\overline{E})$, where $\overline{E}\rightarrow B$ is the rank $k+1$ vector bundle whose sphere bundle is $E\rightarrow B$, see, for example, \cite[p.~383]{Geiges}. This implies that all stable characteristic classes of $E$ lie in $\pi^\ast(H^\ast(B))$. In our case we have $B=\CP^3$, so all stable characteristic classes of $(\sph^6\times \sph^7)/\sg^1$ vanish in degrees above~$6$.
\end{proof}

We now turn our attention to the first Pontryagin class. Note that the cohomology group $H^4$ is isomorphic to $\mathbb Z$ for all quotients. In the statements below we shall only state $p_1$ up to a sign. The interested reader can find, in the proof of the corresponding result, the precise value of $p_1$ in terms of some generator of $H^4$.

We start with the $\sg^1$-quotients of $\sph^7\times \sph^7$. Recall from Subsection~\ref{SS:circle_actions} that a circle $\sg^1< \T^4$ is determined by four parameters $n_i\in\mathbb Z$. Moreover, the $n_i$ define the integers $\ell_i$, $r_i$ from Equation~\eqref{EQ:l_i_r_i}, which are used in Theorem~\ref{thm:free} to determine if the action is free.

\begin{theorem}\label{thm:p1} Assume the $n_i$ are chosen so that the $\sg^1$-action on $\sph^7\times \sph^7$ is free.  Then 
$$p_1( (\sph^7\times \sph^7)/\sg^1)=\pm \sum_{i=1}^4 \ell_i^2 + r_i^2.$$
\end{theorem}

Before proving this, we note the following corollaries. First, observe that infinitely many different choices of free acting circles $\sg^1< \T^4< \spin{8}$ (whose existence was proved in Theorem~\ref{thm:free}) yield infinitely many values of $p_1$. Since rational Pontryagin classes are homeomorphism invariants \cite{Nov}, we get:

\begin{corollary}
There are infinitely many homeomorphism types of spaces of the form $(\sph^7\times \sph^7)/\sg^1$, with $\sg^1<\spin{8}$.
\end{corollary}

Second, we compare the homotopy type of the spaces $(\sph^7\times \sph^7)/\sg^1$ with that of $\sph^7\times \CP^3$.

\begin{corollary}\label{COR:s7xs7_s1} If $\sg^1<\spin{8}$ acts freely on $\sph^7\times \sph^7$, then the first Pontryagin class is divisible by $8$.  In particular, $(\sph^7\times \sph^7)/\sg^1$ is not homotopy equivalent to $\sph^7\times \CP^3$.
\end{corollary}

\begin{proof}
From Equation~\eqref{EQ:l_i_r_i} we get
$$\sum \ell_i^2 + r_i^2 = 4\left(n_1^2 + n_2^2 + n_3^2 + n_4^2 + n_1 n_3 + n_1 n_4 + n_2n_3 -n_2 n_4\right).$$
Thus, we need only demonstrate that the factor in parenthesis is even. This is the case if and only if $$n_1^2 + n_2^2 + n_3^2 + n_4^2 + (n_1 + n_2)(n_3+n_4)$$ is even, since adding the even number $2n_2n_4$ does not change the parity. From here, it is clearly sufficient to show that $n_1\equiv n_2 \pmod{2}$ and $n_3\equiv n_4\pmod{2}$.

To that end, let us first assume for a contradiction that $n_1$ and $n_2$ have opposite parities with, say, $n_1$ even. Since the action is free, Theorem~\ref{thm:free} implies that $\gcd(\ell_1 ,r_3) = \gcd(\ell_1, r_4) =  1$, which in turn imply that both $n_3$ and $n_4$ are even.  But then $\gcd(\ell_1, r_1) = \gcd(n_1,n_3)$ must be a multiple of $2$, giving a contradiction with Theorem~\ref{thm:free}.  An analogous argument shows that we cannot have $n_2$ even and $n_1$ odd, so $n_1\equiv n_2\pmod{2}$.  Finally, an analogous argument establishes that $n_3\equiv n_4\pmod{2}$.

Finally, we recall that $p_1 \pmod{24}$ is a homotopy invariant \cite{AH}.  Since $\sph^7$ is stably parallelizable and $p_1(\CP^3)$ is $4$ times a generator \cite[Example 15.6]{MS}, it follows that  $p_1(\sph^7\times \CP^3)$ is $4$ times a generator. Thus, $\sph^7\times \CP^3$ is homotopically distinct from $(\sph^7\times \sph^7)/\sg^1$, independently of the chosen $\sg^1< \spin{8}$.
\end{proof}

We now prove Theorem \ref{thm:p1}.

\begin{proof}[Proof of Theorem \ref{thm:p1}]

We follow the approach in \cite[Sections 2.3 and 2.4]{DV17}. By \cite[Proposition~2.13]{DV17}, each quotient of $\sph^7\times \sph^7$ under a linear circle action is diffeomorphic to the quotient of $(\u{4}\times \u{4})/(\u{3}\times \u{3})$ under a (free) biquotient action of $\sg^1$. In order to identify these actions, we switch to complex notation using $w_i\in \sg^1\subset \mathbb{C}$ to denote $R(\alpha_i)$. Substituting these in Equation~\eqref{EQ:maximal_torus}, we find that the torus $\T^4$ from Proposition~\ref{prop:maxtorus} has the form
$$\T^4 = \{ (\diag( w_1, w_1 w_3 w_4, w_2 w_3 \overline{w}_4, w_2), \diag(w_3, w_4, \overline{w}_1 w_2 \overline{w}_4, w_1 w_2 w_3))\} < \u{4}\times \u{4}.$$
As we explained in Subsection~\ref{SS:circle_actions}, a general $\sg^1< \T^4$ is obtained by setting $w_i =u^{n_i}$ with $n_i\in \mathbb{Z}$ and $u\in \sg^1$. More precisely, any circle is of the form
\begin{equation}\label{EQ:circle_f1}
u\mapsto (\diag( u^{\ell_1},u^{\ell_2},u^{\ell_3}, u^{\ell_4}), \diag(u^{r_1}, u^{r_2},u^{r_3}, u^{r_4})),
\end{equation}
where $\ell_i$, $r_i$ are the integers from Equation~\eqref{EQ:l_i_r_i} defined in terms of the parameters~$n_i$.

Let $G = \u{4}\times \u{4}$ and $H = \sg^1\times (\u{3}\times \u{3})$. The biquotient action of $H$ is given by a homomorphism $f = (f_1,f_2)\colon H\rightarrow G^2$, where $f_1$ denotes the inclusion $\sg^1< \T^4< G$ induced by Equation~\eqref{EQ:circle_f1} and $f_2\colon\u{3}\times \u{3}\to \u{4}\times \u{4}$ is the standard inclusion.

Denote the standard maximal tori of $H$ and $G\times G$ by $T_H$ and $T_{G\times G}$, respectively, and note that $T_H=\sg^1\times T_{\u{3}\times\u{3}}$, where $T_{\u{3}\times\u{3}}$ is the standard maximal torus of $\u{3}\times\u{3}$. Let $\bar{f}_1$ and $\bar{f}_2$ denote the restrictions of $f_1$ and $f_2$ to the maximal tori of their respective domains; in particular $\bar{f}_1=f_1$.  And let $\bar f = (\bar f_1,\bar f_2)\colon T_H\rightarrow T_{G\times G}$, so that $\bar f$ is likewise the restriction of $f$ to $T_H$. Note that $\bar f_1$ is given by Equation~\eqref{EQ:circle_f1} and $\bar f_2$ is given by 
$$(\diag(w_1,w_2,w_3),\diag(w_4,w_5,w_6))\mapsto (\diag(w_1,w_2,w_3,0),\diag(w_4,w_5,w_6,0)).$$
Let $B\bar f=(B\bar f_1,B\bar f_2)\colon BT_H\to BT_{G\times G}$ be the induced map between the classifying spaces and $B\bar f^\ast=(B\bar f_1^\ast,B\bar f_2^\ast)\colon H^\ast(BT_{G\times G})\to H^\ast(BT_H)$ the pullback in cohomology. The cohomology rings of $BT_H$ and $BT_G$ can be described using canonical elements $x_i,y_i,u_i,v_i,z$ of degree $2$ as:
\begin{align*}
H^\ast(BT_H) &\cong \mathbb{Z}[z,u_1,u_2,u_3, v_1,v_2,v_3],\\
H^\ast(BT_G) &\cong \mathbb{Z}[x_1,x_2,x_3,x_4,y_1,y_2,y_3,y_4].
\end{align*}
Hence the cohomology ring of $BT_{G\times G} \simeq BT_G\times BT_G$ is isomorphic to  
$$
H^\ast(BT_{G\times G}) \cong \mathbb{Z}[x_i\otimes 1, y_i\otimes 1, 1\otimes x_i, 1\otimes y_i], \text{ with } i\in \{1,2,3,4\}.
$$
The procedure to compute $B\bar f^\ast$ from $\bar f$ is detailed in \cite[Proof of Theorem~2.9]{DV17}. In our case we obtain:
$$\begin{matrix} 
B\bar f_1^\ast(x_i\otimes 1)   =   \ell_i z & (1 \leq i \leq 4), &&&  B\bar f_1^\ast(y_i\otimes 1)  = r_i z & (1\leq i \leq 4), \\ 
B\bar f_2^\ast(1\otimes x_i)  =  u_i & (1\leq i\leq 3), &&& B\bar f_2^\ast(1\otimes y_i)   =  v_i & (1\leq i \leq 3),\\ 
B\bar f_2^\ast(1\otimes x_4)  =  0, & &&& B\bar f_2^\ast(1\otimes y_4) =  0. & \end{matrix}$$
In addition, we have $$B\bar f_1^\ast(1\otimes x_i) = B\bar f_1^\ast (1\otimes y_i) = B\bar f_2^\ast ( x_i\otimes 1) = B\bar f_2^\ast(y_i\otimes 1) = 0$$ for all $i\in \{1,2,3,4\}$. 

Now we consider the full homomorphism $f = (f_1,f_2)\colon H\rightarrow G^2$ and the induced map $Bf^\ast\colon H^\ast(B(G\times G))\rightarrow H^\ast(BH)$. Let us describe the cohomology rings. We write $x\otimes 1$ to refer to the ordered $4$-tuple $(x_1\otimes 1, x_2\otimes 1, x_3\otimes 1, x_4\otimes 1)$ and similarly define $1\otimes x$, $y\otimes 1$, and $1\otimes y$.  We define $u$ to denote the ordered triple $(u_1, u_2, u_3)$, with an analogous definition for $v$. Finally, we set $\ell = (\ell_1, \ell_2, \ell_3, \ell_4)$ and $r = (r_1,r_2,r_3,r_4)$. Then we may identify 
$$H^\ast(B(G\times G)) = \mathbb{Z}[\sigma_j(x\otimes 1), \sigma_j(y\otimes 1), \sigma_j(1\otimes x), \sigma_j(1\otimes y)] \text{ where } j\in \{1,2,3,4\},$$ 
and $$H^\ast(BH) = \mathbb{Z}[z, \sigma_1(u), \sigma_2(u), \sigma_3(u), \sigma_1(v), \sigma_2(v), \sigma_3(v)],$$ 
where $\sigma_j$ refers to the $j$-th elementary symmetric polynomial in either $4$ or $3$ variables, as appropriate. Now we can compute $Bf^\ast\colon H^\ast(B(G\times G))\rightarrow H^\ast(BH)$. Using the expressions above for $B\bar f^\ast=(B\bar f_1^\ast,B\bar f_2^\ast)$ we find that:
$$\begin{matrix} Bf_1^\ast(\sigma_j(x\otimes 1))& = & \sigma_j(\ell) z^j,  & \hspace{.5 in} & Bf_1^\ast(\sigma_j(y\otimes 1)) & = & \sigma_j(r)z^j, \\ Bf_2^\ast(\sigma_j(1\otimes x)) & = &\sigma_j(u), & \hspace{.5 in} & Bf_2^\ast(\sigma_j(1\otimes y)) & = & \sigma_j(v), \\ Bf_2^\ast(\sigma_4(1\otimes x))  &= & 0, &  \hspace{.5 in} & Bf_2^\ast(\sigma_4(1\otimes y))& = & 0. &   \end{matrix}$$
In addition, we have 
$$ Bf_1^\ast(\sigma_j(1\otimes x)) = Bf_1^\ast(\sigma_j(1\otimes y)) = Bf_2^\ast(\sigma_j(x\otimes 1)) = Bf_2^\ast(\sigma_j(y\otimes 1)) = 0 $$
for all $j\in \{1,2,3,4\}$.

Recall that the $H$-principal bundle $H\to G\to G\bq H$ given by the biquotient action induces a fibration $G\rightarrow G\bq H\xrightarrow{\phi_H} BH$.  Consider the corresponding Serre spectral sequence.  From \cite[Theorem 1]{Es3}, we know that the differentials are totally transgressive, with images generated by $Bf_1^\ast(\sigma_j(x\otimes 1)) - Bf_2^\ast(\sigma_j(1\otimes x))$ and $Bf_1^\ast(\sigma_j(y\otimes 1)) - Bf_2^\ast(\sigma_j(1\otimes y))$.  Having calculated these, we may therefore carry out the spectral sequence calculation.
Doing this, we find that $H^4(G\bq H)\cong \mathbb{Z}$, generated by $\phi_H^\ast(z^2)$.  In addition, we note that $\phi_H^\ast(\sigma_j(u)) = \phi_H^\ast(\sigma_j(\ell)z^j)$ and $\phi_H^\ast(\sigma_j(v)) = \phi_H^\ast(\sigma_j(r) z^j)$.

Now we are ready to calculate $p_1$. The recipe for computing the total Pontryagin class can be found in \cite[Theorem 2.11]{DV17}, and it specializes to the formula \cite[(2.1)]{DV17} in the case of $p_1$. For that, after identifying $1\otimes x_i$ with $x_i$, recall that the positive roots of $\u{4}$ are $x_i - x_j$ with $1 \leq i < j \leq 4$, those of $\u{3}$ are $u_i - u_j$ (or $v_i-v_j$, for the other factor) with $1\leq i < j \leq 3$, and $\sg^1$ has no positive roots. We find:
\begin{align*} p_1 &= \phi_H^\ast\left( \sum_{1\leq i< j\leq 4}\left[ Bf_2^\ast(x_i - x_j)^2 + Bf_2^\ast(y_i-y_j)^2\right] - \sum_{1\leq i< j\leq 3} \left[(u_i-u_j)^2  + (v_i-v_j)^2\right]  \right)\\ &= \phi_H^\ast\left( \sum_{i=1}^4 Bf_2^\ast(x_i-x_4)^2 + Bf_2^\ast(y_i-y_4)^2\right)= \phi_H^\ast \left(  \sum_{i=1}^3 u_i^2 + v_i^2\right)\\ 
&= \phi_H^\ast( \sigma_1(u)^2 - 2\sigma_2(u) + \sigma_1(v)^2 - 2\sigma_2(v))\\ 
&= \phi_H^\ast((\sigma_1(\ell)z)^2 - 2\sigma_2(\ell)z^2 + (\sigma_1(r) z)^2 -2\sigma_2(r)z^2) \\
&= \phi_H^\ast\left( \sum_{i=1}^4 \ell_i^2 z^2 + \sum r_i^2 z^2\right)= \left(\sum_{i=1}^4 \ell_i^2 + r_i^2\right)\phi_H^\ast (z^2).\qedhere
\end{align*}
\end{proof}

We now compute the first Pontryagin class of the  $\su{2}$-quotient of $\sph^7\times \sph^7$ determined in Theorem~\ref{thm:summary}.

\begin{theorem}\label{thm:su2quotients}  For the $\su{2}$-quotient of $\sph^7\times \sph^7$ we have $p_1=\pm 4$. In particular it is not homotopy equivalent to $\sph^7\times \sph^4$.
\end{theorem}

\begin{proof}

This calculation is almost identical to that for the circle quotients, so we just indicate the differences.  First, one uses $H = \su{2}\times (\u{3}\times \u{3})$ instead.  Then 
$$H^\ast(BH) \cong \mathbb{Z}[z^2, \sigma_1(u), \sigma_2(u), \sigma_3(u), \sigma_1(v), \sigma_2(v), \sigma_3(v)].$$  
The spectral sequence calculation works identically: one finds that $H^4(G\bq H)\cong \mathbb{Z}$ is generated by $\phi_H^\ast(z^2)$, and the same relations hold regarding $\phi_H^\ast(\sigma_j(u))$ and $\phi_H^\ast(\sigma_j(v))$.  In the calculation of $p_1$, $H$ has an additional positive root given by $2z$, which means there is an extra $-4z^2$ occurring in the sum.

Next we need to implement the values of $\ell_i,r_i$ into the computation. Recall from the proof of Proposition~\ref{prop:step1} and Proposition~\ref{prop:two_actions_equiv} that we have two descriptions of the action. On the one hand, it corresponds to the row $3 + 1 + 1  +1 +1 +1$ of Table \ref{table:rep}, from where we get $|\ell_i| = |r_j| = 1$ for all $i,j\in \{1,2,3,4\}$. On the other hand, for the row $2+2+2+2$, we get, up to permutation and sign, that $\ell_1 = \ell_2 = \ell_3 = 0$ and $\ell_4 = 2$, while $|r_j| = 1$ for all $j$. Putting this all together and using either of the values of $\ell_i,r_i$, we see that the first Pontryagin class is $4\phi_H^\ast(z^2)$.

Finally, as $\sph^7$ and $\sph^4$ are stably parallelizable, so is $\sph^7\times \sph^4$.  It follows that $p_1(\sph^7\times \sph^4)$ is trivial.  Since $p_1 \pmod{24}$ is a homotopy invariant \cite{AH}, we see that the $\su{2}$-quotient of $\sph^7\times \sph^7$ and $\sph^7\times \sph^4$ have different homotopy types.
\end{proof}

\begin{remark}\label{rem:kerin}
In \cite{Kerin11}, Kerin considers the biquotients $\sg^1\backslash \so{8}/\gg$ and $\su{2}\backslash \so{8}/\gg$. In our notation, these spaces correspond to $(\sph^7\times \sph^7)/\sg^1$ with $|\ell_i| = |r_j| = 1$, for all $1\leq i,j\leq 4$, and $(\sph^7\times \sph^7)/\su{2}$, where $\su{2}$ extends the previous circle action. Kerin shows that these two spaces admit Riemannian metrics of almost positive curvature. Moreover, he proves that $p_1((\sph^7\times \sph^7)/\sg^1)=\pm 8\in H^4$ and $p_1((\sph^7\times \sph^7)/\su{2})=\pm 2^n$ for some integer $n\geq 0$, see \cite[Theorem~6.10 and Remark~6.6]{Kerin11}. Thus, our topological computations are consistent with his.
\end{remark}

To conclude, we need only compute the first Pontryagin class of the $\sg^1$ and $\su{2}$-quotients of $\sph^6\times \sph^7$.  

\begin{theorem}\label{thm:p1s6}  We have $p_1((\sph^6\times \sph^7)/\sg^1) =  \pm 8 $ and $p_1((\sph^6\times \sph^7)/\su{2}) =  \pm 4$.  In particular, $(\sph^6\times \sph^7)/\sg^1$ and $(\sph^6\times \sph^7)/\su{2}$ are not homotopy equivalent to $\sph^6\times \CP^3$ and $\sph^6\times \sph^4$, respectively.
\end{theorem}
\begin{proof} 
We will only prove this for the $\sg^1$-quotient because the $\su{2}$ case works almost identically: one just has to replace $\sg^1$ by $\su{2}$ everywhere and $\CP^3$ by $\sph^4$ as the base of the corresponding bundle.

As shown in the proof of Theorem~\ref{thm:free}, the free $\sg^1$-action is equivalent to the case where $(\ell_1, \ell_2, \ell_3, \ell_4) = (0,0,0,2)$ and $(r_1,r_2,r_3,r_4) = ( 1,1,1,1)$, up to permutation and sign changes.  In particular, the $\sg^1$-action on $\sph^6\times \sph^7$ extends to the free $\sg^1$-action on $\sph^7\times \sph^7$ for the same values of $\ell_i,r_i$.  This implies that $(\sph^6\times \sph^7)/\sg^1$ is a codimension one submanifold of $(\sph^7\times \sph^7)/\sg^1$. Write $i\colon (\sph^6\times \sph^7)/\sg^1\rightarrow (\sph^7\times \sph^7)/\sg^1$ for the inclusion.

Because $(\sph^6\times \sph^7)/\sg^1$ is simply connected, the normal bundle is trivial.  Thus, we find that $p_1((\sph^6\times \sph^7)/\sg^1) = i^\ast(p_1((\sph^7\times \sph^7)/\sg^1))$. By Theorem~\ref{thm:p1} and its proof, $p_1((\sph^7\times \sph^7)/\sg^1) = 8\phi_H(z^2)$, so it remains to see that $i^\ast$ is an isomorphism on $H^4$.

To see this, simply note that both $(\sph^6\times \sph^7)/\sg^1$ and $(\sph^7\times \sph^7)/\sg^1$ are naturally bundles over $\CP^3$ and that $i$ is compatible with the projection maps.  It now follows easily from comparing the corresponding Gysin sequences that $i^\ast$ is an isomorphism on $H^4$, completing the computation of the Pontryagin classes.

Finally, as in the proof of Corollary \ref{COR:s7xs7_s1}, we note that $p_1(\sph^6\times \CP^3)$ is $4$ times a generator.  Similarly, as in the proof of Theorem \ref{thm:su2quotients}, we see $p_1(\sph^6\times \sph^4)$ is trivial.  Since $p_1\pmod{24}$ is a homotopy invariant \cite{AH}, it follows that $(\sph^6\times \sph^7)/\sg^1$ and $(\sph^6\times \sph^7)/\su{2}$ are not homotopy equivalent to $\sph^6\times \CP^3$ and $\sph^6\times \sph^4$, respectively. 
\end{proof}

\section{Non-simply connected examples}\label{SEC:non_simply}

In this section, we find finite isometric quotients of $\sph^k\times\sph^7$ with respect to the $\Ric_2>0$ metrics that we have constructed, where $k\in\{6,7\}$. Being isometric quotients, all of them inherit metrics of $\Ric_2 > 0$.  We will see that they cannot admit metrics with $\sec > 0$. We will also include a proof of a similar result for finite isometric quotients of $\sph^k\times\sph^\ell$ with a $\Ric_2>0$ metric, for $k,\ell\in\{2,3\}$.

For a positive integer $d$  and a positive odd integer $n$, we let $L^n_d$ denote the homogeneous lens space $\sph^n/ \mathbb{Z}_d$ where $\mathbb{Z}_d$ acts as a subaction of the circle Hopf action.

\begin{theorem}\label{thm:rp7}   For any positive  integer $d$, there is an $\RP^7$-bundle over $L^7_d$ whose total space $M_d$ admits a $\Ric_2 > 0$ metric and has fundamental group $\pi_1(M_d)\cong\mathbb Z_2\times\mathbb Z_d$.  When $d=1,2$ one can take this bundle to be trivial so that $\RP^7\times \sph^7$ and $\RP^7\times \RP^7$ admit metrics of $\Ric_2>0$.   None of the spaces $M_d$ can admit a metric of $\sec > 0$.
\end{theorem}

\begin{proof} 
Consider the following circle in $\so{8}\times\so{8}$:
\begin{equation}\label{eq:free_circle}
\bigl(\diag((R(0),R(0),R(0),R(2\theta)),\diag (R(-\theta),R(\theta),R(\theta),R(\theta))\bigr),\qquad\theta\in [0,2\pi).
\end{equation}
Observe that this is actually a circle in $\spin{8}<\so{8}\times\so{8}$; for that note that it is the maximal torus of the free isometric action of $\su{2}<\spin{8}$ given by Proposition \ref{prop:step1}, and more precisely it is the circle given by the parameters $(n_1,n_2,n_3,n_4) = (0, 2,-1,1)$ as in Subsection~\ref{SS:circle_actions}. For this circle we have $\ell_i=0$ for $i\neq 4$, $\ell_4=2$, $r_1=-1$ and $r_i=1$ for the remaining $i$. By Theorem~\ref{thm:free}, this circle acts freely (and isometrically) on $\sph^7\times \sph^7$. For each positive integer $d$, this circle has a subgroup $\mathbb{Z}_d$ defined by allowing $\theta$ to take values in the set $\left\{\frac{2\pi k}{d}: k\in\mathbb Z\right\}$.

In addition, we observe that $(-I_8, I_8, -I_8)$ satisfies Equation~\eqref{eqn:triality}, so $\mathbb{Z}_2 = \{ (\pm I_8, I_8)\}$ is a subgroup of $\spin{8}$, hence it acts isometrically on $\sph^7\times \sph^7$.  The actions by $\mathbb{Z}_2$ and by the above $\mathbb{Z}_d$ commute, so we obtain an isometric action by $\mathbb{Z}_2\times \mathbb{Z}_d$ on $\sph^7\times \sph^7$.  This action is clearly free, so the quotient space $M_d=(\sph^7\times \sph^7)/(\mathbb{Z}_2\times \mathbb{Z}_d)$ is a manifold $M_d$ with fundamental group $\pi_1(M_d)\cong\mathbb Z_2\times\mathbb Z_d$ and inheriting a metric of $\Ric_2 > 0$.

By first quotienting out the $\mathbb{Z}_2$-action, one can view the $\mathbb{Z}_d$-action as an action on $\RP^7\times \sph^7$. Since the action of the circle on the second factor is precisely the Hopf action, the subaction by $\mathbb{Z}_d$ on the second factor yields precisely the homogeneous lens space $L^7_d$. It follows that the quotient space $(\RP^7\times \sph^7)/\mathbb{Z}_d$ has the structure of an $\RP^7$-bundle over $L^7_d$, as claimed.  When $d=1,2$ the action by $\mathbb{Z}_d$ on the $\RP^7$-factor is trivial, so it follows that $M_1$ and $M_2$ are diffeomorphic to $\RP^7\times \sph^7$ and $\RP^7\times \RP^7$, respectively.

We finally note that all the spaces $M_d$ are orientable because $\mathbb{Z}_2\times \mathbb{Z}_d<\spin{8}$ is a subgroup of the identity component of the diffeomorphism group of $\sph^7\times \sph^7$. Since the $M_d$ have non-trivial fundamental group, they cannot admit a metric of positive sectional curvature by Synge's theorem.
\end{proof}

We can now prove the corresponding result for quotients of $\sph^6\times \sph^7$.

\begin{theorem}\label{thm:rp6} For any positive integer $d$, there is an $\RP^6$-bundle over $L^7_d$  whose total space $N_d$ admits a metric of $\Ric_2>0$ and has fundamental group $\pi_1(N_d)\cong\mathbb Z_2\times\mathbb Z_d$. When $d=1,2$ one can take this bundle to be trivial so that $\RP^6\times \sph^7$ and $\RP^6 \times \RP^7$ admit a metric of $\Ric_2>0$.  None of the spaces $N_d$ can admit a metric of $\sec > 0$.
\end{theorem}

\begin{proof} 
From Proposition~\ref{prop:maxtorus} we know that the circle in Equation~\eqref{eq:free_circle} from the proof of Theorem~\ref{thm:rp7} lies in $\spin{7}$. In particular, for each positive integer $d$, the subaction by $\mathbb{Z}_d$ on $\sph^6\times \sph^7$ is by isometries and preserves orientation.

However, the antipodal map of $\sph^6$ reverses orientation, so the $\mathbb{Z}_2$-action generated by the antipodal map $(x,y)\mapsto (-x,y)$ of $\sph^6\times \sph^7$ cannot be a subgroup of $\spin{7}$, which by Theorem~\ref{thm:isometry_group} is the connected component containing the identity of the isometry group of our $\Ric_2>0$ metrics on $\sph^6\times \sph^7$. We will nonetheless show that this $\mathbb{Z}_2$-action is isometric. Believing this momentarily, one can then argue exactly as in the proof of Theorem~\ref{thm:rp7} to show that each quotient $N_d=(\sph^6\times \sph^7)/(\mathbb{Z}_2\times \mathbb{Z}_d)$, which is an $\RP^6$-bundle over $L^7_d$, has a metric of $\Ric_2>0$ and fundamental group $\pi_1(N_d)\cong\mathbb Z_2\times\mathbb Z_d$. These spaces are non-orientable, so cannot admit a metric of $\sec>0$ by Synge's Theorem. Moreover, when $d=1,2$ the bundle is trivial so the corresponding total spaces $N_1$ and $N_2$ are diffeomorphic to $\RP^6\times \sph^7$ and $\RP^6\times \RP^7$, respectively.

It remains to see that the antipodal map $(x,y)\mapsto (-x,y)$ is an isometry.  We argue as follows.  Let $\sigma = \diag(1,-1,1,-1,1,-1,1,-1)\in \so{8}$, and note that $(\sigma,\sigma,\sigma)$ satisfies Equation~\eqref{eqn:triality}. To prove the latter, one can consider the subsets $X:=\{1,j,\ell, j\ell \} $ and $Y:=\{i,k,i\ell, k\ell\}$ of the standard basis of $\mathbb{O}$ given in Subsection~\ref{subsec:difeo_product_spheres}, so $\sigma$ equals the identity on $X$ and minus the identity on $Y$. Moreover, from Table~\ref{table:cayleymult} we see that $X\cdot X\subset \pm X$, $Y\cdot Y\subset \pm X$ and $X\cdot Y\subset \pm Y$. Now it is immediate to check that if $x_1,x_2 \in X$ then $\sigma(x_1)\sigma(x_2) = x_1 x_2 = \sigma(x_1x_2)$ (and similarly for pairs $x\in X,y\in Y$ and $y_1, y_2\in Y$).

Because $(\sigma,\sigma,\sigma)$ satisfies Equation~\eqref{eqn:triality}, it follows that $(\sigma,\sigma)\in \gg$ (see Subsection~\ref{subsec:difeo_product_spheres}). Under the transitive $\spin{7}$-action on $\sph^6\times \sph^7$ we get $(\sigma,\sigma)(i,1) = (-i,1)$.  Recalling that $\su{3}$ can be taken to be the stabilizer at $(i,1)$ of the $\spin{7}$-action, we see that $(\sigma,\sigma)(A,B)(\sigma,\sigma)^{-1} (i,1) = (i,1)$ for all $(A,B)\in\su{3}$. Altogether, we conclude that $(\sigma,\sigma)\in N_{\gg}(\su{3})\setminus \su{3}$.

Viewing $\sph^6\times \sph^7 = \spin{7}/\su{3}$, we consider the self-map $R_\sigma\colon\spin{7}/\su{3}\rightarrow \spin{7}/\su{3}$ given by $R_\sigma((g_1,g_2)\su{3}) =  (g_1\sigma, g_2\sigma)\su{3}$. Since $(\sigma,\sigma)\in N_{\gg}(\su{3})\setminus \su{3}$, it follows that $R_\sigma$ is a well-defined map and an isometry with respect to the metric $q_t$ on $\spin{7}/\su{3}$ (see the beginning of Subsection~\ref{SS:submersion_metrics}).

Recall from Proposition \ref{prop:diffeo} that the diffeomorphism $f\colon\spin{7}/\su{3}\to \sph^6\times \sph^7$ is given by $f((g_1,g_2) \su{3})= (g_1,g_2)(i,1) = (g_1 i, g_2)$.  We compute that 
\begin{align*}
f(R_\sigma( (g_1,g_2) \su{3})) &= f( (g_1,g_2)(\sigma,\sigma) \su{3}) = (g_1,g_2)(\sigma,\sigma)(i,1)\\ 
&= (g_1,g_2)(-i,1) = (-g_1 i, g_2).
\end{align*}
In other words, under the diffeomorphism $f$, the isometry $R_\sigma$ corresponds to the antipodal map $(x,y)\mapsto (-x,y)$.  This completes the proof.
\end{proof}

\begin{remark}
Any cover of one of the spaces $M_d$ or $N_d$ from Theorems~\ref{thm:rp7} and \ref{thm:rp6} also admits a Riemannian metric with $\Ric_2 > 0$.  Except for $\sph^7\times \sph^7$, Synge's Theorem applies to all covers of any $M_d$, so none of them can admit a Riemannian metric with $\sec > 0$.  On the other hand, for each $N_d$, Synge's Theorem does not obstruct the orientation covering from admitting a metric with $\sec > 0$.
\end{remark}
 
\begin{remark}\label{rem:nonsimply}
We finally note that the metrics we find on $\RP^6\times \RP^7$ and $\RP^7\times \RP^7$ are homogeneous; more precisely they carry transitive (almost effective) isometric actions by $\spin{7}$ and $\spin{8}$, respectively. To see this, simply observe that we showed in the proofs of Theorems~\ref{thm:rp7} and \ref{thm:rp6} that the antipodal maps on each factor of $\sph^6\times \sph^7$ and $\sph^7\times \sph^7$ are isometric.  Since the actions by $\spin{7}$ and $\spin{8}$ are linear in each factor, they commute with the antipodal maps and hence descend to transitive actions on $\RP^6\times \RP^7$ and $\RP^7\times \RP^7$, respectively.
\end{remark}

We conclude this section by providing isometric quotients of Wilking's metric on $\sph^n\times\sph^m$, with $n,m\in\{2,3\}$. The result is likely well-known, but we include it for completeness.

	\begin{proposition}\label{prop:quotients_S2S3}
	  For any positive integers $d,d'$, the spaces $L_d^3 \times L_{d'}^3$, $\mathbb{R}P^2\times L^3_d$, and $\mathbb{R}P^2\times \mathbb{R}P^2$ admit metrics of $\Ric_2 > 0$.  None of these spaces can admit a metric of $\sec > 0$.  Both $\RP^2\times \RP^3$ and $\RP^3\times \RP^3$ admit homogeneous metrics with $\Ric_2 > 0$, but $\RP^2\times \RP^2$ does not.
	\end{proposition}
	
	\begin{proof} Wilking's $\Ric_2>0$ metric on $\sph^3\times \sph^3$ is invariant under left multiplication by $\sph^3\times \sph^3$ and right multiplication by $\Delta \sph^3$, see \cite[Remark 1.3]{DGM}. Observe that the action given by left multiplication by $\sph^3\times \sph^3$ is free, and hence, so is the action by any subgroup.  In particular, if $H_1,H_2$ are two closed subgroups of $\sph^3$, then $(H_1\backslash \sph^3)\times (H_2 \backslash \sph^3)$ admits a metric of $\Ric_2 >0 $.  Taking each $H_i\in \{ \mathbb{Z}_d, \sg^1 \cup j\sg^1\} $ gives all of the examples from the first sentence of the proposition.  The second sentence follows easily from Synge's theorem since all the $6$-dimensional examples are orientable, all the $5$-dimensional examples are non-orientable, and $\RP^2\times \RP^2$ has a fundamental group of order $4$.
		
	We now work towards exhibiting homogeneous $\Ric_2>0$ metrics on $\RP^3\times \RP^3$ and $\RP^2\times \RP^3$.  For $\RP^3\times \RP^3$, this follows easily from the approach in the previous paragraph:  $(\pm 1,\pm 1)\in Z(\sph^3\times \sph^3)$, so left multiplication by $\sph^3\times \sph^3$ descends to a well-defined transitive action on $\RP^3\times \RP^3$.
		
	For $\RP^2\times \RP^3$, it was shown in \cite[Proof of Theorem B, pp.~1986-1987]{DGM} that the projectivized tangent bundle $\mathbb P_{\mathbb{R}} T \RP^3$ admits a homogeneous metric with $\Ric_2 > 0$. More precisely, $\mathbb P_{\mathbb{R}} T \RP^3$ can be written as the homogeneous $\so{4}/\mathsf{S}(\oo{1}^2\times\oo{2})$, whose associated Lie algebras fit into the triple $\g{so}_2\subset\g{so}_3\subset\g{so}_4$ appearing in Theorem~\ref{THM:submersion_triples}. The space $\mathbb P_{\mathbb{R}} T \RP^3$ is known to be diffeomorphic to $\RP^2 \times \RP^3$, see, e.g.~\cite[Corollary~3]{Wi:almost}.
		
	Finally, we argue that no homogeneous metric on $\RP^2\times \RP^2$ has $\Ric_2 > 0$.  For, if such a metric existed, then there would be a homogeneous $\Ric_2>0$ metric on its universal cover $\sph^2\times \sph^2$.  We claim that all homogeneous metrics on $\sph^2\times \sph^2$ are product metrics, and hence, cannot have $\Ric_2 > 0$.  To see this, we first note that from \cite[Therorem 6, p.~274]{On:transitive}, if we have $G/H = \sph^2\times \sph^2$ with $H$ and $G$ sharing no positive-dimensional normal subgroups in common, then $G = \sph^3 \times \sph^3$ and $H = \mathsf{T}^2$.  As all maximal tori are conjugate, we may assume that $\mathsf{T}^2$ is the standard $2$-torus.   Then the isotropy action splits into two inequivalent representations.  It follows that any $(\sph^3\times\sph^3)$-invariant metric on $\sph^2\times \sph^2$ is a product metric, as claimed.
	\end{proof}

\end{document}